\documentclass[a4paper]{article}

\usepackage[english]{babel}
\usepackage[utf8]{inputenc}
\usepackage[T1]{fontenc}
\usepackage{amssymb}
\usepackage{amsthm}
\usepackage{amsfonts}
\usepackage{amsmath}
\usepackage{amssymb}
\usepackage{makeidx}
\usepackage{color}
\usepackage{fullpage}
\usepackage{tikz}
\usetikzlibrary{calc, angles}
\usetikzlibrary{shapes.multipart}
\usepackage{mathtools}
\usepackage[all]{xy}
\usepackage{comment}
\usepackage[backend=biber,style=alphabetic]{biblatex}
\usepackage{multirow}
\usepackage{caption}
\usepackage{stmaryrd}
\usepackage{pgfplots}
\pgfplotsset{compat=1.13}

\usepackage[a4paper,top=3cm,bottom=2cm,left=3cm,right=3cm,marginparwidth=1.75cm]{geometry}

\usepackage{amsmath}
\usepackage{graphicx}
\usepackage[colorinlistoftodos]{todonotes}
\usepackage[colorlinks=true, allcolors=blue]{hyperref}
\usepackage[all]{xy}

\providecommand{\keywords}[1]{\textbf{\textit{Keywords---}} #1}

\providecommand{\MSC}[1]{\textbf{\textit{MSC---}} #1}

\newcommand{\Addresses}{% additional braces for segregating \footnotesize
\bigskip
\footnotesize

D. Gil-Muñoz, \textsc{Charles University, Faculty of Mathematics and Physics, Department of Algebra, Sokolovska 83, 18600 Praha 8, Czech Republic}\par\nopagebreak
\textit{E-mail address}: \texttt{daniel.gil-munoz@mff.cuni.cz}

}

\newtheorem{defi}{Definition}[section]
\newtheorem{example}[defi]{Example}
\newtheorem{teo}[defi]{Theorem}
\newtheorem{coro}[defi]{Corollary}
\newtheorem{lema}[defi]{Lemma}
\newtheorem{nota}[defi]{Note}
\newtheorem{pro}[defi]{Proposition}
\newtheorem{rmk}[defi]{Remark}

\title{The ring of integers of Hopf-Galois  degree $p$ extensions \\ of $p$-adic fields with dihedral normal closure}
\author{Daniel Gil-Muñoz}

\date{}

\begin{document}
\maketitle

\vspace{-0.5cm}

\begin{abstract}
For an odd prime number $p$, we consider degree $p$ extensions $L/K$ of $p$-adic fields with normal closure $\widetilde{L}$ such that the Galois group of $\widetilde{L}/K$ is the dihedral group of order $2p$. We shall prove a complete characterization of the freeness of the ring of integers $\mathcal{O}_L$ over its associated order $\mathfrak{A}_{L/K}$ in the unique Hopf-Galois structure on $L/K$, which is analogous to the one already known for cyclic degree $p$ extensions of $p$-adic fields. We shall derive positive and negative results on criteria for the freeness of $\mathcal{O}_L$ as $\mathfrak{A}_{L/K}$-module.
\end{abstract}

\MSC{11S15, 12F10, 16T05.}


\section{Introduction}

The normal basis theorem states that every Galois extension $L/K$ admits a normal basis, that is, a basis consisting of the Galois conjugates of some element of $L$. In general, normal bases are convenient to work with for many reasons. For instance, the vector of coordinates with respect to a normal basis of a Galois conjugate $\sigma(\alpha)$ of an element $\alpha\in L$ is a permutation of the coordinates of $\alpha$. 

A more subtle problem is to determine the existence of a normal basis for a Galois extension $L/K$ of $p$-adic fields which in addition is a basis of the ring of integers $\mathcal{O}_L$ of $L$ as module over the ground ring $\mathcal{O}_K$. This amounts to asking whether $\mathcal{O}_L$ is free as $\mathcal{O}_K[G]$-module, where $G=\mathrm{Gal}(L/K)$. Noether \cite{noether} proved that this happens if and only if $L/K$ is tamely ramified. In order to study wildly ramified extensions, $\mathcal{O}_K[G]$ is replaced by the associated order $$\mathfrak{A}_{L/K}=\{\lambda\in K[G]\,|\,\lambda\cdot x\in\mathcal{O}_L\hbox{ for every }x\in\mathcal{O}_L\}$$ of $\mathcal{O}_L$ in $K[G]$, which is by definition the maximal $\mathcal{O}_K$-order in $K[G]$ acting $\mathcal{O}_K$-linearly on $\mathcal{O}_L$. This approach is due to Leopoldt, see \cite{leopoldt}. The associated order is the right object to choose, as it is the unique $\mathcal{O}_K$-order in $K[G]$ over which $\mathcal{O}_L$ may be free.

However, little is known about the $\mathfrak{A}_{L/K}$-freeness of $\mathcal{O}_L$ in general, and research has shown that there is not a uniform behaviour. Hence, the determination of criteria for the freeness of $\mathcal{O}_L$ has been a recurring problem in the recent decades. Let $p$ be an odd prime number. A complete answer is known for cyclic degree $p$ extensions, consisting in a characterization in terms of the ramification of the extension (see \cite{bertrandiasferton,bertrandiasbertrandiasferton}). Namely:

\begin{teo}[F. Bertrandias, J.P. Bertrandias, M.J. Ferton]\label{cyclicpfreeness} Let $L/K$ be a totally ramified cyclic degree $p$ extension of $p$-adic fields and let $e$ be the ramification index of $K/\mathbb{Q}_p$. Let $G\coloneqq\mathrm{Gal}(L/K)=\langle\sigma\rangle$, let $t$ be its ramification jump and let $a$ be the remainder of $t$ mod $p$.
\begin{itemize}
    \item[(1)] If $a=0$, then $\mathfrak{A}_{L/K}$ is the maximal $\mathcal{O}_K$-order in $K[G]$ and $\mathcal{O}_L$ is $\mathfrak{A}_{L/K}$-free.
    \item[(2)] If $a\neq0$, $\mathfrak{A}_{L/K}$ is $\mathcal{O}_K$-free with generators $\pi_K^{-n_i}f^i$, where $f=\sigma-1$, $n_i=\mathrm{min}_{0\leq j\leq p-1-i}(\nu_{i+j}-\nu_j)$ and $\nu_i=\Big\lfloor\frac{a+it}{p}\Big\rfloor$ for every $0\leq i\leq p-1$.
    \item[(3)] If $a\mid p-1$, then $\mathcal{O}_L$ is $\mathfrak{A}_{L/K}$-free. Moreover, if $t<\frac{pe}{p-1}-1$, the converse holds.
    \item[(4)] If $t\geq\frac{pe}{p-1}-1$, then $\mathcal{O}_L$ is $\mathfrak{A}_{L/K}$-free if and only if the length of the expansion of $\frac{t}{p}$ as continued fraction is at most $4$.
\end{itemize}
\end{teo}

It is well known that the condition on $t$ in (1) is equivalent to $t=\frac{pe}{p-1}$. These extensions are called maximally ramified, as this is the maximal value that $t$ may take for a cyclic degree $p$ extension. Accordingly, extensions satisfying the condition of the fourth statement are called almost maximally ramified.

The first three statements are essentially announced in \cite{ferton1972}, \cite{ferton1974} and \cite{bertrandiasferton}, and in some cases short sketches of proofs are given. Recently, Del Corso, Ferri and Lombardo \cite{delcorsoferrilombardo} provided an alternative proof of these three statements using the notion of minimal index of a Galois extension of $p$-adic fields. The third statement has also appeared as a particular case of \cite[Lemma 4.1]{byottchildselder} and \cite[Corollary 3.6]{elder} (in a slightly weaker form) in the language of scaffolds. The fourth statement is treated in \cite{bertrandiasbertrandiasferton}, but again many of the details are omitted and it does not provide a significant insight on the ideas that lead to the proofs of the stated results. However, a detailed proof appears in Ferton's PhD thesis \cite[Chapitre II]{fertonthesis}. One of the aims of this paper is to spread some of these ideas which have not been available for the general public, in this case working with another family of degree $p$ extensions.

A cyclic degree $p$ extension of $p$-adic fields is a degree $p$ extension whose normal closure is itself. In this paper we work with degree $p$ extensions of $p$-adic fields $L/K$ with normal closure $\widetilde{L}$ such that the Galois group of $\widetilde{L}/K$ is isomorphic to the dihedral group $D_p$ of $2p$ elements. The main difference is that such an extension $L/K$ is not Galois. We will make use of the setting provided by Hopf-Galois theory in order to study the module structure of $\mathcal{O}_L$.

A finite extension $L/K$ is said to be Hopf-Galois if there is a $K$-Hopf algebra $H$ and a $K$-linear action $\cdot\colon H\otimes_KL\longrightarrow L$ such that $L$ is an $H$-module algebra and the canonical map $L\otimes_KH\longrightarrow\mathrm{End}_K(L)$ induced by the action $\cdot$ is a $K$-linear isomorphism. In that case, we say that $(H,\cdot)$ is a Hopf-Galois structure or that $L/K$ is $H$-Galois. Every Galois extension is Hopf-Galois, as the $K$-group algebra $K[G]$ of its Galois group together with the Galois action is a Hopf-Galois structure, called the classical Galois structure. Now, if $L/K$ is an $H$-Galois extension of $p$-adic fields, the associated order of $\mathcal{O}_L$ in $H$ is defined as $$\mathfrak{A}_H=\{h\in H\,|\,h\cdot x\in\mathcal{O}_L\hbox{ for all }x\in\mathcal{O}_L\},$$ and it turns out to be the unique $\mathcal{O}_K$-order in $H$ over which $\mathcal{O}_L$ can be free (a proof of this can be found in \cite[(12.5)]{childs}). Then, the problem turns to the study of the $\mathfrak{A}_H$-freeness of $\mathcal{O}_L$.

From now on and unless stated otherwise, $p$ is an odd prime and $L/K$ denotes a degree $p$ extension of $p$-adic fields with dihedral normal closure $\widetilde{L}$. By Byott's uniqueness theorem \cite[Theorem 2]{byottuniqueness}, this admits a unique Hopf-Galois structure $H$, which is almost classically Galois. In analogy with the cyclic case, the associated order of $\mathcal{O}_L$ in $H$ will be denoted by $\mathfrak{A}_{L/K}$. In this paper, we give a complete answer to the question of whether $\mathcal{O}_L$ is $\mathfrak{A}_{L/K}$-free or not.

Surprisingly, most of the techniques used for the cyclic case in \cite{fertonthesis} translate naturally to this situation, giving rise to similar arithmetical information. The case in which $\widetilde{L}/K$ is not totally ramified is treated in Section \ref{nontotramified} and we prove that the freeness on $L/K$ follows exactly the same criteria in Theorem \ref{cyclicpfreeness}. 

Otherwise, if $\widetilde{L}/K$ is totally ramified, then we can imitate the setting of the cyclic case. A generator of $H$ as a $K$-algebra is $w=z(\sigma-\sigma^{-1})$, where $z\in\widetilde{L}-K$ is such that $z^2\in K$ and $\sigma$ is the order $p$ element in $\mathrm{Gal}(\widetilde{L}/K)$. While in the cyclic case the behaviour depends on the ramification jump of the extension, in this one the relevant number is $\ell\coloneqq\frac{p+t}{2}$, where $t$ is the ramification jump of $\widetilde{L}/K$ (from \cite[Chapter IV, Proposition 9]{serre} it is easily proved that $\ell$ is always an integer number). This is due to the fact that letting $w$ act on $\mathcal{O}_L$ increases valuations by at least $\ell$ (see Proposition \ref{proraisevaluation}). Note that this number $\ell$ is different from the ramification jump of the extension $L/K$ denoted in \cite{elder} in the same way; with our notation, that one is $\frac{t}{2}$ (see \cite[Chapter IV, \textsection 3, Remark 2]{serre} for a definition of ramification jump of a subextension of a Galois extension).

As a result of the resemblance with the cyclic case, we will prove the following result, quite analogous to Theorem \ref{cyclicpfreeness}.

\begin{teo}\label{maintheorem} Let $L/K$ be a degree $p$ extension of $p$-adic fields with normal closure $\widetilde{L}$ such that $G\coloneqq\mathrm{Gal}(\widetilde{L}/K)\cong D_p$. Let $e$ be the ramification index of $K/\mathbb{Q}_p$ and let $H$ be the only Hopf-Galois structure on $L/K$. Assume that $\widetilde{L}/K$ is totally ramified. Let $t$ be the ramification jump of $\widetilde{L}/K$ (i.e, the integer number $t$ such that $G_t\cong C_p$ and $G_{t+1}=\{1\}$) and let $a$ be the remainder of $\ell\coloneqq\frac{p+t}{2}$ mod $p$.
\begin{itemize}
    \item[(1)] If $a=0$, then $\mathfrak{A}_{L/K}$ is the maximal $\mathcal{O}_K$-order in $H$ and $\mathcal{O}_L$ is $\mathfrak{A}_{L/K}$-free.
    \item[(2)] If $a\neq0$, $\mathfrak{A}_{L/K}$ is $\mathcal{O}_K$-free with generators $\pi_K^{-n_i}w^i$, where $n_i=\mathrm{min}_{0\leq j\leq p-1-i}(\nu_{i+j}-\nu_j)$ and $\nu_i=\Big\lfloor\frac{a+i\ell}{p}\Big\rfloor$ for every $0\leq i\leq p-1$.
    \item[(3)] If $a$ divides $p-1$, then $\mathcal{O}_L$ is $\mathfrak{A}_{L/K}$-free. Moreover, if $t<\frac{2pe}{p-1}-2$, then the converse holds.
    \item[(4)] If $t\geq\frac{2pe}{p-1}-2$, then $\mathcal{O}_L$ is $\mathfrak{A}_{L/K}$-free if and only if the length of the expansion of $\frac{\ell}{p}$ as continued fraction is at most $4$.
\end{itemize}
\end{teo}

The first statement is proved in Section \ref{sectinfdegp}, and it is derived from some nice arithmetic properties which are analogous to the ones for a maximally ramified cyclic degree $p$ extension (see Lemma \ref{lemamaxram}). In Section \ref{sectnonmaxram} we will assume that $L/K$ is not maximally ramified and for $\theta=\pi_L^a$ we will introduce the lattice $$\mathfrak{A}_{\theta}=\{\lambda\in H\,|\,\lambda\cdot\theta\in\mathcal{O}_L\}.$$ This is the natural adaptation of the approach followed in \cite{ferton1972}. We will be able to translate the arguments in \cite[Theorem 7.2]{delcorsoferrilombardo} to prove Theorem \ref{maintheorem} (2). On the other hand, our proof of Theorem \ref{maintheorem} (3) will be a consequence of the work by Elder \cite{elder} with what he calls typical degree $p$ extensions: separable totally ramified degree $p$ extensions that are not generated by a root of a prime element. For these, the criteria for the freeness is proved by means of the theory of scaffolds (see Section \ref{sectscaffolds}). Note that the condition on $t$ in Theorem \ref{maintheorem} (3) is equivalent to $\frac{t}{2}<\frac{pe}{p-1}-1$, recovering the bound in Theorem \ref{cyclicpfreeness}, in this case for the ramification jump $\frac{t}{2}$ of $L/K$. Moreover, this is an improvement of \cite[Corollary 3.6 (2)]{elder} for this class of extensions, where the same result is obtained under the condition $\frac{t}{2}<\frac{pe}{p-1}-2$. The last statement will be proved in Section \ref{secthighram} by characterizing when $\mathfrak{A}_{\theta}$ is principal as $\mathfrak{A}_{L/K}$-fractional ideal, where the theory of continued fractions introduced in Section \ref{sectcontfrac} will play an important role. 

Finally, in Section \ref{sectconseq} we will derive some consequences and particular cases of Theorem \ref{maintheorem}. We shall prove that the freeness always holds in the case that $K$ is unramified over $\mathbb{Q}_p$. Moreover, we will make use of Proposition \ref{progivenram}, which proves the existence of dihedral extensions with given ramification, to construct examples of extensions $L/K$ illustrating new behaviours.

If $L/K$ is an arbitrary degree $p$ extension of $p$-adic fields, then the Galois group of its normal closure is a Frobenius group of the form $C_p\rtimes C_d$, where $d$ divides $p-1$. Most of the setting established in Sections \ref{prelimpadic} and \ref{sectinfdegp} remains valid in this case. However, the Hopf-Galois structure is generated by an element different from $w$, so that Proposition \ref{w-poly} is no longer valid. These cases require further investigation.

\section{Preliminaries}\label{prelimpadic}

A Hopf-Galois structure on a degree $n$ extension $L/K$ of fields is a pair $(H,\cdot)$ where $H$ is an $n$-dimensional $K$-Hopf algebra, $\cdot$ is a $K$-linear map $H\otimes_KL\longrightarrow L$ which is compatible with the Hopf algebra structure of $H$ (accurately, it endows $L$ with an $H$-module algebra structure, see \cite[Chapter 2, Section 2.1, Page 39]{underwood}) and the canonical map \begin{equation}\label{mapj}
    \begin{array}{rccl}
    j\colon & L\otimes_KH & \longrightarrow & \mathrm{End}_K(L) \\
     & x\otimes h & \longmapsto & y\mapsto x(h\cdot y)
\end{array}
\end{equation} is an isomorphism of $K$-vector spaces. When such a pair exists, we say that $L/K$ is $H$-Galois. Throughout this paper, we will use the label $H$ indistinctly for a Hopf-Galois structure on $L/K$ and its underlying Hopf algebra.

Let $p$ be an odd prime number. By a $p$-adic field we mean any finite extension of the field $\mathbb{Q}_p$ of $p$-adic numbers, that is, a local field with characteristic zero and residue characteristic $p$. For any $p$-adic field $F$, we call $\mathcal{O}_F$ the ring of integers of $F$, $v_F$ its valuation, $\pi_F$ its uniformizer, and $\mathfrak{p}_F$ the ideal generated by $\pi_F$ in $\mathcal{O}_F$.

Let $E/K$ be a Galois degree $n$ extension of $p$-adic fields and assume that $v_p(n)=1$. Let us denote by $\{G_i\}_{i=-1}^{\infty}$ the chain of ramification groups of $E/K$. If $E/K$ is unramified, the ramification jump of $E/K$ is $t=-1$. Otherwise, we have by \cite[Chapter IV, Corollary 3]{serre} that $G_1\cong C_p$, the cyclic group with $p$ elements. In that case, the ramification jump of $E/K$ is defined as the integer number $t$ such that $G_t\cong C_p$ and $G_{t+1}$ is trivial. Note that we ignore any possible jump prior to $G_1$ in the chain of ramification groups, as it is done for instance in \cite[\textsection 1.2]{berge1978}.

Let $e$ be the ramification index of $K$ over $\mathbb{Q}_p$. By \cite[Proposition 2]{berge1978}, we have $$1\leq t\leq\frac{rpe}{p-1}.$$ Note that even though in the statement of that result it is assumed that $e=1$, this fact is not used in the proof and it also works for the general case. As in \cite{berge1978}, if $t=\frac{rpe}{p-1}$ (resp. $t\geq\frac{rpe}{p-1}-1$) we will say that $E/K$ is maximally ramified (resp. almost maximally ramified).

\subsection{A method to determine the associated order}\label{sectredmethod}

Let $L/K$ be a degree $n$ $H$-Galois extension of $p$-adic fields. In \cite[Section 3]{gilrioinduced}, Rio and the author introduced a method to determine an $\mathcal{O}_K$-basis of the associated order $\mathfrak{A}_H$ in $H$ and to find whether $\mathcal{O}_L$ is $\mathfrak{A}_H$-free or not. A detailed exposition of this method can be found in \cite[Chapter 2]{gilthesis}. In this section, we outline the part of the method devoted to obtain a basis of the associated order, which will be used in Section \ref{sectmaxram}.

Let us fix a $K$-basis $W=\{w_i\}_{i=1}^{n}$ of $H$ and an $\mathcal{O}_K$-basis $B=\{\gamma_j\}_{j=1}^{n}$ of $\mathcal{O}_L$. For each $1\leq i,j\leq n$, we write $$w_i\cdot\gamma_j=\sum_{k=1}^nm_{ij}^{(k)}(H,L)\gamma_k,\quad m_{ij}^{(k)}(H,L)\in K.$$ Let us denote $M_j(H,L)=(m_{ij}^{(k)}(H,L))_{k,i=1}^{n}$. The matrix of the action of $H$ on $L$ is defined as $$M(H,L)=\begin{pmatrix}
\\[-1ex]
M_1(H,L) \\[1ex] 
\hline\\[-1ex]
\cdots \\[1ex]
\hline \\[-1ex]
M_n(H,L)\\
\\[-1ex]
\end{pmatrix}.$$ For $0\leq j\leq n-1$, the submatrix $M_j(H,L)$ will be referred to as the $j$-th block of $M(H,L)$. Let $\rho_H\colon H\longrightarrow\mathrm{End}_K(L)$ be the restriction to $H$ of the map $j$ introduced in \eqref{mapj}, where $H$ is canonically embedded in $L\otimes_KH$. 

The matrix $M(H,L)$ owes its name to the fact that it is the matrix of $\rho_H$ as $K$-linear map when we fix in $H$ the $K$-basis $W$ and in $\mathrm{End}_K(L)$ the $K$-basis $\{\varphi_i\}_{i=1}^{n^2}$ defined as follows: for each $1\leq i\leq n^2$, there are unique $1\leq k,j\leq n$ such that $i=k+(j-1)n$. Then $\varphi_i$ is the $K$-endomorphism of $L$ that sends $\gamma_j$ to $\gamma_k$ and the other $\gamma_l$ to $0$, that is, $\varphi_i(\gamma_l)=\delta_{jl}\gamma_k$.

We have that there is some unimodular matrix $U\in\mathcal{M}_{n^2}(\mathcal{O}_K)$ such that $$UM(H,L)=\begin{pmatrix}D \\ \hline \\[-2ex] O\end{pmatrix},$$ for some invertible matrix $D\in\mathcal{M}_n(\mathcal{O}_K)$ (see \cite[Chapter 2, Theorem 2.15]{gilthesis}). In the terminology of \cite{gilthesis}, $D$ is a reduced matrix. An example of reduced matrix is the content of $D$ times the Hermite normal form of the primitive form of $D$ (see \cite[Definition 2.13]{gilthesis} for a definition of content and primitive form for matrices). Now, the associated order $\mathfrak{A}_H$ is determined from the matrix $M(H,L)$ as follows.

\begin{teo}\label{teoassocorder} Let $D$ be a reduced matrix of $M(H,L)$ and call $D^{-1}=(d_{ij})_{i,j=1}^{n}$. Then, an $\mathcal{O}_K$-basis of the associated order is given by $$v_i=\sum_{l=1}^{n}d_{li}w_l,\quad1\leq i\leq n.$$ Moreover, for every $1\leq i,j\leq n$, the coordinates of $v_i\cdot\gamma_j$ with respect to the basis $B$ are given by the product of $M(H,L)$ by the column vector $(d_{1i},\dots,d_{ni})^t$.
\end{teo}

In \cite[Theorem 3.5]{gilrioinduced} this is proved for a particular matrix $D$, but the same proof works for any reduced matrix.

\subsection{$H$-normal bases and algebraic integers}\label{secthnormalbases}

Let $L/K$ be a degree $n$ $H$-Galois extension. The $K$-linear action $\cdot$ endows $L$ with an $H$-module structure. From \cite[(2.16)]{childs}, we know that $L$ is free of rank one as $H$-module. Note that if $L/K$ is Galois and $H$ is chosen to be its classical Galois structure, we recover the normal basis theorem. Accordingly, the elements of any $K$-basis of $H$ acting on a generator $\theta$ of $L$ as $H$-module form a $K$-basis of $L$. However, unlike the Galois case, in general there is no canonical choice for a $K$-basis of $H$. In any case, we will say that any generator $\theta$ of $L$ as $H$-module generates an $H$-normal basis or that it is an $H$-normal basis generator.

Now, assume that the extension $L/K$ is of $p$-adic fields. The associated order of $\mathcal{O}_L$ in $H$ is defined as $$\mathfrak{A}_H=\{h\in H\,|\,h\cdot x\in\mathcal{O}_L\hbox{ for all }x\in\mathcal{O}_L\},$$ that is, the maximal $\mathcal{O}_K$-order in $H$ among the ones acting on $\mathcal{O}_L$ (see \cite[(12.4)]{childs}). It inherits the $\mathcal{O}_K$-algebra structure of $H$ and has rank $n$ as $\mathcal{O}_K$-module, so it is an $\mathcal{O}_K$-order in $H$. Moreover, it endows $\mathcal{O}_L$ with a module structure. Recall that if $\mathfrak{H}$ is an $\mathcal{O}_K$-order in $H$ such that $\mathcal{O}_L$ is a free $\mathfrak{H}$-module, then $\mathfrak{H}=\mathfrak{A}_H$. However, $\mathcal{O}_L$ is not $\mathfrak{A}_H$-free in general.

In the situation corresponding to Theorem \ref{maintheorem} (4), the $\mathfrak{A}_{L/K}$-freeness of $\mathcal{O}_L$ will be studied through the following object.

\begin{defi} For an element $\theta\in\mathcal{O}_L$ generating an $H$-normal basis for $L/K$, we denote $$\mathfrak{A}_{\theta}=\{h\in H\,|\,h\cdot\theta\in\mathcal{O}_L\}.$$
\end{defi}

It is trivial that $\mathfrak{A}_H\subseteq\mathfrak{A}_{\theta}$ and $\mathfrak{A}_{\theta}$ is an $\mathfrak{A}_H$-module.

If we choose $L/K$ to be Galois and $H$ as its classical Galois structure, we recover the definition in \cite[\textsection 1.2]{ferton1972}. In that case, $\mathfrak{A}_{\theta}$ contains the $\mathcal{O}_K$-group algebra of the Galois group of $L/K$. In our situation, we can find an analogous $\mathcal{O}_K$-order in $\mathfrak{A}_{\theta}$. By the Greither-Pareigis theorem, we can write $H=\widetilde{L}[N]^G$, where $\widetilde{L}$ is the normal closure of $L/K$ and for $G'=\mathrm{Gal}(\widetilde{L}/L)$, $N$ is a regular and $G$-stable subgroup of $\mathrm{Perm}(G/G')$ (see \cite{greitherpareigis} or \cite[Chapter 2]{childs} for more details). Then $\mathcal{O}_{\widetilde{L}}[N]^G$ is an $\mathcal{O}_K$-order in $H$ and Truman \cite[Proposition 2.5]{truman2011} proves that $\mathcal{O}_{\widetilde{L}}[N]^G\subseteq\mathfrak{A}_H$.

\begin{pro}\label{proatheta} Let $L/K$ be a degree $n$ $H$-Galois extension of $p$-adic fields. Let $\theta\in\mathcal{O}_L$ be an element generating an $H$-normal basis for $L/K$. We have that:
\begin{itemize}
    \item[(1)] $\mathfrak{A}_{\theta}$ is a fractional $\mathfrak{A}_H$-ideal.
    \item[(2)] $\mathcal{O}_L=\mathfrak{A}_{\theta}\cdot\theta$, and the map $\varphi\colon\mathfrak{A}_{\theta}\longrightarrow\mathcal{O}_L$ defined by $\varphi(\lambda)=\lambda\cdot\theta$ is an isomorphism of $\mathfrak{A}_H$-modules.
\end{itemize}
\end{pro}
\begin{proof}
\begin{itemize}
    \item[(1)] Let $W=\{w_i\}_{i=1}^n$ be an $\mathcal{O}_K$-basis of the $\mathcal{O}_K$-order $\mathcal{O}_{\widetilde{L}}[N]^G$ in $H$ (so $W$ is also a $K$-basis of $H$) and let $d_{\theta}$ be the discriminant of $(w_i\cdot\theta)_{i=1}^n$ (note that $d_{\theta}\neq0$ because $\theta$ is an $H$-normal basis generator). For $\lambda=\sum_{i=1}^n\lambda_iw_i\in\mathfrak{A}_{\theta}$, we have that $\lambda\cdot\theta=\sum_{i=1}^n\lambda_iw_i\cdot\theta\in\mathcal{O}_L$. Since $d_{\theta}\mathcal{O}_L\subseteq\mathcal{O}_K[\{w_i\cdot\theta\}_{i=1}^n]$, we deduce that $d_{\theta}\lambda_i\in\mathcal{O}_K$ for every $1\leq i\leq n$. Hence $d_{\theta}\lambda\in\mathcal{O}_{\widetilde{L}}[N]^G\subseteq\mathfrak{A}_H$. Since $\lambda$ is arbitrary, $d_{\theta}\mathfrak{A}_{\theta}\subseteq\mathfrak{A}_H$, proving that $\mathfrak{A}_{\theta}$ is a fractional $\mathfrak{A}_H$-ideal.
    
    \item[(2)] The equality $\mathcal{O}_L=\mathfrak{A}_{\theta}\cdot\theta$ follows directly from the definition of $\mathfrak{A}_{\theta}$. On the other hand, since $\{w_i\cdot\theta\}_{i=1}^n$ is an $H$-normal basis, for all $\alpha\in\mathcal{O}_L$ there are unique elements $\lambda_i(\alpha)\in L$, $1\leq i\leq n$, such that $\alpha=\sum_{i=1}^n\lambda_i(\alpha)w_i\cdot\theta$. This defines a map $$\begin{array}{ccl}
        \mathcal{O}_L & \longrightarrow & \mathfrak{A}_{\theta} \\
        \alpha & \longmapsto & \sum_{i=1}^n\lambda_i(\alpha)w_i
    \end{array}$$ which is the inverse of $\varphi$, and hence $\varphi$ is an isomorphism of $\mathfrak{A}_H$-modules.
\end{itemize}
\end{proof}

Now, we relate the study of the $\mathfrak{A}_H$-freeness of $\mathcal{O}_L$ with $\mathfrak{A}_{\theta}$ by means of the following generalization of \cite[Proposition 1]{ferton1972} to this case.

\begin{pro}\label{prohnormalfreeness} Let $L/K$ be a degree $n$ $H$-Galois extension of $p$-adic fields. The following statements are equivalent:
\begin{itemize}
    \item[(1)] $\mathcal{O}_L$ is $\mathfrak{A}_H$-free.
    \item[(2)] There is an element $\theta\in\mathcal{O}_L$ generating an $H$-normal basis of $L$ such that $\mathfrak{A}_H=\mathfrak{A}_{\theta}$.
    \item[(3)] There is an element $\theta\in\mathcal{O}_L$ generating an $H$-normal basis of $L$ such that $\mathfrak{A}_{\theta}$ is a ring (namely, a subring of the underlying ring structure of $H$).
    \item[(4)] For every element $\theta\in\mathcal{O}_L$ generating an $H$-normal basis for $L$, $\mathfrak{A}_{\theta}$ is principal as $\mathfrak{A}_H$-fractional ideal.
\end{itemize}
\end{pro}
\begin{proof}
The equivalence between (1) and (2), and the implication from (2) to (3), are trivial. Let us assume that $\mathfrak{A}_{\theta}$ is a subring of $H$ for some $\theta\in\mathcal{O}_L$ generating an $H$-normal basis for $L/K$. In particular the inherited multiplication from $H$ of two elements in $\mathfrak{A}_{\theta}$ belongs to $\mathfrak{A}_{\theta}$. Using that $\mathcal{O}_L=\mathfrak{A}_{\theta}\cdot\theta$ (see Proposition \ref{proatheta} (2)), one checks easily that $\mathcal{O}_L$ can be endowed with an $\mathfrak{A}_{\theta}$-module structure. Moreover, the map $\varphi\colon\,\mathfrak{A}_{\theta}\longrightarrow\mathcal{O}_L$ in Proposition \ref{proatheta} (2) is now an isomorphism of $\mathfrak{A}_{\theta}$-modules, so $\mathcal{O}_L$ is $\mathfrak{A}_{\theta}$-free. Hence $\mathfrak{A}_{\theta}=\mathfrak{A}_H$.

Finally, we prove the equivalence between (1) and (4). If $\mathcal{O}_L$ is $\mathfrak{A}_H$-free, there is some $\alpha\in\mathcal{O}_L$ such that $\mathcal{O}_L=\mathfrak{A}_H\cdot\alpha$. Let $\theta$ be an $H$-normal basis generator for $L/K$. Since $\mathfrak{A}_H\cdot\alpha=\mathfrak{A}_{\theta}\cdot\theta$, there is some $\varphi\in\mathfrak{A}_{\theta}$ such that $\alpha=\varphi\cdot\theta$. Then $\mathfrak{A}_{\theta}=\mathfrak{A}_H\varphi$. Conversely, if $\mathfrak{A}_{\theta}$ is $\mathfrak{A}_H$-principal, there is some $\varphi\in\mathfrak{A}_{\theta}$ such that $\mathfrak{A}_{\theta}=\mathfrak{A}_H\varphi$. Then $\mathcal{O}_L=(\mathfrak{A}_H\varphi)\cdot\theta=\mathfrak{A}_H\cdot(\varphi\cdot\theta)$.
\end{proof}

In practice, if we find an $H$-normal basis generator $\theta$ such that $\mathfrak{A}_H=\mathfrak{A}_{\theta}$, we know that $\mathcal{O}_L$ is $\mathfrak{A}_H$-free and that $\theta$ is a generator. Alternatively, we may try to determine whether $\mathfrak{A}_{\theta}$ is a principal $\mathfrak{A}_H$-ideal for some $\theta$. If it is not, then one obtains from the above that $\mathcal{O}_L$ is not $\mathfrak{A}_H$-free. However, if $\mathfrak{A}_{\theta}$ is a principal $\mathfrak{A}_H$-ideal for some $\theta$, the same proof of (4) implies (1) above gives that $\mathcal{O}_L$ is $\mathfrak{A}_H$-free. We deduce that if $\mathfrak{A}_{\theta}$ is $\mathfrak{A}_H$-principal for some $\theta$ generating an $H$-normal basis for $L$, so is for every such a $\theta$.

In order to study $\mathfrak{A}_{\theta}$ as $\mathfrak{A}_H$-ideal, for each $\alpha\in\mathfrak{A}_{\theta}$ we may consider the map $$\begin{array}{rccl}
    \psi_{\alpha}\colon & \mathfrak{A}_H & \longrightarrow & \mathfrak{A}_{\theta}, \\
     & \lambda & \longmapsto & \lambda\alpha.
\end{array}$$ If we fix $\mathcal{O}_K$-bases of $\mathfrak{A}_H$ and $\mathfrak{A}_{\theta}$ and call $M(\alpha)$ the matrix of $\psi_{\alpha}$ with respect to these bases, then $\alpha$ is an $\mathfrak{A}_H$-generator of $\mathcal{O}_L$ if and only if $\mathrm{det}(M(\alpha))\not\equiv0\,(\mathrm{mod}\,\mathfrak{p}_K)$.

\subsection{Scaffolds on degree $p$ extensions}\label{sectscaffolds}

The theory of scaffolds was originally introduced by Byott, Childs and Elder \cite{byottchildselder} as a unification of several works concerning the Galois module structure of the ring of integers. A scaffold on a totally ramified $H$-Galois extension $L/K$ of local fields is roughly speaking a finite set of elements $\Psi_i$ on $H$ acting on a family of elements $\lambda_k$ of $L$ with prescribed valuation in such a way that the valuations of the elements $\Psi_i\cdot\lambda_j$ can be determined up to a certain precision. The utility of this theory is the validity of some criteria on the module structure in terms of the precision of the scaffold. We present here the main notions and results restricted to the case of totally ramified degree $p$ extensions of $p$-adic fields.

Let $L/K$ be such an extension. Let $b\in\mathbb{Z}$ coprime with $p$. Call $\mathbb{S}_p=\{0,1,\dots,p-1\}$ and let $\mathfrak{b}\colon\mathbb{S}_{p}\longrightarrow\mathbb{Z}$ be defined by $\mathfrak{b}(s)=bs$. Since $b$ is coprime with $p$, the map $r\circ(-\mathfrak{b})$ is bijective, where $r\colon\mathbb{Z}\longrightarrow\mathbb{S}_p$ sends each integer number to its remainder mod $p$. Let $\mathfrak{a}\colon\mathbb{S}_p\longrightarrow\mathbb{S}_p$ be its inverse. The definition of scaffold in this case is the following (see \cite[Definition 2.3]{byottchildselder}):

\begin{defi}\label{defiscaffold} Let $A$ be a $K$-algebra of dimension $p$ acting $K$-linearly on $L$ and let $\mathfrak{c}\in\mathbb{Z}_{\geq1}$. An $A$-scaffold on $L$ with precision $\mathfrak{c}$ and shift parameter $b$ consists of the following data:
\begin{itemize}
    \item[1.] A family $\{\lambda_k\}_{k\in\mathbb{Z}}$ of elements in $L$ such that $v_L(\lambda_k)=k$ for all $k\in\mathbb{Z}$ and $\lambda_{k_1}\lambda_{k_2}^{-1}\in K$ if and only if $k_1\equiv k_2\,(\mathrm{mod}\,p)$.
    \item[2.] An element $\Psi\in A$ with $\Psi\cdot1=0$ and the property that for every $k\in\mathbb{Z}$ there is some $u_k\in\mathcal{O}_K^*$ such that:
    $$\begin{cases}
    \Psi\cdot\lambda_k\equiv u_k\lambda_{k+b}\,(\mathrm{mod}\,\lambda_{k+b}\mathfrak{p}_L^{\mathfrak{c}}) &\hbox{ if }\mathfrak{a}(k)\geq1 \\
    \Psi\cdot\lambda_k\equiv0\,(\mathrm{mod}\,\lambda_{k+b}\mathfrak{p}_L^{\mathfrak{c}}) &\hbox{ if }\mathfrak{a}(k)=0 \\
    \end{cases}$$
\end{itemize}
\end{defi}

In \cite{byottchildselder}, the module structure of the ideals $\mathfrak{p}_L^h$, where $h\geq0$, is studied. Since we are interested in the ring of integers $\mathcal{O}_L$, we just take $h=0$. Consequently, the parameters \cite[(3.3),(3.4)]{byottchildselder} become $$d(s)=\Big\lfloor\frac{\mathfrak{b}(s)+b}{p}\Big\rfloor,\quad w(s)=\mathrm{min}\{d(s+j)-d(j)\,|\,0\leq j\leq p-1-s\},$$ where $s\in\mathbb{S}_p$. Moreover, the result \cite[Theorem 3.1]{byottchildselder} relating the precision of the scaffold with the module structure of $\mathcal{O}_L$ restricted to this situation becomes:

\begin{teo}\label{criteriascaffold} Assume that $L/K$ admits an $A$-scaffold of precision $\mathfrak{c}$ and let $b\in\mathbb{S}_p$ be such that $\mathfrak{a}(b)=p-1$.
\begin{itemize}
    \item[(1)] If $\mathfrak{c}\geq\mathrm{max}(b,1)$ and $w(s)=d(s)$ for all $s\in\mathbb{S}_p$, then $\mathcal{O}_L$ is $\mathfrak{A}_H$-free.
    \item[(2)] If $\mathfrak{c}\geq p+b$, then $\mathcal{O}_L$ is $\mathfrak{A}_H$-free if and only if $w(s)=d(s)$ for all $s\in\mathbb{S}_p$.
\end{itemize}
\end{teo}

\subsection{Continued fractions and distance to the nearest integer}\label{sectcontfrac}

The criteria in Theorem \ref{cyclicpfreeness} (4) and Theorem \ref{maintheorem} (4) depend on the continued fraction expansion of the number $\frac{a}{p}$ for $0\leq a\leq p-1$. In this part we introduce the notation and some results which will be needed later on.

For a non-negative integer $x$ and an odd prime number $p$, we write $$\frac{x}{p}=[a_0;a_1,\dots,a_n]$$ for the continued fraction expansion of the rational number $\frac{x}{p}$. By definition, $a_0$ is just its integer part. If $a$ is the remainder of $x$ mod $p$, then $\frac{a}{p}=[0;a_1,\dots,a_n]$. Note that this is the situation in Theorem \ref{cyclicpfreeness} for $x=t$ and in Theorem \ref{maintheorem} for $x=\ell$. We will usually omit the explicit mention to the number $x$ and consider the continued fraction expansion of $\frac{a}{p}$.

For $i\geq0$, the $i$-th convergent of $\frac{x}{p}$ is the number defined by $$\frac{p_i}{q_i}=[a_0;a_1,\dots,a_i],$$ where the fraction is in irreducible form. It is easy to check that $\frac{p_i}{q_i}<\frac{x}{p}$ if $i$ is even and $\frac{p_i}{q_i}>\frac{x}{p}$ otherwise. We are especially interested in the denominators $q_i$. By definition, $q_0=1$, $q_1=a_1$ and $q_n=p$. Moreover, it is satisfied that $q_{i+2}=a_{i+2}q_{i+1}+q_i$ for every $i\geq2$.

The numerators $p_i$ satisfy an analogous recurrence relation. Now, let us denote $$q_{i,r_i}=r_iq_{i+1}+q_{i},\quad0\leq r_i\leq a_{i+2},$$ and define integer numbers $p_{i,r_i}$ similarly. The fractions $\frac{p_{i,r_i}}{q_{i,r_i}}$ are called semiconvergents or intermediate convergents of $\frac{\ell}{p}$ (see for instance \cite[I, \textsection 4]{lang1995}). It is immediate from the definition that $q_{i,0}=q_i$ and $q_{i,a_{i+2}}=q_{i+2}$.

For a real number $\alpha$, we denote by $\lfloor \alpha\rfloor$ the integer part and by $\widehat{\alpha}$ the fractional part, so that $\alpha=\lfloor\alpha\rfloor+\widehat{\alpha}$. On the other hand, if $2\alpha\notin\mathbb{Z}$, we write $||\alpha||$ for the distance of $\alpha$ to the nearest integer. Note that this restriction is not problematic for us because in our case $\alpha$ is a rational number with denominator $1$ or $p$ (in irreducible form). Note also that if $a$ is the remainder of $x$ mod $p$, then $||q\frac{x}{p}||=||q\frac{a}{p}||$ for every $q\in\mathbb{Z}$.

\begin{pro}\label{propcontfrac} Let $\alpha=\frac{a}{p}$ with $0<a<p$ and let $\frac{p_i}{q_i}$ be its $i$-th convergent.
\begin{itemize}
    \item[(1)] $||q_{i+1}\frac{a}{p}||<||q_i\frac{a}{p}||$ for every $0\leq i\leq n-1$.
    \item[(2)] If $q\in\mathbb{Z}$ and $q<q_i$, then $||q\frac{a}{p}||\geq||q_{i-1}\frac{a}{p}||$, and in particular, $||q\frac{a}{p}||>||q_i\frac{a}{p}||$.
    \item[(3)] If $i>0$ is even (resp. odd), then $\widehat{q_i\frac{a}{p}}<\frac{1}{2}$ (resp. $\widehat{q_i\frac{a}{p}}>\frac{1}{2}$).
    \item[(4)] $||q_{n-1}\frac{a}{p}||=\frac{1}{p}$, so $\widehat{q_{n-1}\frac{a}{p}}=\frac{1}{p}$ if $n$ is odd and $\widehat{q_{n-1}\frac{a}{p}}=\frac{p-1}{p}$ otherwise.
\end{itemize}
\end{pro}
\begin{proof}
The first statement is proved in \cite[I, \textsection 2, Corollary of Theorem 5]{lang1995}. The second one is equivalent to the statement that $q_i$ is the smallest integer $q>q_{i-1}$ such that  $||q\frac{a}{p}||<||q_{i-1}\frac{a}{p}||$, which is just \cite[Theorem 6]{lang1995}. For the third part, we note that from \cite[I, \textsection 1, Corollary 3 of Theorem 3]{lang1995} we have that $$q_i\frac{x}{p}-p_i=\frac{(-1)^ia_{i+2}}{a_{i+2}q_{i+1}+q_i}.$$ Again by \cite[Theorem 6]{lang1995}, $\frac{p_i}{q_i}$ is a best approximation to $\frac{x}{p}$ (see \cite[I, \textsection 2, paragraph before Theorem 6]{lang1995}), and in particular $|q_i\frac{x}{p}-p_i|=||q_i\frac{a}{p}||$. Hence, $$\Big|\Big|q_i\frac{a}{p}\Big|\Big|=\Big|\Big|\frac{(-1)^ia_{i+2}}{a_{i+2}q_{i+1}+q_i}\Big|\Big|.$$ Since $a_{i+2}q_{i+1}+q_i>2a_{i+2}$, (3) follows. To prove (4), we just note that \cite[I, \textsection 2, Corollary 1 of Theorem 2]{lang1995} gives that $|q_{n-1}\frac{a}{p}-p_{n-1}|=\frac{1}{p}$.
\end{proof}

\section{Degree $p$ extensions with dihedral normal closure}\label{sectinfdegp}

Let $L/K$ be a degree $p$ extension of $p$-adic fields whose normal closure $\widetilde{L}$ satisfies $G\coloneqq\mathrm{Gal}(L/K)\cong D_p$.

\subsection{The unique Hopf-Galois structure}\label{secthgstr}

Let us establish the presentation \begin{equation}\label{presentG}
    G=\langle\sigma,\tau\,|\,\sigma^p=\tau^2=1,\,\tau\sigma=\sigma^{-1}\tau\rangle.
\end{equation} Since $p$ is a Burnside number, we can apply \cite[Theorem 2]{byottuniqueness} to obtain that $L/K$ admits a unique Hopf-Galois structure $H$. We know that $H=\widetilde{L}[N]^G$, where $N$ is the unique regular subgroup of $\mathrm{Perm}(G/G')$ which is stable under the action of $G$. It can be easily checked that the only permutation subgroup satisfying these conditions is $N=\langle\lambda(\sigma)\rangle$, where $\lambda\colon G\longrightarrow\mathrm{Perm}(G/G')$ is the left translation map. The action of $H$ on $L$ is as follows: if $h=\sum_{i=1}^ph_i\eta_i\in H$ with $h_i\in\widetilde{L}$, $\eta_i\in N$ and $x\in L$, then $$h\cdot x=\sum_{i=1}^ph_i\eta_i^{-1}(\overline{1})(x).$$

On the other hand, since $G$ is the dihedral group of $2p$ elements, it has a unique order $p$ subgroup, which by means of the Galois correspondence yields a unique quadratic subextension $M/K$ of $\widetilde{L}/K$. Hence there is some $z\in \widetilde{L}$ such that $z^2\in K$, or equivalently, $\tau(z)=-z$. Then, the signs of both $\lambda(\sigma)-\lambda(\sigma)^{-1}$ and $z$ are switched by the action of $G/\mathrm{Gal}(L/M)$, and so their product is fixed and hence it belongs to $H$. Since $\lambda(\sigma)$ acts as $\sigma^{-1}$, we can identify $\lambda(\sigma)$ and $\sigma$, which is the same as taking the negative of $\lambda(\sigma)-\lambda(\sigma)^{-1}$. Hence, we consider the element $$w\coloneqq z(\sigma-\sigma^{-1}),$$ which belongs to $H$. Actually, it can be seen that its powers up to $p-1$ are $K$-linearly independent, and $H$ has dimension $p$. Then, $H$ is generated by $w$ as a $K$-algebra, that is, $H=K[w]$. In fact, $w$ generates $\mathcal{O}_{\widetilde{L}}[N]^G$ as an $\mathcal{O}_K$-algebra, and this can be understood as an analogue of \cite[Proposition 6.8. (1)]{delcorsoferrilombardo}.

\begin{rmk}\normalfont\label{rmkdescrH} Following the terminology of \cite{elder}, in the case of a typical extension (a separable totally ramified degree $p$ extension of local fields which is not generated by a $p$-th root of a uniformizer), our description of a generator fits with the one in \cite[Theorem 3.1]{elder}, namely it coincides up to multiplication by a unit of $K$.
\end{rmk}

Since $H=K[w]$ is of dimension $p$, we know that the powers of $w$ up to $p-1$ form a $K$-basis of $H$. In the following, we determine the coordinates of $w^p$ with respect to that basis. This can be seen as an analogue of \cite[Proposition 6.8. (2)]{delcorsoferrilombardo}.

\begin{pro}\label{w-poly} With the notation as above, we have
$$w^p+c_1z^2w^{p-2}+c_2z^4w^{p-4}+\dots+c_{\frac{p-1}{2}}z^{p-1}w=0,$$ where $$c_j=\frac{p}{j}\binom{p-j-1}{j-1}=\binom{p-j}{j}+\binom{p-j-1}{j-1}.$$ In particular, we have that $c_j\in\mathbb{Z}_{>0}$ and $v_p(c_j)=1$ for every $1\leq j\leq\frac{p-1}{2}$.
\end{pro}
\begin{proof}
The two expressions for $c_j$ are equal since \begin{equation*}
    \begin{split}
        \binom{p-j}{j}+\binom{p-j-1}{j-1}&=\frac{(p-j)!}{j!(p-2j)!}+\frac{(p-j-1)!}{(j-1)!(p-2j)!}\\&=\left(\frac{p-j}{j}+1\right)\frac{(p-j-1)!}{(j-1)!(p-2j)!}\\&=\frac{p}{j}\binom{p-j-1}{j-1},
    \end{split}
\end{equation*} and from the first one it follows that $c_j$ is divisible by $p$ but not by $p^2$. By convention, we call $c_0=1$. Then the stated equality is equivalent to \begin{equation}\label{eqredpoly}
    \sum_{j=0}^{\frac{p-1}{2}}c_j(\sigma-\sigma^{-1})^{p-2j}=0.
\end{equation} 

Let us call $g=\sigma-\sigma^{-1}$. It is easily checked that \begin{equation}\label{eqdiff}
    \sigma^{2i+1}-\sigma^{-(2i+1)}=(\sigma^{2i-1}-\sigma^{-(2i-1)})(\sigma^2+\sigma^{-2})-(\sigma^{2i-3}-\sigma^{-(2i-3)}),
\end{equation}  and $\sigma^2+\sigma^{-2}=g^2+2$. Recursively, we see that $\sigma^{2i+1}-\sigma^{-(2i+1)}$ can be written as a polynomial in $g$ where the coefficients of even powers vanish. That is, for each $0\leq j\leq i$ there is $s_j^{(2i+1)}\in\mathbb{Z}_{\geq0}$, such that $$\sigma^{2i+1}-\sigma^{-(2i+1)}=\sum_{j=0}^is_j^{(2i+1)}g^{2i+1-2j}.$$ We shall prove by induction on $1\leq i\leq\frac{p-1}{2}$ that $s_0^{(2i+1)}=1$ and $s_j^{(2i+1)}=\binom{2i+1-j}{j}+\binom{2i-j}{j-1}$ for every $1\leq j\leq i$, which for $i=\frac{p-1}{2}$ implies the result. For $i=1$, we see that $\sigma^3-\sigma^{-3}=g^3+3g$, so $s_0^{(3)}=1$ and $s_1^{(3)}=3$, which satisfy the required formula. Now, suppose that our claim holds for $i$, and let us prove it for $i+1$. Using the expression \eqref{eqdiff}, we find the relations $$\begin{cases}
s_0^{(2i+3)}=s_0^{(2i+1)}, \\
s_1^{(2i+3)}=2s_0^{(2i+1)}+s_1^{(2i+1)}, \\
s_j^{(2i+3)}=2s_{j-1}^{(2i+1)}+s_j^{(2i+1)}-s_{j-2}^{(2i-1)},\quad2\leq j\leq i, \\
s_{i+1}^{(2i+3)}=2s_{i}^{(2i+1)}-s_{i-1}^{(2i-1)}.
\end{cases}$$ By the induction hypothesis, we have that $s_0^{(2i+3)}=1$ and $s_1^{(2i+3)}=s_{i+1}^{(2i+3)}=2i+3$, obtaining the required formula. For the third relation, we use repeteadly the property that $\binom{n}{k}+\binom{n}{k-1}=\binom{n+1}{k}$ for non-negative integers $k\leq n$ to calculate: \begin{equation*}
    \begin{split}
        s_j^{(2i+3)}&=2\binom{2i+2-j}{j-1}+2\binom{2i+1-j}{j-2}+\binom{2i+1-j}{j}+\binom{2i-j}{j-1}-\binom{2i+1-j}{j-2}-\binom{2i-j}{j-3}\\&=2\binom{2i+2-j}{j-1}+\binom{2i+1-j}{j-2}+\binom{2i+1-j}{j}+\binom{2i-j}{j-1}-\binom{2i-j}{j-3}\\&=2\binom{2i+2-j}{j-1}+\binom{2i-j}{j-2}+\binom{2i+1-j}{j}+\binom{2i-j}{j-1}\\&=2\binom{2i+2-j}{j-1}+\binom{2i+1-j}{j}+\binom{2i+1-j}{j-1}\\&=2\binom{2i+2-j}{j-1}+\binom{2i+2-j}{j}=\binom{2i+3-j}{j}+\binom{2i+2-j}{j-1}.
    \end{split}
\end{equation*}
\end{proof}

\begin{rmk}\normalfont Let $\zeta_p$ be a primitive $p$-th root of unity in $\mathbb{C}$. In the previous proof, we have essentially determined the minimal polynomial of $\zeta_p-\zeta_p^{-1}=2i\mathrm{sin}(\frac{2\pi}{p})$ over $\mathbb{Q}$ (where now $i$ denotes the imaginary unit). Indeed, since $\sigma^p=\mathrm{Id}$, identifying $\sigma$ with $\zeta_p$ defines an isomorphism $\mathbb{Q}[\sigma]\cong\mathbb{Q}(\zeta_p)$ as rings.
\end{rmk}

We will refer to $$f(x)=x^p+\sum_{j=1}^{\frac{p-1}{2}}c_jz^{2j}w^{p-2j}$$ as the minimal polynomial of $w$ over $K$, in the sense that it is the monic polynomial $f\in K[x]$ of minimal degree such that $f(w)$ is the identically zero endomorphism of $L$.

\begin{example}\normalfont\label{examplepolymin}The coefficients $c_j$ in the expression given by Proposition \ref{w-poly} are obtained directly as soon as we choose a value for $p$. However, the determination of an element $z\in\widetilde{L}$ such that $z^2\in K$ is not an easy task. If $K=\mathbb{Q}_p$, it is known that the quadratic extensions are $\mathbb{Q}_p(\sqrt{p})$, $\mathbb{Q}_p(\sqrt{np})$ and $\mathbb{Q}_p(\sqrt{n})$, where $n$ is a non-quadratic residue mod $p$. In \cite[Sections 4.4 and 4.5]{gilthesis}, a suitable $z$ is found for the cases $K=\mathbb{Q}_3$ and $K=\mathbb{Q}_5$ by studying directly the factorization of a generating polynomial of $L/K$ in $L[x]$. For $p=3$, the degree $3$ extensions of $\mathbb{Q}_3$ with dihedral degree $6$ normal closure are the ones generated by the polynomials $$x^3+3,\quad x^3+12,\quad x^3+21,$$ $$x^3+3x^2+3,\quad x^3+3x+3,\quad x^3+6x+3$$ (see \cite[Theorem 5.2]{awtreyedwards}). These correspond respectively to \cite[\href{https://www.lmfdb.org/LocalNumberField/3.3.5.1}{$p$-adic field 3.3.5.1}, \href{https://www.lmfdb.org/LocalNumberField/3.3.5.3}{$p$-adic field 3.3.5.3}, \href{https://www.lmfdb.org/LocalNumberField/3.3.5.2}{$p$-adic field 3.3.5.2}, \href{https://www.lmfdb.org/LocalNumberField/3.3.4.4}{$p$-adic field 3.3.4.4}, \href{https://www.lmfdb.org/LocalNumberField/3.3.3.2}{$p$-adic field 3.3.3.2}, \href{https://www.lmfdb.org/LocalNumberField/3.3.3.1}{$p$-adic field 3.3.3.1}]{lmfdb} in the LMFDB database. The minimal polynomial for these extensions are shown in the following table.

\begin{center}
\begin{tabular}{|l|c|c|}
    \hline
     & $z$ & Minimal polynomial of $w$ \\ \hline
    $x^3+3a$, $a\in\{1,4,7\}$ & $\sqrt{-3}$ & $x^3-9x$ \\ \hline
    $x^3+3x^2+3$ & $\sqrt{-1}$ & $x^3-3x$ \\ \hline
    $x^3+3x+3$ & $\sqrt{-3}$ & $x^3-9x$ \\ \hline
    $x^3+6x+3$ & $\sqrt{3}$ & $x^3+9x$ \\ \hline
\end{tabular}
\end{center}

\end{example}

Later on, it will be useful to know the valuations of coordinates of the powers of $w$ from $p$ to $2p-2$. To that end, we prove:

\begin{pro}\label{w-highpoly} Assume that $z$ is a uniformizer of $M$. Let $1\leq m\leq p-1$. We have that:
\begin{itemize}
    \item[(1)] If $m=2m'-1$ is odd, then there are $r_1^{(m)},r_3^{(m)},\dots,r_{m-2}^{(m)}\in\mathcal{O}_K$ and $r_m^{(m)},r_{m+2}^{(m)},\dots,r_{p-2}^{(m)}\in\mathcal{O}_K^*$ such that \begin{equation*}\label{eqhighpoly1}
        \begin{split}
            w^{p-1+m}&=\sum_{j=1}^{\frac{p+1}{2}-m'}\pi_K^{e+m'+j-1}r_{p-2j}^{(m)}w^{p-2j}+\sum_{j=\frac{p+1}{2}-m'+1}^{\frac{p-1}{2}}\pi_K^{2e+m'+j-1}r_{p-2j}^{(m)}w^{p-2j},
        \end{split}
    \end{equation*} where the second sum only appears if $m>1$.
    \item[(2)] If $m=2m'$ is even then there are $r_2^{(m)},r_4^{(m)},\dots,r_{m-2}^{(m)}\in\mathcal{O}_K$ and $r_m^{(m)},r_{m+2}^{(m)},\dots,r_{p-1}^{(m)}\in\mathcal{O}_K^*$ such that \begin{equation*}\label{eqhighpoly2}
        \begin{split}
            w^{p-1+m}&=\sum_{j=1}^{\frac{p+1}{2}-m'}\pi_K^{e+m'+j-1}r_{p-2j+1}^{(m)}w^{p-2j+1}+\sum_{j=\frac{p+1}{2}-m'+1}^{\frac{p-1}{2}}\pi_K^{2e+m'+j-1}r_{p-2j+1}^{(m)}w^{p-2j+1},
        \end{split}
    \end{equation*} where the second sum only appears if $m>2$.
\end{itemize}
\end{pro}
\begin{proof}
Note that if $z$ is a square root of a uniformizer of $K$, then it is a uniformizer of $M$ such that $z^2\in K$ and $\tau(z)=-z$. Thus, we can actually choose $z$ as in the statement.

The second statement follows immediately from multiplying the equality in the first one by $w$, with $r_i^{(m)}=r_{i-1}^{(m-1)}$ for $i\in\{2,4,\dots,p-1\}$. Then, we only need to prove (1). We proceed by induction on $m'$. If $m'=1$ (so that $m=1$), we have by Proposition \ref{w-poly} that $$w^p=-\Big(\sum_{j=1}^{\frac{p-1}{2}}c_jz^{2j}w^{p-2j}\Big).$$ Since $z$ is a uniformizer of $M$, the valuation of the coefficient of $w^{p-2j}$ is $$v_K(c_jz^{2j})=e+j$$ for every $1\leq j\leq\frac{p-1}{2}$. Then there are $r_1^{(1)},r_3^{(1)},\dots,r_{p-2}^{(1)}\in\mathcal{O}_K^*$ such that $$w^p=\pi_K^{e+1}r_{p-2}^{(1)}w^{p-2}+\pi_K^{e+2}r_{p-4}^{(1)}w^{p-4}+\dots+\pi_K^{e+\frac{p-1}{2}}r_1^{(1)}w=\sum_{j=1}^{\frac{p-1}{2}}\pi_K^{e+j}r_{p-2j}^{(1)}w^{p-2j}.$$ Assume that the result holds for $m=2m'-1$ for $1\leq m'<\frac{p-1}{2}$ and let us prove it for $m+2=2m'+1$. From the hypothesis, we have that there are $r_1^{(m)},r_3^{(m)},\dots,r_{m-2}^{(m)}\in\mathcal{O}_K$ and $r_m^{(m)},r_{m+2}^{(m)},\dots,r_{p-2}^{(m)}\in\mathcal{O}_K^*$ such that $$w^{p-1+m}=\sum_{j=1}^{\frac{p+1}{2}-m'}\pi_K^{e+m'+j-1}r_{p-2j}^{(m)}w^{p-2j}+\sum_{j=\frac{p+1}{2}-m'+1}^{\frac{p-1}{2}}\pi_K^{2e+m'+j-1}r_{p-2j}^{(m)}w^{p-2j}.$$ Multiplying by $w^2$ at both sides, we obtain \begin{equation}\label{eqhypind}
    \begin{split}
        w^{p-1+m+2}&=\pi_K^{e+m'}r_{p-2}^{(m)}w^p+\pi_K^{e+m'+1}r_{p-4}^{(m)}w^{p-2}+\dots+\pi_K^{e+\frac{p-1}{2}}r_m^{(m)}w^{m+2}\\&+\pi_K^{2e+\frac{p+1}{2}}r_{m-2}^{(m)}w^m+\pi_K^{2e+\frac{p+3}{2}}r_{m-4}^{(m)}w^{m-2}+\dots+\pi_K^{2e+m'+\frac{p-3}{2}}r_1^{(m)}w^3
    \end{split}
\end{equation} From the case $m'=1$ we know the explicit expression of $w^p$. Replacing it in \eqref{eqhypind}, we are able to determine the coefficient of $w^{p-2j}$ for each $1\leq j\leq\frac{p-1}{2}$:
\begin{itemize}
    \item For $j\in\{1,2,\dots,\frac{p-1}{2}-m'-1\}$, the coefficient of $w^{p-2j}$ is $$\pi_K^{2e+m'+j}r_{p-2}^{(m)}r_{p-2j}^{(1)}+\pi_K^{e+m'+j}r_{p-2j}^{(m)}.$$ Since $r_{p-2j}^{(m)}\in\mathcal{O}_K^*$, the $K$-valuation of this term is exactly $e+m'+j$. Then there is some $r_{p-2j}^{(m+2)}\in\mathcal{O}_K^*$ such that the coefficient above is equal to $\pi_K^{e+(m'+1)+j-1}r_{p-2j}^{(m+2)}$.
    \item For each $j\in\{\frac{p-1}{2}-m',\dots,\frac{p-5}{2},\frac{p-3}{2}\}$ the coefficient of $w^{p-2j}$ is $$\pi_K^{2e+m'+j}r_{p-2}^{(m)}r_{p-2j}^{(1)}+\pi_K^{2e+m'+j}r_{p-2j}^{(m)},$$ whose $K$-valuation is lower bounded by $2e+m'+j$. Therefore, there is some $r_{p-2j}^{(m+2)}\in\mathcal{O}_K$ such that the coefficient above is equal to $\pi_K^{2e+(m'+1)+j-1}r_{p-2j}^{(m+2)}$.
    \item Finally, for $j=\frac{p-1}{2}$, we have $w^{p-2j}=w$ and its coefficient is $r_1^{(m+2)}\pi_K^{2e+(m'+1)+\frac{p-1}{2}-1}$, where $r_1^{(m+2)}=r_{p-2}^{(m)}r_1^{(1)}\in\mathcal{O}_K^*$.
\end{itemize}
\end{proof}

\subsection{The arithmetic of $L/K$}\label{sectarithm}

In this part, we assume the notation of Section \ref{prelimpadic} with $E=\widetilde{L}$. As discussed in Section \ref{secthgstr}, $\widetilde{L}/K$ has a unique quadratic subextension $M/K$. Note that the extension $\widetilde{L}/M$ is cyclic of degree $p$, and hence the $\mathfrak{A}_{\widetilde{L}/M}$-module structure of $\mathcal{O}_{\widetilde{L}}$ is completely determined by Theorem \ref{cyclicpfreeness}. The absolute ramification index of $M$ is $e$ if $M/K$ is unramified and $2e$ otherwise.

It is a standard fact that any unramified extension is Galois. Therefore, a degree $p$ extension $L/K$ of $p$-adic fields with dihedral normal closure $\widetilde{L}$ must be totally ramified. Hence, $\widetilde{L}/K$ is totally ramified if and only if so is $M/K$. For the rest of this section, we assume that this is the case. 

The ramification jump of the extension $\widetilde{L}/K$ satisfies $$1\leq t\leq\frac{2pe}{p-1}.$$ From \cite[Chapter IV, Proposition 2]{serre} we deduce that $t$ is also the unique ramification jump of the cyclic degree $p$ extension $\widetilde{L}/M$. Then, $t=\frac{2pe}{p-1}$ if and only if $t$ is divisible by $p$.

From \cite[Chapter IV, Proposition 7]{serre} we have monomorphisms $$\theta_0\colon G_0/G_1\longrightarrow k^*,\quad\theta_t\colon G_t\longrightarrow k,$$ where $k$ is the residue field of $K$. Now, \cite[Chapter IV, Proposition 9]{serre} gives that $$\theta_t(\tau\sigma\tau^{-1})=\theta_0(\tau)^t\theta_t(\sigma).$$ Since $\tau\sigma\tau^{-1}=\sigma^{-1}$, we conclude that $t$ is odd. Therefore, $$\ell\coloneqq\frac{p+t}{2}$$ is an integer number. 

Actually, it is known that there always exists a dihedral degree $2p$ extension of $p$-adic fields with given ramification parameters.

\begin{pro}\label{progivenram} Let $M/K$ be a quadratic extension of $p$-adic fields. Assume that $M$ contains the primitive $p$-th roots of unity and that $K$ does not. Let $e_0$ be the ramification index of $M/\mathbb{Q}_p$ and let $t$ be an integer satisfying one of the following conditions:
\begin{itemize}
    \item[(a)] $t=\frac{pe_0}{p-1}$.
    \item[(b)] $1\leq t<\frac{pe_0}{p-1}$, $t$ is coprime with $p$ and if $M/K$ is ramified, $t$ is odd.
\end{itemize}
Then there is some field extension $E$ of $M$ such that $E/K$ is a dihedral degree $2p$ extension and its ramification jump is $t$.
\end{pro}

A proof of this result can be found in \cite[Chapitre III, Lemme p.50]{fertonthesis}. Any dihedral extension $E/K$ given by Proposition \ref{progivenram} posseses exactly $p$ degree $p$ subextensions with normal closure $E$. Therefore, we can use this result to give a degree $p$ extension with dihedral normal closure and with given ramification.

We now explore the action of $w$ on $\mathcal{O}_L$. The integer number $\ell$ encodes the behaviour of the Hopf-Galois action on $\mathcal{O}_L$, in the sense that $w$ raises valuations of elements in $\mathcal{O}_L$ by at least $\ell$. Concretely:

\begin{pro}\label{proraisevaluation} Let $L/K$ be a degree $p$ extension with dihedral normal closure $\widetilde{L}$ and let $t$ be its ramification jump. Assume that $\widetilde{L}/K$ is totally ramified. Given $x\in\mathcal{O}_L$, we have that $$v_L(w\cdot x)\geq\ell+v_L(x),$$ with equality if and only if $p$ does not divide $v_L(x)$.
\end{pro}
\begin{proof}
Denote $G\coloneqq\mathrm{Gal}(\widetilde{L}/K)$. Let $\sigma$ be an order $p$ element of $G$ as in \eqref{presentG}. Combining the information given by \cite[Proposition 6.8, (6)]{delcorsoferrilombardo} and \cite[Lemma 7.2]{berge1972}, we see that for every $s\in G_t$ and every $x\in\mathcal{O}_{\widetilde{L}}$, $v_{\widetilde{L}}((s-1)\cdot x)\geq t+v_{\widetilde{L}}(x)$, with equality if and only if $p$ does not divide $v_{\widetilde{L}}(x)$. Let us choose $s=\sigma^2\in G_t$. Since letting the Galois group act on an element does not change its valuation and $v_{\widetilde{L}}(z)=p$, we have $v_{\widetilde{L}}(w\cdot x)\geq p+t+v_{\widetilde{L}}(x)$. In particular, for $x\in L$, we have $$2v_L(w\cdot x)\geq p+t+2v_L(x),$$ with equality if and only if $p$ does not divide $v_L(x)$.
\end{proof}

This result is the analogue of \cite[Proposition 6.8. (6)]{delcorsoferrilombardo} in this case.

\begin{rmk}\normalfont The arithmetical meaning of $\ell$ is that it is congruent to the ramification jump $\frac{t}{2}$ of $L/K$ as a $p$-adic integer.
\end{rmk}

\section{The non-totally ramified case}\label{nontotramified}

In this section we assume that $\widetilde{L}/L$ is unramified, or equivalently, that $M/K$ is unramified. Under this hypothesis, we have that:

\begin{pro} The ring of integers $\mathcal{O}_L$ is $\mathfrak{A}_{L/K}$-free if and only if $\mathcal{O}_{\widetilde{L}}$ is $\mathfrak{A}_{\widetilde{L}/M}$-free.
\end{pro}

This follows immediately from the following more general result, which is an improvement of \cite[Corollary 5.19]{gilrioinduced}.

\begin{pro}\label{proaftertensoring} Let $K$ be the fraction field of a Dedekind domain and let $N/K$ be a Galois extension with Galois group $G=J\rtimes G'$, for some normal subgroup $J$ of $G$. Let $E\coloneqq N^{G'}$ and $F\coloneqq N^J$, and assume that $E/K$ and $F/K$ are arithmetically disjoint. Let $H$ be a Hopf-Galois structure on $E/K$. Then, $\mathcal{O}_E$ is $\mathfrak{A}_H$-free if and only if $\mathcal{O}_N$ is $\mathfrak{A}_{H\otimes_KF}$-free.
\end{pro}
\begin{proof}
The left-to-right implication is just \cite[Proposition 5.18]{gilrioinduced}, so we only need to prove the converse. Let us assume that $\mathcal{O}_N$ is $\mathfrak{A}_{H\otimes_KF}$-free. We argue as in \cite[Proposition 1]{lettl}. The assumption means that $\mathcal{O}_N\cong\mathfrak{A}_{H\otimes_KF}$ as $\mathfrak{A}_{H\otimes_KF}$-modules. By \cite[Proposition 5.18]{gilrioinduced}, the arithmetic disjointness gives that $\mathfrak{A}_{H\otimes_KF}=\mathfrak{A}_H\otimes_{\mathcal{O}_K}\mathcal{O}_F$. Since $\mathcal{O}_F$ is $\mathcal{O}_K$-free, $\mathfrak{A}_{H\otimes_KF}\cong\mathfrak{A}_H^d$ and $\mathcal{O}_N\cong\mathcal{O}_E^d$ as $\mathfrak{A}_H$-modules, where $d\coloneqq[F:K]$. We deduce that $\mathcal{O}_E^d\cong\mathfrak{A}_H^d$ as $\mathfrak{A}_H$-modules. Now, $\mathfrak{A}_H$ is an $\mathcal{O}_K$-algebra finitely generated as $\mathcal{O}_K$-module and $\mathcal{O}_K$ is a complete discrete valuation ring. By the Krull-Schmidt-Azumaya theorem \cite[(6.12)]{curtisreiner}, $\mathcal{O}_E\cong\mathfrak{A}_H$ as $\mathfrak{A}_H$-modules, so $\mathcal{O}_E$ is $\mathfrak{A}_H$-free.
\end{proof}

Now, $\widetilde{L}/M$ is a totally ramified cyclic degree $p$ extension of $p$-adic fields. We know that Theorem \ref{cyclicpfreeness} gives a complete answer for the freeness over the associated order in these extensions. Then, we deduce:

\begin{coro}\label{coropnonram} Let $L/K$ be a degree $p$ extension of $p$-adic fields with dihedral normal closure $\widetilde{L}$ such that $\widetilde{L}/K$ is not totally ramified. Let $e$ be the absolute ramification index of $K$, let $t$ be the ramification jump of $\widetilde{L}/K$ and let $a$ be the remainder of $t$ mod $p$.
\begin{itemize}
    \item[(1)] If $a=0$, then $\mathcal{O}_L$ is $\mathfrak{A}_{L/K}$-free.
    \item[(2)] If $a$ divides $p-1$, then $\mathcal{O}_L$ is $\mathfrak{A}_{L/K}$-free. Moreover, if $t<\frac{pe}{p-1}-1$, the converse holds.
    \item[(3)] If $t\geq\frac{pe}{p-1}-1$, $\mathcal{O}_L$ is $\mathfrak{A}_{L/K}$-free if and only if the length of the expansion of $\frac{t}{p}$ as continued fraction is at most $4$.
\end{itemize}
\end{coro}

For the sake of completeness, we discuss the structure of the associated order $\mathfrak{A}_{L/K}$ as an $\mathcal{O}_K$-module. We will use Galois descent theory. For the presentation of $G$ in \eqref{presentG}, denote $J=\langle\sigma\rangle$. Then the classical Galois structure at $\widetilde{L}/M$ has $M$-Hopf algebra $M[J]$, on which $G$ acts by means of the classical action on $M$ and by conjugation on $J$. Since the Galois action preserves algebraic integers, this also defines an action on $\mathfrak{A}_{\widetilde{L}/M}$.

\begin{pro} Any $\mathcal{O}_M$-basis of $\mathfrak{A}_{\widetilde{L}/M}$ fixed by $\tau$ is an $\mathcal{O}_K$-basis of $\mathfrak{A}_{L/K}$.
\end{pro}
\begin{proof}
The hypothesis that $M/K$ is unramified is the same as saying that $\mathcal{O}_M$ is a Galois extension of $\mathcal{O}_K$ with group $\overline{G}\coloneqq\mathrm{Gal}(M/K)=\langle\tau|_M\rangle$ (see \cite[Definition (2.5)]{childs} and the following discussion). Then, by \cite[(2.12)]{childs}, there is an equivalence between the category $_{\mathcal{O}_K}\mathcal{M}$ of $\mathcal{O}_K$-modules and the category $_{\mathcal{O}_M,\overline{G}}\mathcal{M}$ of $\mathcal{O}_M$-modules with a compatible $\overline{G}$-action, given by the functors $$\begin{array}{ccc}
    _{\mathcal{O}_K}\mathcal{M} & \longleftrightarrow & _{\mathcal{O}_M,\overline{G}}\mathcal{M} \\
    \mathfrak{H} & \longmapsto & \mathcal{O}_M\otimes_{\mathcal{O}_K}\mathfrak{H}, \\
    \mathfrak{A}^{\overline{G}} & \longmapsfrom & \mathfrak{A}.
\end{array}$$ Now, in the case that $\mathfrak{A}=\mathfrak{A}_{\widetilde{L}/M}$, the action of $\overline{G}$ described above is compatible. Since $\widetilde{L}/M$ is of Burnside degree, Byott's uniqueness theorem for Galois extensions \cite[Theorem 1]{byottuniqueness} gives that the only Hopf-Galois structure on $\widetilde{L}/M$ is its classical Galois structure, so $M[J]=H\otimes_KM$. Thus, $\mathfrak{A}_{\widetilde{L}/M}=\mathfrak{A}_{H\otimes_KM}=\mathfrak{A}_{L/K}\otimes_{\mathcal{O}_K}\mathcal{O}_M$. Hence, using the already mentioned action, any $\mathcal{O}_M$-basis of $\mathfrak{A}_{\widetilde{L}/M}$ fixed by $\overline{G}$ is an $\mathcal{O}_K$-basis of $\mathfrak{A}_{L/K}$.
\end{proof}

\begin{example}\normalfont In Example \ref{examplepolymin}, $\widetilde{L}/\mathbb{Q}_3$ is not totally ramified only when the defining polynomial is $g(x)=x^3+3x^2+3$. Indeed, in that case we have $z=\sqrt{-1}\in L$, that is, the unique quadratic subextension of $\widetilde{L}/K$ is $M=\mathbb{Q}_3(\sqrt{-1})$, which is unramified over $\mathbb{Q}_3$. Also, it is seen that $t=a=1$, so $\mathcal{O}_L$ is $\mathfrak{A}_{L/K}$-free. On the other hand, $\mathfrak{A}_{L/K}$ has $\mathcal{O}_K$-basis $$\Big\{1,z(\sigma-\sigma^2),\frac{1+\sigma+\sigma^2}{3}\Big\}.$$ Indeed, it is straightforward to check that it is an $\mathcal{O}_M$-basis of $\mathfrak{A}_{\widetilde{L}/M}$ and it is fixed by the action of $\tau$.
\end{example}

\section{The maximally ramified case}\label{sectmaxram}

In the setting of Section \ref{sectarithm}, let us assume that $p$ divides $\ell$, which we know is equivalent to $t=\frac{2pe}{p-1}$. The aim of this section is to prove that under this hypothesis, $\mathcal{O}_L$ is $\mathfrak{A}_{L/K}$-free. Note that if we write $\ell=pa_0+a$ with $a_0=\lfloor\frac{\ell}{p}\rfloor$ and $0\leq a<p$, we have that $t=(2a_0-1)p+2a$, so $p$ divides $t$ if and only if $p$ divides $\ell$. Therefore, this corresponds just to the situation in Theorem \ref{maintheorem} (1).

The proofs of the cyclic case given in \cite[Chapitre II, Proposition 1.b)]{fertonthesis} and \cite[Corollary 6.5]{delcorsoferrilombardo} use the following result:

\begin{pro}\label{prokummer} If $L/K$ is a maximally ramified cyclic degree $p$ extension of $p$-adic fields, then $K$ contains the primitive $p$-th roots of unity and there is a uniformizer $\pi_K$ such that $L=K(\pi_K^{\frac{1}{p}})$.
\end{pro}

Under the assumption that $a=0$, the extension $\widetilde{L}/M$ is maximally ramified, and since this is a cyclic degree $p$ extension, it satisfies Proposition \ref{prokummer}. The strategy we will follow is to translate the nice properties of $\widetilde{L}/M$ given by this result to the extension $L/K$. Namely, $M$ contains a $p$-th root of unity $\xi$ and there is some element $\gamma\in\mathcal{O}_{\widetilde{L}}$ with $v_{\widetilde{L}}(\gamma)=1$ and $\gamma^p\in \mathcal{O}_M$. Note that we have always that $\xi\notin K$. Indeed, otherwise we would have $p=v(1-\xi)^{p-1}$ for some $v\in\mathcal{O}_K^*$, and then taking valuations gives that $p-1$ divides $e$, whence $t=\frac{2pe}{p-1}$ is even, which is impossible. Call $J=\mathrm{Gal}(\widetilde{L}/M)$, for which the element $\sigma$ in \eqref{presentG} is a generator. Since $\xi\in M-K$, we can assume without loss of generality that $\sigma(\gamma)=\xi\gamma$. Then, for each $0\leq j\leq p-1$ we have  $\sigma(\gamma^j)=\xi^j\gamma^j$. 

Now, the extensions $L/K$ and $M/K$ are linearly disjoint, that is, the canonical map $L\otimes_KM\longrightarrow LM$ is a $K$-linear isomorphism. As a consequence $L\cap M=K$, and we can use this fact to translate the arithmetical properties of $\widetilde{L}/M$ above to the extension $L/K$. Concretely:

\begin{lema}\label{lemamaxram} Let $L/K$ be a degree $p$ extension of $p$-adic fields with dihedral normal closure $\widetilde{L}$ such that $\widetilde{L}/K$ is maximally ramified. Then:
\begin{itemize}
    \item There is some element $\alpha\in\mathcal{O}_L$ with $v_L(\alpha)=1$ such that $\alpha^p\in\mathcal{O}_K$.
    \item For each $0\leq j\leq p-1$ there is some $\lambda_j\in\mathcal{O}_K$ such that $w\cdot\alpha^j=\lambda_j\alpha^j$.
\end{itemize}
\end{lema}
\begin{proof}
Let $\gamma$ be the element in the above discussion and define $\alpha=\gamma\tau(\gamma)$, where $\tau\in G$ is the automorphism in \eqref{presentG}. Since $\tau(\alpha)=\alpha$, $\alpha\in L$. Moreover, $\alpha^p=\gamma^p\tau(\gamma^p)$, which belongs to $M$ because $\gamma^p\in M$ and $\tau$ restricts to $M$. But also $\alpha^p\in L$ because $\alpha\in L$, so $\alpha^p\in L\cap M=K$. This proves the first statement.

Let us prove the second one. Since $\alpha=\gamma\tau(\gamma)$, we have $v_{\widetilde{L}}(\alpha)=2v_L(\alpha)=2$, so $\alpha$ is a uniformizer of $L$. On the other hand, $\tau(\xi)=\xi^{-1}$ because $\tau$ has order $2$ in $G$ and fixes $\alpha$. Then, $$\sigma^k\tau(\gamma)=\tau\sigma^{-k}(\gamma)=\tau(\xi^{-k}\gamma)=\xi^k\tau(\gamma)$$ for every $k$. Then \begin{equation*}
    \begin{split}
        w\cdot\alpha^j&=z(\sigma-\sigma^{-1})\cdot(\gamma^j\tau(\gamma)^j)\\&=z(\xi^{2j}\gamma^j\tau(\gamma)^j-\xi^{-2j}\gamma^j\tau(\gamma)^j)\\&=z(\xi^{2j}-\xi^{-2j})\alpha^j
    \end{split}
\end{equation*} for every $0\leq j\leq p-1$. Let us call $\lambda_j=z(\xi^{2j}-\xi^{-2j})$ for all these $j$. Then $w\cdot\alpha^j=\lambda_j\alpha^j$. To obtain the second statement, it remains to prove that $\lambda_j\in\mathcal{O}_K$. It is clear that they belong to $M$, but also $\lambda_j\alpha^j\in L$ because it is a result of the action of $H$ on $L$. Since $\alpha^j$ has an inverse in $L$, we obtain that $\lambda_j\in L$, whence $\lambda_j\in L\cap M=K$. Finally, it is clear that the $\lambda_j$ are algebraic integers.
\end{proof}

\begin{rmk}\normalfont\label{rmkeigminpoly} The elements $\lambda_j$ are the roots of the minimal polynomial $f$ of $w$ computed in Proposition \ref{w-poly}. Indeed, since $w\cdot\alpha^j=\lambda_j\alpha^j$ for every $0\leq j\leq p-1$ and $w$ operates $K$-linearly, we have $$0=f(w)\cdot\alpha^j=f(\lambda_j)\alpha^j,$$ whence $f(\lambda_j)=0$ for every $0\leq j\leq p-1$. Since $f$ has degree $p$, it has no other roots.
\end{rmk}

Since $w$ generates $H$ as a $K$-algebra, this result completely determines the action of $H$ on $L$. Now, we shall use the techniques described in Section \ref{sectredmethod} to study both the associated order and the $\mathfrak{A}_{L/K}$-freeness of $\mathcal{O}_L$.

\begin{pro}\label{proassocordermaxram} Let $\Lambda=(\lambda_j^i)_{i,j=0}^{p-1}$ and let $\Omega=(\omega_{ij})_{i,j=0}^{p-1}$ be its inverse matrix. Then the associated order $\mathfrak{A}_{L/K}$ of $\mathcal{O}_L$ in $H=K[w]$ has $\mathcal{O}_K$-basis $$v_i=\sum_{l=0}^{p-1}\omega_{li}w^l,\quad0\leq i\leq p-1.$$ Moreover, the action of this basis on $\mathcal{O}_L$ is given by $$v_i\cdot\alpha^j=\delta_{ij}\alpha^j.$$
\end{pro}
\begin{proof}
We take the $\lambda_j\in \mathcal{O}_K$ and the uniformizer $\alpha\in\mathcal{O}_L$ from Lemma \ref{lemamaxram}. We fix the $K$-basis $W=\{w^i\}_{i=0}^{p-1}$ of $H$ and the $\mathcal{O}_K$-basis $\{\alpha^j\}_{j=0}^{p-1}$ of $\mathcal{O}_L$. Since $w^i\cdot\alpha^j=\lambda_j^i\alpha^j$, the $j$-th block $M_j(H,L)$ of the matrix of the action $M(H,L)$ is the matrix whose only non-zero row is the $j$-th one, which is $(\lambda_j^i)_{i=0}^{p-1}$. We may exchange rows so as to place the zero ones at the bottom of the matrix, and since these transformations preserve vectors of integers, $\Lambda$ is a reduced matrix of $M(H,L)$. Then, by Theorem \ref{teoassocorder}, a basis of the associated order is just as described in the statement. Moreover, we also obtain that the coordinates of $v_i\cdot\alpha^j$ are the result of multiplying $M_j(H,L)$ by the $i$-th column of $\Omega$. But the only non-zero row of $M_j(H,L)$ is the $j$-th row of $\Lambda$, which is the inverse of $\Omega$. Therefore, the result of this product is the $(i,j)$-entry of the identity matrix of order $n$, which implies that $v_i\cdot\alpha^j=\delta_{ij}\alpha^j$.
\end{proof}

\begin{coro}\label{coromaxram} Let us assume that $a=0$ and set $\theta=\sum_{j=0}^{p-1}\alpha^j$. Then $\mathfrak{A}_{L/K}=\mathfrak{A}_{\theta}$ and this is the maximal $\mathcal{O}_K$-order in $H$. In particular, $\mathcal{O}_L$ is $\mathfrak{A}_{L/K}$-free with generator $\theta$.
\end{coro}
\begin{proof}
From Proposition \ref{proassocordermaxram} we see that $v_i\cdot\theta=\sum_{j=0}^{p-1}\delta_{ij}\alpha^j=\alpha^i$, and then $\{v_i\cdot\theta\}_{i=0}^{p-1}$ is an $\mathcal{O}_K$-basis of $\mathcal{O}_L$. Since $\{v_i\}_{i=0}^{p-1}$ is an $\mathcal{O}_K$-basis of $\mathfrak{A}_{L/K}$ and $\mathcal{O}_L=\mathfrak{A}_{\theta}\cdot\theta$, necessarily $\mathfrak{A}_{\theta}=\mathfrak{A}_{L/K}$.

Given $0\leq i,j,k\leq p-1$, by Proposition \ref{proassocordermaxram} we have $$(v_iv_j)\cdot\alpha^k=v_i\cdot(\delta_{jk}\alpha^k)=\delta_{ik}\delta_{jk}\alpha^k.$$ 
The bijectivity of the defining map \eqref{mapj} of the Hopf-Galois structure gives that $v_iv_j=\delta_{ij}v_j$, which proves that $\{v_i\}_{i=0}^{p-1}$ is a system of primitive pairwise orthogonal idempotents in $H$. In particular, the product of elements of $H$ written with respect to this system is component-wise. Hence, the map $$\begin{array}{rccl}
    \varphi\colon & H & \longrightarrow & K^p, \\
     & v_i & \longmapsto & e_i\coloneqq(\delta_{ij})_{j=0}^{p-1}.
\end{array}$$ is an isomorphism of $K$-algebras. Now, the maximal $\mathcal{O}_K$-order in $K^p$ is clearly $\mathcal{O}_K^p$, and its inverse image by $\varphi$ is the maximal $\mathcal{O}_K$-order in $H$. But this inverse image is just the set of all $\mathcal{O}_K$-linear combinations of the $v_i$, that is, $\mathfrak{A}_{L/K}$. Then $\mathfrak{A}_{L/K}$ is the maximal $\mathcal{O}_K$-order in $H$.

As for the last statement, it is well known that $\mathfrak{A}_{L/K}$ being the maximal $\mathcal{O}_K$-order in $H$ implies the $\mathfrak{A}_{L/K}$-freeness of $\mathcal{O}_L$ (for a proof, see \cite[Proposition 2.5.5]{trumanthesis}). Finally, that $\theta$ is a generator follows from Proposition \ref{prohnormalfreeness}.
\end{proof}

\begin{example}\normalfont Among the degree $3$ extensions of $\mathbb{Q}_3$ with dihedral degree $6$ normal closure, listed in Example \ref{examplepolymin}, the maximally ramified ones are defined by the radical polynomials $g(x)=x^3+3a$, $a\in\{1,4,7\}$. In each case, $g$ is $3$-Eisenstein, so a root $\alpha$ of $g$ is a uniformizer of $\mathcal{O}_L$. From Example \ref{examplepolymin} we know that $z=\sqrt{-3}$ and $w=z(\sigma-\sigma^{-1})$ has minimal polynomial $f(x)=x^3-9x$. Now, $f$ has roots $0$, $3$ and $-3$, and by Remark \ref{rmkeigminpoly}, these are the elements $\lambda_j$ such that $w\cdot\alpha^j=\lambda_j\alpha^j$. We already know that $w\cdot1=0$. Actually, it is checked without difficulty that $\lambda_1=-3$ and $\lambda_2=3$. Then, the blocks of the matrix $M(H,L)$ of the action are $$M_0(H,L)=\begin{pmatrix}
1 & 0 & 0 \\
0 & 0 & 0 \\
0 & 0 & 0
\end{pmatrix}, \quad M_1(H,L)=\begin{pmatrix}
0 & 0 & 0 \\
1 & -3 & 9 \\
0 & 0 & 0
\end{pmatrix}, \quad M_2(H,L)=\begin{pmatrix}
0 & 0 & 0 \\
0 & 0 & 0 \\
1 & 3 & 9 
\end{pmatrix}.$$ Removing the zero rows of $M(H,L)$ gives the reduced matrix $$\Lambda=\begin{pmatrix}
1 & 0 & 0 \\
1 & -3 & 9 \\
1 & 3 & 9 
\end{pmatrix},$$ with inverse $$\Omega=\frac{1}{18}\begin{pmatrix}
18 & 0 & 0 \\
0 & -3 & 3 \\
-2 & 1 & 1 
\end{pmatrix}.$$ Then, the associated order $\mathfrak{A}_{L/K}$ (and so the maximal $\mathcal{O}_K$-order in $H$) has $\mathcal{O}_K$-basis $$\Big\{\frac{18-2w^2}{18},\frac{-3w+w^2}{18},\frac{3w+w^2}{18}\Big\},$$ which is also a system of primitive pairwise orthogonal idempotents in $H$.
\end{example}

\section{The non-maximally ramified cases}\label{sectnonmaxram}

\subsection{Some generalities}\label{sectgener}

Let $L/K$ be a degree $p$ extension of $p$-adic fields with totally ramified dihedral normal closure $\widetilde{L}$. In this section, we assume that $t<\frac{2pe}{p-1}$, i.e. that $\widetilde{L}/K$ is not maximally ramified. We already know that this happens if and only if $p\nmid t$, and that this is also the same as assuming that $p$ does not divide $\ell$, that is, $a\neq0$.

Fix a uniformizer $\pi_L$ of $L$. Under the assumption on $a$, if we call $\theta=\pi_L^a$, we can obtain analogous statements to \cite[Proposition 6.8. (3),(4),(5)]{delcorsoferrilombardo}.

\begin{pro}\label{proarithmhopfgalois} Let $L/K$ be a separable degree $p$ extension with totally ramified dihedral normal closure and let $t$ be its ramification jump. Let $a$ be the remainder of $\ell$ mod $p$ and assume that $a\neq0$. Call $\theta=\pi_L^a$. Then:
\begin{itemize}
    \item[(1)] $v_L(w^i\cdot\theta)=a+i\ell$ for every $0\leq i\leq p-1$.
    \item[(2)] $\{w^i\cdot\theta\}_{i=0}^{p-1}$ is a $K$-basis of $H$, so $\theta$ generates an $H$-normal basis for $L/K$.
    \item[(3)] $v_L(w^p\cdot\theta)=ep+\frac{p^2+t}{2}+a$.
\end{itemize}
\end{pro}
\begin{proof}
\begin{itemize}
    \item[(1)] It follows from an easy induction on $i$ applying Proposition \ref{proraisevaluation} at each step, using the fact that $a+i\ell$ is not divisible by $p$ for every $0\leq i\leq p-2$.
    \item[(2)] It is immediate from the fact that the valuations of $w^i\cdot\theta$ are all different mod $p$, which follows from (1).
    \item[(3)] Let $e$ be the absolute ramification index of $K$. Taking valuations in the expression of $w^p$ in Proposition \ref{w-poly}, we obtain \begin{equation*}
        \begin{split}
            v_L(w^p\cdot\theta)&=\mathrm{min}_{1\leq j\leq \frac{p-1}{2}}\Big(ep+a+p\ell-jt\Big)\\&=ep+a+p\ell-\frac{p-1}{2}t\\&=ep+a+\frac{p^2+t}{2}.
        \end{split}
    \end{equation*}
\end{itemize}
\end{proof}

In analogy with the cyclic case (see \cite[Definition 6.9]{delcorsoferrilombardo}), we define \begin{equation}\label{defnui}
    \nu_i=\Big\lfloor\frac{a+i\ell}{p}\Big\rfloor,\quad0\leq i\leq p-1.
\end{equation} We will indistinctly use this definition or the following alternative expression: \begin{equation}\label{defaltnui}
    \nu_i=\Big\lfloor\frac{a+i(pa_0+a)}{p}\Big\rfloor=ia_0+\Big\lfloor(i+1)\frac{a}{p}\Big\rfloor.
\end{equation}

Clearly the sequence $\{\nu_i\}_{i=0}^{p-1}$ is increasing. Some basic properties of the numbers $\nu_i$ are listed in the following:

\begin{pro}\label{lemaparam} The following statements hold:
\begin{itemize}
    \item[(1)] $\nu_{p-1}=a+(p-1)a_0$, where $a_0=\lfloor\frac{\ell}{p}\rfloor$.
    \item[(2)] $\nu_{p-1}\leq e+\frac{p-1}{2}$.
    \item[(3)] $p\Big(e+\frac{p-1}{2}-\nu_{p-1}\Big)+a=\frac{p-1}{2}\Big(\frac{2pe}{p-1}-t\Big)$.
    \item[(4)] $\nu_{p-1}=e+\frac{p-1}{2}$ if and only if $t\geq\frac{2pe}{p-1}-2$.
    \item[(5)] Given $0\leq k\leq p-3$, $\nu_{k+2}-\nu_k\geq1$. Moreover, if $a_0>0$, then $\nu_{k+2}-\nu_k\geq2$.
    \item[(6)] $\nu_{p-1}-\nu_{p-2}$ and $\nu_{p-2}-\nu_{p-3}$ are lower bounded by $1$.
    \item[(7)] $\nu_s<e+\Big\lceil\frac{s}{2}\Big\rceil$ for all $0\leq s\leq p-2$.
\end{itemize}
\end{pro}
\begin{proof}
\begin{itemize}
    \item[(1)] It is enough to replace $\ell=pa_0+a$ in the definition of $\nu_{p-1}$.
    \item[(2)] Since $t\leq\frac{2pe}{p-1}$, we have $$\ell\leq\frac{p+\frac{2pe}{p-1}}{2}=\frac{p(p-1)+2pe}{2(p-1)}.$$ Then, $$\nu_{p-1}=\Big\lfloor\frac{a+(p-1)\ell}{p}\Big\rfloor\leq\Big\lfloor\frac{a+\frac{p(p-1+2e)}{2}}{p}\Big\rfloor=\Big\lfloor\frac{a}{p}+\frac{p-1}{2}+e\Big\rfloor=e+\frac{p-1}{2}.$$
    \item[(3)] We check the validity of the identity directly: \begin{equation*}
        \begin{split}
            p\Big(e+\frac{p-1}{2}-\nu_{p-1}\Big)+a&=pe+\frac{p-1}{2}p-pa-p(p-1)a_0+a\\&=pe+\frac{p-1}{2}p-(p-1)a-p(p-1)a_0\\&=pe+\frac{p-1}{2}p-(p-1)\ell\\&=pe+\frac{p-1}{2}(p-2\ell)=pe-\frac{p-1}{2}t=\frac{p-1}{2}\Big(\frac{2pe}{p-1}-t\Big).
        \end{split}
    \end{equation*}
    \item[(4)] This statement follows easily from the identity in (3). If $\nu_{p-1}=e+\frac{p-1}{2}$, we have that $t=\frac{2pe}{p-1}-\frac{2a}{p-1}\geq\frac{2pe}{p-1}-2$ because $a\leq p-1$. Conversely, if $t\geq\frac{2pe}{p-1}-2$, then $p\Big(e+\frac{p-1}{2}-\nu_{p-1}\Big)+a\leq p-1$, which implies that $\nu_{p-1}=e+\frac{p-1}{2}$.
    \item[(5)] Since $\nu_i=ia_0+\lfloor(i+1)\frac{a}{p}\rfloor$ for every $0\leq i\leq p-1$, the second part of the statement is trivial. Assume that $a_0=0$. We have that $$\nu_{k+2}-\nu_k=\Big\lfloor(k+3)\frac{a}{p}\Big\rfloor-\Big\lfloor(k+1)\frac{a}{p}\Big\rfloor.$$ Since $a_0=0$, $a=\frac{p+t}{2}\geq\frac{p+1}{2}$, so $2a\geq p+1$. Then, $(k+3)\frac{a}{p}-(k+1)\frac{a}{p}=\frac{2a}{p}>1$. Necessarily, $\nu_{k+2}-\nu_k\geq1$.
    \item[(6)] If $a_0>0$, the statement is trivial. Assume that $a_0=0$. Immediately we obtain that $\nu_{p-1}=a$ and $\nu_{p-2}=\Big\lfloor a-\frac{a}{p}\Big\rfloor=a-1$. Finally, $\nu_{p-3}=\Big\lfloor a-\frac{2a}{p}\Big\rfloor$, and since $a_0=0$, $p+1\leq 2a\leq 2(p-1)$, so $\nu_{p-3}=a-2$.
    \item[(7)] Using (5), we obtain that $\nu_{p-2}<\nu_{p-1}\leq e+\frac{p-1}{2}$ and $\nu_{p-3}<\nu_{p-1}-1\leq e+\frac{p-3}{2}$. If $s$ is odd, using successively (4) gives $\nu_{p-2}-\nu_s\geq\frac{p-2-s}{2}$, so $$\nu_s\leq\nu_{p-2}-\frac{p-2-s}{2}<e+\frac{p-1}{2}-\frac{p-2-s}{2}=e+\frac{s+1}{2}.$$ If $s$ is even, similarly we obtain $$\nu_s\leq\nu_{p-3}-\frac{p-3-s}{2}<e+\frac{p-3}{2}-\frac{p-3-s}{2}=e+\frac{s}{2}.$$ This completes the proof.
\end{itemize}
\end{proof}

\begin{rmk}\normalfont
The quantity in Proposition \ref{lemaparam} (3) is the precision of the scaffold on $L/K$ introduced by Elder in \cite{elder}, see Section \ref{sectfreenessnonmax} below.
\end{rmk}

We see that the numbers $\nu_i$ have some meaningful differences with respect to the cyclic case. The upper bound is now $e+\frac{p-1}{2}$ instead of $e$ (see \cite[Lemma 6.11]{delcorsoferrilombardo}). This number is not new in this paper: it is just the power of $\pi_K$ in the coefficient of $w^m$ in the expression of $w^{p-1+m}$ in Proposition \ref{w-highpoly}.
Another difference is that the sequence $\nu_0\leq\nu_1\leq\dots\leq\nu_{p-1}$ increases rapidly, in the sense that there cannot be two consecutive equalities. These two facts will play an important role to prove Theorem \ref{maintheorem} (4). The other difference is that now the upper bound $e+\frac{p-1}{2}$ for $\nu_{p-1}$ is reached exactly in the cases in which $t\geq\frac{2pe}{p-1}-2$, so now the almost maximally ramified condition $t\geq\frac{2pe}{p-1}-1$ does not seem to be important for the behaviour of $\mathcal{O}_L$ as $\mathfrak{A}_{L/K}$-module.

Next, we proceed to describe $\mathfrak{A}_{\theta}$ in terms of $\{\nu_i\}_{i=0}^{p-1}$ as in the cyclic case.

\begin{pro}\label{propintbasisL} Assume that $a\neq0$ and let $\theta=\pi_L^a$. Then, $\mathfrak{A}_{\theta}$ is $\mathcal{O}_K$-free with $\mathcal{O}_K$-basis $$\{\pi_L^{-\nu_i}w^i\}_{i=0}^{p-1}.$$ Consequently, $\{\pi_L^{-\nu_i}w^i\cdot\theta\}_{i=0}^{p-1}$ is an integral basis for $L/K$.
\end{pro}
\begin{proof}
Let $h\in H$, so that $h=\sum_{i=0}^{p-1}h_iw^i$ for unique $h_0,\dots,h_{p-1}\in K$. Then, $h\cdot\theta\in\mathcal{O}_L$ if and only if $$\sum_{i=0}^{p-1}h_iw^i\cdot\theta\in\mathcal{O}_L.$$ Now, by Proposition \ref{proarithmhopfgalois} (1), the valuations of the terms $w^i\cdot\theta$ are all diferent, so the condition above happens if and only if $h_iw^i\cdot\theta\in\mathcal{O}_L$ for all $0\leq i\leq p-1$. Since $v_L(w^i\cdot\theta)=a+i\ell$, this is equivalent to $v_L(h_i)\geq-\frac{a+i\ell}{p}$, which at the same time is equivalent to $v_L(h_i)\geq-\nu_i$. The last statement follows from the fact that $\mathcal{O}_L=\mathfrak{A}_{\theta}\cdot\theta$.
\end{proof}

This result is the analogue of \cite[Proposition 2.a)]{ferton1972} and \cite[Theorem 6.10]{delcorsoferrilombardo}.

\subsection{Determination of the associated order}

In this part we will prove Theorem \ref{maintheorem} (2). We will see that we can translate the proof of Theorem \ref{cyclicpfreeness} (2) in \cite[Section 7]{delcorsoferrilombardo} to our case. For each $0\leq i\leq p-1$, let us denote \begin{equation*}
    \begin{split}
        n_i&=\mathrm{min}_{0\leq j\leq p-1-i}(\nu_{i+j}-\nu_j)\\&=ia_0+\mathrm{min}_{0\leq j\leq p-1-i}\Big(\Big\lfloor(i+j+1)\frac{a}{p}\Big\rfloor-\Big\lfloor(j+1)\frac{a}{p}\Big\rfloor\Big).
    \end{split}
\end{equation*}

It is clear that $n_i\leq \nu_i$ for all $0\leq i\leq p-1$. We start with the following result, which is the analogue of \cite[Corollary 7.1]{delcorsoferrilombardo}.

\begin{pro}\label{propineqactionsep} Let $0\leq i,j\leq p-1$ such that $i+j\geq p$. Then, $$(\pi_K^{-\nu_i}w^i)\cdot(\pi_K^{-\nu_j}w^j\cdot\theta)\in\mathcal{O}_L.$$
\end{pro}
\begin{proof}
First, we have that $$(\pi_K^{-\nu_i}w^i)\cdot(\pi_K^{-\nu_j}w^j\cdot\theta)=\pi_K^{-\nu_i-\nu_j}w^{i+j-p}w^p\cdot\theta.$$ Following Proposition \ref{proarithmhopfgalois}, we compute $$v_L(w^{i+j-p}w^p\cdot\theta)\geq v_L(w^p\cdot\theta)+(i+j-p)\frac{p+t}{2}=ep+\frac{p^2+t}{2}+a+(i+j-p)\frac{p+t}{2}.$$ Then, \begin{equation*}
    \begin{split}
        v_L((\pi_K^{-\nu_i}w^i)\cdot(\pi_K^{-\nu_j}w^j\cdot\theta))&\geq(-\nu_i-\nu_j)p+ep+\frac{p^2+t}{2}+a+(i+j-p)\frac{p+t}{2}\\&\geq-a-i\frac{p+t}{2}-a-j\frac{p+t}{2}+ep+\frac{p^2+t}{2}+a+(i+j-p)\frac{p+t}{2}\\&=ep-\frac{p-1}{2}t-a.
    \end{split}
\end{equation*} Using Proposition \ref{lemaparam} (3), we see that $ep-\frac{p-1}{2}t-a=e+\frac{p-1}{2}-\nu_{p-1}$, which is known to be non-negative from Proposition \ref{lemaparam} (2).
\end{proof}

Now, we proceed to restate and proof Theorem \ref{maintheorem} (2).

\begin{pro}\label{probasisassocorder} Let $L/K$ be a degree $p$ extension of $p$-adic fields with dihedral normal closure $\widetilde{L}$ such that $\widetilde{L}/K$ is not maximally ramified. Then, $\mathfrak{A}_{L/K}$ is $\mathcal{O}_K$-free and has $\mathcal{O}_K$-basis $\{\pi_K^{-n_i}w^i\}_{i=0}^{p-1}$.
\end{pro}
\begin{proof}
Many of the details will be omitted, since we essentially follow the arguments in the proof of \cite[Theorem 7.2]{delcorsoferrilombardo}. Let us take $h=\sum_{i=0}^{p-1}h_iw^i\in H$ with $h_i\in K$. As usual, call $\theta=\pi_L^a$.

Assume that $h\in \mathfrak{A}_{L/K}$, so that $h\cdot\theta\in\mathcal{O}_L$. Using Proposition \ref{proarithmhopfgalois} (1), we see that the valuations of the elements $h_iw^i\cdot\theta$, $0\leq i\leq p-1$, are all distinct mod $p$. Hence each of these elements belongs to $\mathcal{O}_L$, which gives $v_K(h_i)\geq-\nu_i$ for every $0\leq i\leq p-1$. On the other hand, for every $0\leq j\leq p-1$ we have that $$h\cdot(\pi_K^{-\nu_j}w^j\cdot\theta)=\sum_{i=0}^{p-1-j}h_i\pi_K^{-\nu_j}w^{i+j}\cdot\theta+\sum_{i=p-j}^{p-1}h_i\pi_K^{-\nu_j}w^{i+j}\cdot\theta\in\mathcal{O}_L,$$ (where the second sum only appears if $j\geq 1$). The second sum belongs to $\mathcal{O}_L$ by Proposition \ref{propineqactionsep}, so the first one also does. Again, the valuations of its summands are all distinct mod $p$, obtaining that $h_i\pi_K^{-\nu_j}w^{i+j}\cdot\theta\in\mathcal{O}_L$ whenever $i+j\leq p-1$. Taking valuations, this gives that $\nu_{i+j}-\nu_j\geq-v_K(h_i)$ for all $0\leq j\leq p-1-i$, which is easily seen to be equivalent to $v_K(h_i)\geq-n_i$ for all $0\leq i\leq p-1$.

Conversely, if $v_K(h_i)\geq-n_i$ for all $0\leq i\leq p-1$, we can go back in the previous argument to show that $h_i\pi_K^{-\nu_j}w^{i+j}\cdot\theta\in\mathcal{O}_L$ for all $j$ such that $i+j\leq p-1$, and Proposition \ref{propineqactionsep} shows that it is also satisfied for $i+j\geq p$. Then $h_iw^i\in\mathfrak{A}_{L/K}$ for every $0\leq i\leq p-1$, whence $h\in\mathfrak{A}_{L/K}$.
\end{proof}

\begin{example}\normalfont\label{examplep3} Among the degree $3$ extensions of $\mathbb{Q}_3$ in Example \ref{examplepolymin} with dihedral normal closure $\widetilde{L}$, the ones for which $\widetilde{L}/\mathbb{Q}_3$ is totally and not maximally ramified correspond to the defining polynomials $x^3+3x+3$ and $x^3+6x+3$. Note that since $e=1$, $t$ is odd and $\widetilde{L}/\mathbb{Q}_3$ is not maximally ramified, $1\leq t\leq\frac{2pe}{p-1}=3$ implies that $t=1$, so $\ell=a=2$. We take $3$ as uniformizer of $\mathbb{Z}_3$ and $\theta=3^2=9$. We can find either directly or using Proposition \ref{lemaparam} that $$\nu_0=0,\quad\nu_1=1,\quad\nu_2=2.$$ Note that $\nu_2=a=e+\frac{p-1}{2}$, which is compatible with the fact that $t\geq\frac{2pe}{p-1}-2$. On the other hand, one sees directly that $$n_0=0,\quad n_1=1,\quad n_2=2.$$ We deduce that $\mathfrak{A}_{\theta}=\mathfrak{A}_{L/K}$ and that they have $\mathcal{O}_K$-basis $$\Big\{1,\frac{w}{3},\frac{w^2}{9}\Big\}.$$ 
\end{example}

\begin{example}\normalfont\label{examplep>3} If $p>3$, there are two possible degree $p$ extensions of $\mathbb{Q}_p$ with totally ramified dihedral degree $2p$ normal closure, given by the polynomials $x^p+2ax^{\frac{p-1}{2}}+p$, $a\in\{2,p-2\}$ (see \cite[Theorem 5.2 2.]{awtreyedwards}). We shall find the explicit form of $\mathfrak{A}_{L/K}$ for these cases. Actually, the development below will also be valid for the cases in which $K/\mathbb{Q}_p$ is unramified, since we will only use that $e=1$. 

We have that $1\leq t\leq\frac{2pe}{p-1}$ with $t$ odd and $e=1$, so $t=1$ and $\ell=a=\frac{p+1}{2}$. Then, $$\nu_i=\Big\lfloor\frac{\frac{p+1}{2}(i+1)}{p}\Big\rfloor=\Big\lfloor\frac{ip+p+i+1}{2p}\Big\rfloor=\begin{cases}\lfloor\frac{i+1}{2}+\frac{i+1}{2p}\rfloor=\frac{i+1}{2} & \hbox{if }i\hbox{ is odd,} \\\lfloor\frac{i}{2}+\frac{p+i+1}{2p}\rfloor & \hbox{if }i\hbox{ is even.}\end{cases}$$ For even $i<p-1$, clearly $\nu_i=\frac{i}{2}$. On the other hand, $\nu_{p-1}=\frac{p+1}{2}$. Then, $$\nu_i=\begin{cases}
\frac{i}{2} & \hbox{if }i\hbox{ is even and }i<p-1, \\
\frac{i+1}{2} & \hbox{if }i\hbox{ is odd}, \\
\frac{p+1}{2} & \hbox{if }i=p-1.
\end{cases}$$ 

From the definition, we find $$n_i=\begin{cases}
0 & \hbox{if }i=1, \\
\frac{i+1}{2} & \hbox{if }i\hbox{ is odd and }i>1, \\
\frac{i}{2} & \hbox{if }i\hbox{ is even and }i<p-1, \\
\frac{p+1}{2} & \hbox{if }i=p-1. \\
\end{cases}$$ Note that we always have $n_0=n_1=0$, $n_2=1$, $n_{p-2}=\frac{p-1}{2}$ and $n_{p-1}=\frac{p+1}{2}$, which are all the terms of the sequence for $p=5$. For $p>5$, this is completed with $n_i=n_{i+1}=\frac{i+1}{2}$ for every odd $i$ such that $3\leq i\leq p-4$. Then $\mathfrak{A}_{L/K}$ has $\mathcal{O}_K$-basis $\Big\{1,w,\frac{w^2}{p},\frac{w^3}{p^2},\frac{w^4}{p^3}\}$ for $p=5$ and $$\Big\{1,w,\frac{w^2}{p},\frac{w^3}{p^2},\frac{w^4}{p^2},\dots,\frac{w^{p-4}}{p^{\frac{p-3}{2}}},\frac{w^{p-3}}{p^{\frac{p-3}{2}}},\frac{w^{p-2}}{p^{\frac{p-1}{2}}},\frac{w^{p-1}}{p^{\frac{p+1}{2}}}\Big\}$$ for $p>5$. In particular, $\mathfrak{A}_{\theta}\neq\mathfrak{A}_{L/K}$.
\end{example}

\subsection{Characterization of the equality $\mathfrak{A}_{L/K}=\mathfrak{A}_{\theta}$}

In this part, we determine whether the equality $\mathfrak{A}_{L/K}=\mathfrak{A}_{\theta}$ is fulfilled. First, we will study under which conditions $\mathfrak{A}_{\theta}$ is a subring of $H$.

\begin{lema}\label{lemasubr} Assume that $a\neq0$ and let $\theta=\pi_L^a$. Let $1\leq m\leq p-1$. Given $\nu\in\mathbb{Z}$, $\pi_K^{-\nu}w^{p-1+m}\in\mathfrak{A}_{\theta}$ if and only if $\nu\leq e+\frac{p-1}{2}+\nu_m$.
\end{lema}
\begin{proof}
We will prove it only for the case that $m$ is odd, as the proof for even $m$ is similar. Let us write $m=2m'-1$ as in Proposition \ref{w-highpoly}. By that result and Proposition \ref{propintbasisL}, we have that $\pi_K^{-\nu}w^{p-1+m}\in\mathfrak{A}_{\theta}$ if and only if $$e+m'+j-1-\nu\geq-\nu_{p-2j}\quad\hbox{for all }1\leq j\leq\frac{p+1}{2}-m',$$ $$2e+m'+j-1-\nu\geq-\nu_{p-2j}\quad\hbox{for all }\frac{p+1}{2}-m'+1\leq j\leq\frac{p-1}{2}.$$ This is equivalent to $\nu$ being upper bounded by all the terms \begin{equation}\label{eqseqj}
    \begin{split}
        e+m'+j-1+\nu_{p-2j},\quad&1\leq j\leq \frac{p+1}{2}-m',\\
        2e+m'+j-1+\nu_{p-2j},\quad&\frac{p+1}{2}-m'+1\leq j\leq\frac{p-1}{2}.
    \end{split}
\end{equation} This is equivalent to $j$ being upper bounded by the smallest of all these terms. The terms at each of the two sequences defined by the two lines in \eqref{eqseqj} decrease as $j$ increases. Indeed, by Proposition \ref{lemaparam} (4), we have that $$(e+m'+j-1+\nu_{p-2j})-(e+m'+j+\nu_{p-2(j+1)})=\nu_{p-2j}-\nu_{p-2(j+1)}-1\geq0$$ for all $1\leq j\leq\frac{p+1}{2}-m'$, and similarly for the terms in the second line. Hence, at each line the greatest value of $j$ gives the minimum among all the terms at that line. Therefore, the minimum at the first (resp. second) line of \eqref{eqseqj} is $e+\frac{p-1}{2}+\nu_m$ (resp. $2e+m'+\frac{p-1}{2}-1+\nu_1$). Finally, Proposition \ref{lemaparam} (7) gives $\nu_m<e+m'$, and this combined with the fact that $\nu_1\geq1$ yields $$2e+m'+\frac{p-1}{2}-1+\nu_1>e+\frac{p-1}{2}+\nu_m.$$ Hence, $\pi_K^{-\nu}w^{p-1+m}\in\mathfrak{A}_{\theta}$ if and only if $\nu\leq e+\frac{p-1}{2}+\nu_m$, finishing the proof.
\end{proof}

\begin{pro}\label{prosubr} For $\theta=\pi_L^a$, $\mathfrak{A}_{\theta}$ is a ring if and only if the following implications are satisfied:
\begin{itemize}
    \item[(1)] If $i+j\leq p-1$, then $\nu_i+\nu_j\leq\nu_{i+j}$.
    \item[(2)] If $i+j\geq p$, then $\nu_i+\nu_j\leq e+\frac{p-1}{2}+\nu_{i+j+1-p}$.
\end{itemize}
\end{pro}
\begin{proof}
Since $\mathfrak{A}_{\theta}$ is a full $\mathcal{O}_K$-lattice of $H$ and contains the identity, it is a ring if and only if it is closed under the multiplication of $H$. Moreover, this is equivalent to the product being closed for the elements of the basis $\{\pi_K^{-\nu_k}w^k\}_{k=0}^{p-1}$ of $\mathfrak{A}_{\theta}$. Then, the necessity (for $\mathfrak{A}_{\theta}$ being a ring) of the implication in statement (1) is obtained immediately. As for the case $i+j\geq p$, we call $m=i+j-(p-1)$, so that $i+j=p-1+m$. By Lemma \ref{lemasubr} we have that $\pi_K^{-\nu_i}w^i\pi_K^{-\nu_j}w^j=\pi_K^{-\nu_i-\nu_j}w^{p-1+m}\in\mathfrak{A}_{\theta}$ if and only if $\nu_i+\nu_j\leq e+\frac{p-1}{2}+\nu_m=e+\frac{p-1}{2}+\nu_{i+j+1-p}$, which gives the necessity of the statement (2). Also, we obtain directly that if (1) and (2) are valid, then $\mathfrak{A}_{\theta}$ is a ring.
\end{proof}

This result is the analogue of \cite[Lemme]{ferton1972}. Now, we need to check under which conditions the implications in (1) and (2) are actually satisfied. We will use the following technical lemma.

\begin{lema}\label{lematestequal} Assume that $a\neq0$. The following statements are satisfied:
\begin{itemize}
    \item[(1)] If $1\leq i,j\leq p-1$ and $i+j\geq p$, then $\Big\lfloor(i+1)\frac{a}{p}\Big\rfloor+\Big\lfloor(j+1)\frac{a}{p}\Big\rfloor\leq e+\frac{p-1}{2}-(p-1)a_0+\Big\lfloor(i+j+2-p)\frac{a}{p}\Big\rfloor$.
    \item[(2)] $a$ divides $p-1$ if and only if for every $1\leq i,j\leq p-1$ such that $i+j\leq p-1$, $\Big\lfloor(i+1)\frac{a}{p}\Big\rfloor+\Big\lfloor(j+1)\frac{a}{p}\Big\rfloor\leq\Big\lfloor(i+j+1)\frac{a}{p}\Big\rfloor$
\end{itemize}
\end{lema}
\begin{proof}
We shall follow an adaptation of the arguments in \cite[Proof of Proposition 4.a)]{ferton1972} and \cite[Chapitre II, Lemme 3]{fertonthesis} to our case.

If $i+j\geq p$, then $$\Big\lfloor(i+1)\frac{a}{p}\Big\rfloor+\Big\lfloor(j+1)\frac{a}{p}\Big\rfloor\leq\Big\lfloor(i+j+2)\frac{a}{p}\Big\rfloor=a+\Big\lfloor(i+j+2-p)\frac{a}{p}\Big\rfloor.$$ Now, by Lemma \ref{lemaparam} (1),(2) we have $\nu_{p-1}=a+(p-1)a_0\leq e+\frac{p-1}{2}$, so $$\Big\lfloor(i+1)\frac{a}{p}\Big\rfloor+\Big\lfloor(j+1)\frac{a}{p}\Big\rfloor\leq e+\frac{p-1}{2}-(p-1)a_0+\Big\lfloor(i+j+2-p)\frac{a}{p}\Big\rfloor.$$ This concludes the proof of (1).

Let us prove (2). First, we assume that $a$ divdes $p-1$ and call $d=\frac{p-1}{a}$. By \cite[Lemma 7.4]{delcorsoferrilombardo}, we have that $\Big\lfloor(k+1)\frac{a}{p}\Big\rfloor=\Big\lfloor\frac{k}{d}\Big\rfloor$ for every $1\leq k\leq p-1$. Then, $$\Big\lfloor(i+1)\frac{a}{p}\Big\rfloor+\Big\lfloor(j+1)\frac{a}{p}\Big\rfloor=\Big\lfloor\frac{i}{d}\Big\rfloor+\Big\lfloor\frac{j}{d}\Big\rfloor\leq\Big\lfloor\frac{i+j}{d}\Big\rfloor=\Big\lfloor(i+j+1)\frac{a}{p}\Big\rfloor.$$ If $a$ does not divide $p-1$, choosing integers $i,j$ as in \cite[Proof of Proposition 4.a)]{ferton1972}, we can show, with the help of Proposition \ref{propcontfrac}, that $$\widehat{(i+1)\frac{a}{p}}+\widehat{(j+1)\frac{a}{p}}<\widehat{(i+j+1)\frac{a}{p}},$$ which is equivalent to $$\Big\lfloor(i+1)\frac{a}{p}\Big\rfloor+\Big\lfloor(j+1)\frac{a}{p}\Big\rfloor>\Big\lfloor(i+j+1)\frac{a}{p}\Big\rfloor.$$
\end{proof}

We are ready to prove the main result of this section.

\begin{pro}\label{propequalorders} Let $L/K$ be a degree $p$ extension of $p$-adic fields with dihedral normal closure $\widetilde{L}$. Let $t$ be the ramification jump of $\widetilde{L}/K$ and let $a$ be the remainder of $\ell=\frac{p+t}{2}$ mod $p$. Let $\pi_L$ be a uniformizer of $L$ and, if $a=0$, assume that $v_K(\pi_L^p)=1$. Define $\theta\coloneqq\sum_{i=0}^{p-1}\pi_L^i$ if $a=0$ and $\theta\coloneqq\pi_L^a$ otherwise. Then $\mathfrak{A}_{L/K}=\mathfrak{A}_{\theta}$ if and only if $a=0$ or $a\mid p-1$.
\end{pro}
\begin{proof}
By Proposition \ref{prohnormalfreeness}, $\mathfrak{A}_{\theta}=\mathfrak{A}_{L/K}$ if and only if $\mathfrak{A}_{\theta}$ is a subring of $H$. If $a=0$, by Corollary \ref{coromaxram}, we have that $\mathfrak{A}_{\theta}=\mathfrak{A}_{L/K}$. Therefore, it is enough to prove that if $a\neq0$, $\mathfrak{A}_{\theta}$ is a subring if and only if $a\mid p-1$. We will follow the characterization given in Proposition \ref{prosubr}.

Let $0\leq i,j\leq p-1$ be such that $i+j\geq p$. The inequality $\nu_i+\nu_j\leq e+\frac{p-1}{2}+\nu_{i+j+1-p}$ holds if and only if $$\Big\lfloor(i+1)\frac{a}{p}\Big\rfloor+\Big\lfloor(j+1)\frac{a}{p}\Big\rfloor\leq e+\frac{p-1}{2}-(p-1)a_0+\Big\lfloor(i+j+2-p)\frac{a}{p}\Big\rfloor.$$ Now, Lemma \ref{lematestequal} says that this is always the case. As for the case $i+j\leq p-1$, $\nu_i+\nu_j\leq\nu_{i+j}$  if and only if $$\Big\lfloor(i+1)\frac{a}{p}\Big\rfloor+\Big\lfloor(j+1)\frac{a}{p}\Big\rfloor\leq\Big\lfloor(i+j+1)\frac{a}{p}\Big\rfloor.$$ By Lemma \ref{lematestequal}, this holds for every $0\leq i,j\leq p-1$ such that $i+j\leq p-1$ if and only if $a$ divides $p-1$. Therefore, the implications in statements (1) and (2) of Proposition \ref{prosubr} hold if and only if $a$ divides $p-1$.
\end{proof}

\subsection{Freeness over the associated order}\label{sectfreenessnonmax}

Again, we assume that $a\neq0$. The aim of this part is to prove Theorem \ref{maintheorem} (3). Using the setting established in Section \ref{sectgener}, we could adapt the approach in \cite{ferton1972} for the cyclic case. However, for the sake of simplicity, we will make use of the machinery of scaffolds, which we recalled in Section \ref{sectscaffolds}. We assume all the notations that we introduced there.

As already mentioned in Remark \ref{rmkdescrH}, in \cite{elder} Elder works with typical extensions, which are the ones that are not generated by a $p$-th root of a uniformizer. Since $a\neq0$, $L/K$ is a typical extension. Then, by \cite[Theorem 3.5]{elder} and the paragraph following it, $L/K$ posseses a scaffold consisting of:
\begin{itemize}
    \item The family of elements $\lambda_k=x^{\mathfrak{a}(k)}\pi_K^{f_k}\in\mathbb{Z}$, where $x$ is a solution of the equation in \cite[Theorem 2.2]{elder} and $f_k$ is defined by the equality $k=-\mathfrak{a}(k)b+f_kp$ (which makes sense because $\mathfrak{a}(k)\equiv-kb^{-1}\,(\mathrm{mod}\,p)$ for every $0\leq k\leq p-1$).
    \item The element $\Psi$ defined in \cite[Theorem 3.1]{elder}.
\end{itemize}

The precision of the scaffold is $$\mathfrak{c}=pe-\frac{p-1}{2}t,$$ which is just the number that appears in Proposition \ref{lemaparam} (3). There, we proved that $\mathfrak{c}=pl+a$, where $$l=e+\frac{p-1}{2}-\nu_{p-1}.$$ By Proposition \ref{lemaparam} (2),(4) we have that $l\geq0$ with strict inequality if and only if $t<\frac{2pe}{p-1}-2$. The proof of Theorem \ref{maintheorem} (3) will follow from translating Theorem \ref{criteriascaffold} to the case of the scaffold above.

\begin{coro} Let us assume that $a\neq0$. If $a$ divides $p-1$, then $\mathcal{O}_L$ is $\mathfrak{A}_{L/K}$-free. Moreover, if $t<\frac{2pe}{p-1}-2$, then the converse holds.
\end{coro}
\begin{proof}
The shift parameter of the scaffold above is the number $b$ introduced in \cite[Theorem 2.2]{elder}, which with our notation coincides with $\frac{t-p}{2}$. Note that this is congruent to $a$ mod $p$. Following \cite[Remark 2.9]{byottchildselder}, if we multiply the element $\Psi$ by a suitable power of the uniformizer $\pi_L$, we obtain a scaffold of the same precision and with shift parameter $a$. From now on, we work with this new scaffold. In order to apply Theorem \ref{criteriascaffold}, we need to check whether the identity $w(s)=d(s)$ holds for all $s\in\mathbb{S}_p$. Since $\mathfrak{b}(s)=as$ for every $s\in\mathbb{S}_p$, it is easy to check that $d(s)=\nu_s-sa_0$ and $w(s)=n_s-sa_0$ for all $0\leq s\leq p-1$. Then, $w(s)=d(s)$ for all $s\in\mathbb{S}_p$ is equivalent to $\nu_i=n_i$ for all $0\leq i\leq p-1$. By Proposition \ref{propintbasisL} and Proposition \ref{probasisassocorder}, this is the same as $\mathfrak{A}_{L/K}=\mathfrak{A}_{\theta}$. Since $a\neq0$, from Proposition \ref{propequalorders} we know that this happens if and only if $a\mid p-1$. We have then proved that $w(s)=d(s)$ for all $s\in\mathbb{S}_p$ is equivalent to $a\mid p-1$.

It is easily checked that $(r\circ(-\mathfrak{b}))(p-1)=a$, so $\mathfrak{a}(a)=p-1$ and we can apply Theorem \ref{criteriascaffold} with $b=a$. Since $l\geq0$, the inequality $\mathfrak{c}\geq a$ always holds. Now, this implies the weak assumption in Theorem \ref{criteriascaffold} (1), giving that if $a\mid p-1$, then $\mathcal{O}_L$ is $\mathfrak{A}_{L/K}$-free. If $t<\frac{2pe}{p-1}-2$, then Proposition \ref{lemaparam} (4) gives that $\nu_{p-1}<\frac{2pe}{p-1}-2$, that is, $l>0$. Then $\mathfrak{c}\geq p+a$, which is the strong assumption in Theorem \ref{criteriascaffold} (2), and then the converse holds.
\end{proof}

\begin{rmk}\normalfont Actually, in \cite[Section 3]{elder}, Elder interprets \cite[Theorem 3.1]{byottchildselder} in the case of his scaffold, giving rise to \cite[Corollary 3.6]{elder}. Since he works with arbitrary ideals $\mathfrak{p}_L^h$, he imposes the stronger condition $\mathfrak{c}\geq2p-1$, which leads to the inequality $\frac{t}{2}<\frac{pe}{p-1}-2$. This is the reason why our Theorem \ref{maintheorem} (3) is a slightly stronger form of \cite[Corollary 3.6 (2)]{elder} for the case of degree $p$ extensions with dihedral normal closure.
\end{rmk}

\begin{example}\normalfont For the extension in Example \ref{examplep3}, we have that $a=2$ divides $p-1=2$, so $\mathcal{O}_L$ is $\mathfrak{A}_{L/K}$-free. However, in the ones corresponding to Example \ref{examplep>3}, $a=\frac{p+1}{2}$ does not divide $p-1$. We will eventually prove that for these ones $\mathcal{O}_L$ is $\mathfrak{A}_{L/K}$-free as well.
\end{example}

\section{The case $t\geq\frac{2pe}{p-1}-2$}\label{secthighram}

In this section, we work with a degree $p$ extension $L/K$ of $p$-adic fields with dihedral normal closure $\widetilde{L}$ such that the ramification jump $t$ of $L/K$ satisfies $$\frac{2pe}{p-1}-2\leq t\leq\frac{2pe}{p-1}.$$ In this part we will prove Theorem \ref{maintheorem} (4), and to do so we will adapt the techniques used in \cite{bertrandiasbertrandiasferton} for the cyclic case.

In order to study the $\mathfrak{A}_{L/K}$-freeness of $\mathcal{O}_L$, we will follow the method described at the end of Section \ref{secthnormalbases}. That is, for $\theta=\pi_L^a$, let $\alpha\in\mathfrak{A}_{\theta}$ and consider the map $\psi_{\alpha}\colon\mathfrak{A}_{L/K}\longrightarrow\mathfrak{A}_{\theta}$ defined by $\psi_{\alpha}(h)=h\alpha$. Let $M(\alpha)$ be the matrix of $\psi_{\alpha}$ as $\mathcal{O}_K$-linear map, where in $\mathfrak{A}_{L/K}$ we fix the $\mathcal{O}_K$-basis $\{\pi_K^{-n_i}w^i\}_{i=0}^{p-1}$ from Proposition \ref{probasisassocorder} and in $\mathfrak{A}_{\theta}$ we fix the $\mathcal{O}_K$-basis $\{\pi_K^{-\nu_k}w^k\}_{k=0}^{p-1}$ from Proposition \ref{propintbasisL}. Let us write $\alpha=\sum_{k=0}^{p-1}x_k\pi_K^{-\nu_k}w^k$. Then, $$M(\alpha)=\sum_{k=0}^{p-1}x_kM(\pi_K^{-\nu_k}w^k)$$ and denote $M(\pi_K^{-\nu_k}w^k)=(\mu_{j,i}^{(k)})_{j,i=0}^{p-1}$ for every $0\leq k\leq p-1$. Note that $\mu_{i,j}^{(k)}\in\mathcal{O}_K$ for every $0\leq j,i,k\leq p-1$ by definition, as they are the coordinates of $\psi_{\pi_K^{-\nu_k}w^k}(\pi_K^{-n_i}w^i)\in\mathfrak{A}_{\theta}$ with respect to the $\mathcal{O}_K$-basis $\{\pi_K^{-\nu_j}w^j\}_{j=0}^{p-1}$ of $\mathfrak{A}_{\theta}$. For any element $\mu\in\mathcal{O}_K$, we denote by $\overline{\mu}$ the projection of $\mu$ in $\mathcal{O}_K/\mathfrak{p}_K$.

\subsection{The continued fraction expansion of $\frac{\ell}{p}$}\label{sectsetE}

In order to determine the $\mathfrak{A}_{L/K}$-freeness of $\mathcal{O}_L$, we need to study the coefficients $\overline{\mu}_{j,i}^{(k)}$. As in the cyclic case, this is closely related with the continued fraction expansion of the number $\frac{\ell}{p}$. Let us acquire the notation and considerations in Section \ref{sectcontfrac}. In particular, from now on we work with the continued fraction expansion of $\frac{a}{p}$. Note that the notions and results in \cite[\textsection 1]{bertrandiasbertrandiasferton} only depend on the continued fraction expansion of a number of the form $\frac{a'}{p}$, where $0\leq a'\leq p-1$. Therefore, most of them can be translated to our case with few or no changes.

Let us denote $$E=\Big\{h\in\mathbb{Z}\,\Big|\,1\leq h< p,\,1\leq h'<h\,\Longrightarrow\,\widehat{h'\frac{a}{p}}>\widehat{h\frac{a}{p}}\Big\}.$$ The relevance of the set $E$ for our purposes relies in the following fact: the numbers $n_i$ admit descriptions depending on whether $p-i\in E$. It is the analogous result to \cite[Proposition 1]{bertrandiasbertrandiasferton} and \cite[Chapitre II, Proposition 5]{fertonthesis} for the cyclic case.

\begin{lema}\label{lemanuiE} Let $0\leq i\leq p-1$. We have that $n_i=ia_0+\Big\lfloor i\frac{a}{p}\Big\rfloor+\epsilon$, where $\epsilon=0$ if $p-i\notin E$ and $\epsilon=1$ if $p-i\in E$.
\end{lema}
\begin{proof}
We know that $$n_i=ia_0+\mathrm{min}_{0\leq j\leq p-1-i}\Big(\Big\lfloor(i+j+1)\frac{a}{p}\Big\rfloor-\Big\lfloor(j+1)\frac{a}{p}\Big\rfloor\Big).$$ We have that
\begin{align*}
    (i+j+1)\frac{a}{p}-(j+1)\frac{a}{p}&=\Big\lfloor(i+j+1)\frac{a}{p}\Big\rfloor-\Big\lfloor(j+1)\frac{a}{p}\Big\rfloor+\widehat{(i+j+1)\frac{a}{p}}-\widehat{(j+1)\frac{a}{p}}, \\
    (i+j+1)\frac{a}{p}-(j+1)\frac{a}{p}&=\Big\lfloor i\frac{a}{p}\Big\rfloor+\widehat{i\frac{a}{p}}.
\end{align*} Hence, $$\Big\lfloor(i+j+1)\frac{a}{p}\Big\rfloor-\Big\lfloor(j+1)\frac{a}{p}\Big\rfloor=\Big\lfloor i\frac{a}{p}\Big\rfloor+\widehat{i\frac{a}{p}}+\widehat{(j+1)\frac{a}{p}}-\widehat{(i+j+1)\frac{a}{p}}.$$ From the right hand side being an integer number, we deduce that $$\Big\lfloor(i+j+1)\frac{a}{p}\Big\rfloor-\Big\lfloor(j+1)\frac{a}{p}\Big\rfloor=\begin{cases}
\lfloor i\frac{a}{p}\rfloor+1 & \hbox{if }\widehat{i\frac{a}{p}}+\widehat{(j+1)\frac{a}{p}}\geq1, \\
\lfloor i\frac{a}{p}\rfloor & \hbox{if }\widehat{i\frac{a}{p}}+\widehat{(j+1)\frac{a}{p}}<1. \\
\end{cases}$$ Note that $\widehat{i\frac{a}{p}}+\widehat{(j+1)\frac{a}{p}}\geq1$ if and only if $\widehat{(j+1)\frac{a}{p}}\geq\widehat{(p-i)\frac{a}{p}}$, and then the result follows from the definition of $E$.
\end{proof}

\begin{rmk}\normalfont We know that $\nu_i\geq n_i$ for every $0\leq i\leq p-1$. Since $\nu_i=ia_0+\Big\lfloor(i+1)\frac{a}{p}\Big\rfloor$, Lemma \ref{lemanuiE} gives that $\nu_i=n_i$ or $\nu_i=n_i+1$. If $p-i\in E$, using the expression of $n_i$ in this result and \eqref{defaltnui} shows that $\nu_i\leq n_i$, so $\nu_i=n_i$. If $p-i\notin E$, both situations are possible.
\end{rmk}

On the other hand, the elements of the set $E$ can be described in terms of the continued fraction expansion of $\frac{\ell}{p}$. Namely, we can think of $E$ as a collection of the denominators $q_{2i,r_{2i}}$ of the even semiconvergents of $\frac{\ell}{p}$ (see Section \ref{sectcontfrac}).

\begin{pro}\label{lemaparamE} If $n\geq2$,
\begin{equation*}
    \begin{split}
        E=\Big\{r_{2i}q_{2i+1}+q_{2i}\,\Big|\,0\leq i<\frac{n-1}{2},\,r_{2i}\in\mathbb{Z},&\,0\leq r_{2i}<a_{2i+2}\hbox{ if }i\neq\frac{n-3}{2},\\&\,0\leq r_{2i}\leq a_{2i+2}\hbox{ if }i=\frac{n-3}{2}\Big\}.
    \end{split}
\end{equation*}
\end{pro}

Let us analyze what this statement means. Recall that the denominators of the even semiconvergents of $\frac{\ell}{p}$ are $$q_{2i,r_{2i}}=r_{2i}q_{2i+1}+q_{2i},\quad0\leq r_{2i}\leq a_{2i+2},$$ with $q_{2i,0}=q_{2i}$ and $q_{2i,a_{2i+2}}=q_{2i+2}$. From now on, in order to ease the notation, we denote $r=r_{2i}$ (and keep in mind that $r$ always depends on $i$). We can think of this description of the set $E$ as a list of the denominators $q_{2i,r}$ as follows. Let us take $0\leq i<\frac{n-1}{2}=\lfloor\frac{n}{2}\rfloor$. Starting from $q_{2i}$, we add successively the next denominator $q_{2i+1}$, until we reach $(a_{2i+2}-1)q_{2i+1}+q_{2i}$. All these numbers are $q_{2i,r}$, $0\leq r\leq a_{2i+2}-1$, and they belong to $E$. If $i<\frac{n-3}{2}$, the next even denominator $q_{2i+2}$ is covered by the next value $i+1$ and it has already been listed. Otherwise, if $i=\Big\lfloor\frac{n-2}{2}\Big\rfloor$, $q_{2i+2}\notin E$ unless $n$ is odd, in which case $i=\frac{n-3}{2}$ and $q_{2i+2}=q_{n-1}$ is the maximum of $E$. The set $E$ for the lowest values of $n$ is: \begin{align*}
    &\{q_0,q_1+q_0,2q_1+q_0,\dots,(a_2-1)q_1+q_0\} && \hbox{if }n=2, \\
    &\{q_0,q_1+q_0,2q_1+q_0,\dots,(a_2-1)q_1+q_0,q_2\} && \hbox{if }n=3, \\
    &\{q_0,q_1+q_0,2q_1+q_0,\dots,(a_2-1)q_1+q_0,q_2,q_3+q_2,2q_3+q_2,\dots,(a_4-1)q_3+q_2\} && \hbox{if }n=4, \\
    &\{q_0,q_1+q_0,2q_1+q_0,\dots,(a_2-1)q_1+q_0,q_2,q_3+q_2,2q_3+q_2,\dots,(a_4-1)q_3+q_2,q_4\} && \hbox{if }n=5.
\end{align*}

The remainder of this section is devoted to proving Proposition \ref{lemaparamE}. It is stated without proof in \cite[Lemme 1]{bertrandiasbertrandiasferton}. A detailed proof can be found in \cite[Chapitre II, Lemme 4]{fertonthesis}, which we present here because it contains some useful ideas that will be used later on. First, let us introduce the following notations.

\begin{nota}\label{circleC} Let us fix a circle $C$ of length one (i.e with radius $\frac{1}{2\pi}$) in which we fix an origin $O$ over the $X$-axis.
\begin{itemize}
    \item[1.] To each integer $h\in\mathbb{Z}$ we assign a point $M_h$ of $C$ in such a way that $\widehat{h\frac{a}{p}}$ is the length of the arc from $O$ to $M_h$ in anticlockwise direction (see Figure \ref{fig1}). 
    \item[2.] The length of the minimal arc between two of these points $M_h$ and $M_k$ is given by \begin{equation}\label{distanceint}
    d(h,k)=\Big|\Big|(h-k)\frac{a}{p}\Big|\Big|
    \end{equation} (see Figure \ref{fig2}). This notion will be referred to as the distance within $C$ between two integer numbers henceforth.
    \item[3.] When a point $M_h$ of $O$ lies in the arc (in anticlockwise direction) between other two points $M_{h'}$, $M_{h''}$ with $\widehat{h'\frac{a}{p}}<\widehat{h''\frac{a}{p}}$, we will say that $M_h$ is between $M_{h'}$ and $M_{h''}$. This condition means that $\widehat{h'\frac{a}{p}}<\widehat{h\frac{a}{p}}<\widehat{h''\frac{a}{p}}$ (see Figure \ref{fig3}).
\end{itemize}
\end{nota}

\begin{figure}[h!]
\begin{minipage}{0.3 \textwidth}
    \centering
    \begin{tikzpicture}
        \draw[opacity=0.7] (-2,0) -- (2,0);
        \draw[opacity=0.7] (0,-2) -- (0,2);
        \draw (0,0) circle (1.5);

        \draw[line width=0.5mm] (1.5,0) arc (0:60:1.5) node at (30:2) {$\widehat{h\frac{a}{p}}$};
        
        \fill (1.5,0) circle (0.1) node [below right] {$O$};
        \fill (60:1.5) circle (0.1) node [above right] {$M_h$};

    \end{tikzpicture}
    \caption{}
    \label{fig1}
\end{minipage}
\begin{minipage}{0.3 \textwidth}
    \centering
    \begin{tikzpicture}
        \draw[opacity=0.7] (-2,0) -- (2,0);
        \draw[opacity=0.7] (0,-2) -- (0,2);
        \draw (0,0) circle (1.5);

        \draw[line width=0.5mm] (60:1.5) arc (60:160:1.5) node at (110:2) {$d(h,k)$};

        \fill (1.5,0) circle (0.1) node [below right] {$O$};
        \fill (60:1.5) circle (0.1) node [above right] {$M_h$};
        \fill (160:1.5) circle (0.1) node [above left] {$M_k$};
    
    \end{tikzpicture}
    \caption{}
    \label{fig2}
\end{minipage}
\begin{minipage}{0.3 \textwidth}
    \centering
    \begin{tikzpicture}
        \draw[opacity=0.7] (-2,0) -- (2,0);
        \draw[opacity=0.7] (0,-2) -- (0,2);
        \draw (0,0) circle (1.5);

        \draw[line width=0.5mm] (50:1.5) arc (50:190:1.5);

        \fill (1.5,0) circle (0.1) node [below right] {$O$};
        \fill (50:1.5) circle (0.1) node [above right] {$M_{h'}$};
        \fill (120:1.5) circle (0.1) node [above left] {$M_h$};
        \fill (190:1.5) circle (0.1) node [left] {$M_{h''}$};
        
    \end{tikzpicture}
    \caption{}
    \label{fig3}
\end{minipage}

\end{figure}

From now on and unless stated otherwise, all the arcs within $C$ will be taken as follows:
\begin{itemize}
    \item From the origin $O$ in anticlockwise direction, if $O$ is one of the nodes in the arc.
    \item Not containing the origin $O$, if none of the nodes in the arc is $O$.
\end{itemize}

This convention might not be followed in the notion of distance within $C$: in that case, we take the arc of minimal length between two points, regardless of whether it contains the origin or not.

\begin{rmk}\normalfont\label{rmkminarc} If two integer numbers $h,k$ are such that $0<|h-k|<q_i$ for some $i$, Proposition \ref{propcontfrac} (2) gives that $d(h,k)\geq||q_{i-1}\frac{a}{p}||$. Geometrically, this means that the minimal arc within $C$ between $M_h$ and $M_k$ has length $||q_{i-1}\frac{a}{p}||$.
\end{rmk}

Since $q_{2i}$ increases with $i$ and they all belong to $E$, $\widehat{q_{2i}\frac{a}{p}}$ decreases as $i$ increases. In Figure \ref{fig4} we represent this situation in the circle $C$ for $a\geq\frac{p+1}{2}$ and $n$ odd and large enough.

\begin{figure}[h!]
    \centering
    \begin{tikzpicture}
        \draw[opacity=0.7] (-2,0) -- (2,0);
        \draw[opacity=0.7] (0,-2) -- (0,2);
        \draw (0,0) circle (1.5);

        \draw[line width=0.5mm] (1.5,0) arc (0:27:1.5);
        \draw[line width=0.5mm] (20:1.5) arc (20:54:1.5);
        \draw[line width=0.5mm] (54:1.5) arc (54:135:1.5);
        \draw[line width=0.5mm] (135:1.5) arc (135:160:1.5);
        \draw[line width=0.5mm] (160:1.5) arc (160:190:1.5);
        
        \fill (1.5,0) circle (0.1) node [below right] {$O$};
        \fill (20:1.5) circle (0.1) node [above right] {$M_{q_{n-1}}$};
        \fill (54:1.5) circle (0.1) node [above right] {$M_{q_{n-3}}$};
        \fill (135:1.5) circle (0.1);
        \fill (135:1.6) node [above left] {$M_{q_{4}}$};
        \fill (160:1.5) circle (0.1) node [above left] {$M_{q_2}$};
        \fill (190:1.5) circle (0.1) node [left] {$M_{q_0}$};
    
    \end{tikzpicture}
    \caption{}
    \label{fig4}
\end{figure}

\begin{rmk}\normalfont Note that $a\geq\frac{p+1}{2}$ if $a_0=0$ (see the proof of Proposition \ref{lemaparam} (5)), so in that case the point $M_{q_0}$ is in the lower half semiplane. However, Proposition \ref{propcontfrac} (3) gives that $\widehat{q_{2i}\frac{a}{p}}<\frac{1}{2}$ for $i>0$, so the corresponding points $M_{q_{2i}}$ are in the upper half semiplane.
\end{rmk}

Actually, the arc between two consecutive points $M_{q_{2i}}$, $M_{q_{2i+2}}$ contains all the points $M_{q_{2i+2,r}}$ corresponding to semiconvergents. This leads to the following technical but important result.

\begin{lema}\label{lemadensity} Let $h\in\mathbb{Z}$, $1\leq h<p$. If $M_h$ is between $M_{q_{2i+2}}$ and $M_{q_{2i}}$ for some $0\leq i<\frac{n-1}{2}$, there is $0\leq r\leq a_{2i+2}$ such that $d(h,q_{2i,r})<||q_{2i+1}\frac{a}{p}||$. If in addition $|h-q_{2i,r}|<q_{2i+2}$, then $h=q_{2i,r}$.
\end{lema}
\begin{proof}
Given $r\geq1$, $$q_{2i,r}-q_{2i,r-1}=q_{2i+1}\quad\Longrightarrow\quad\Big|\Big|(q_{2i,r}-q_{2i,r-1})\frac{a}{p}\Big|\Big|=\Big|\Big|q_{2i+1}\frac{a}{p}\Big|\Big|.$$ Moreover, $\widehat{q_{2i,r}\frac{a}{p}}=\widehat{q_{2i+1}\frac{a}{p}}+\widehat{q_{2i,r-1}\frac{a}{p}}-1<\widehat{q_{2i,r-1}\frac{a}{p}}$ and $\widehat{q_{2i,r}\frac{a}{p}}>\widehat{q_{2i+2}\frac{a}{p}}$. By recurrence, we deduce that $$\widehat{q_{2i+2}\frac{a}{p}}<\widehat{q_{2i,a_{2i+2}-1}\frac{a}{p}}<\dots<\widehat{q_{2i,1}\frac{a}{p}}<\widehat{q_{2i}\frac{a}{p}},$$ and the difference between any two consecutive numbers is $||q_{2i+1}\frac{a}{p}||$ (see Figure \ref{fig5}). 

\begin{figure}[h!]
    \centering
    \begin{tikzpicture}
        \draw (30:7) arc (30:90:7);
        
        \draw[line width=0.5mm] (30:7) arc (30:42:7);
        \draw[line width=0.5mm] (66:7) arc (66:78:7);
        \draw[line width=0.5mm] (78:7) arc (78:90:7);
        
        \fill (30:7) circle (0.1) node [above right] {$M_{q_{2i+2}}$};
        \fill (42:6.3) node [below] {$||q_{2i+1}\frac{a}{p}||$};
        \fill (42:7) circle (0.1) node [above right] {$M_{q_{2i,a_{2i+2}-1}}$};
        \fill (66:7) circle (0.1) node [above right] {$M_{q_{2i,2}}$};
        \fill (76:6.7) node [below] {$||q_{2i+1}\frac{a}{p}||$};
        \fill (78:7) circle (0.1) node [above right] {$M_{q_{2i,1}}$};
        \fill (90:6.7) node [below] {$||q_{2i+1}\frac{a}{p}||$};
        \fill (90:7) circle (0.1) node [above right] {$M_{q_{2i}}$};
        
    \end{tikzpicture}
    \caption{}
    \label{fig5}
\end{figure}

Thus, given $h\in\mathbb{Z}$ as in the statement, $\widehat{h\frac{a}{p}}$ is between two consecutive numbers in the chain of inequalities above, so the distance within $C$ of $h$ and any of the two integers corresponding to these consecutive numbers is lower than $||q_{2i+1}\frac{a}{p}||$, proving the existence of $r$ as in the statement. The last sentence follows from Remark \ref{rmkminarc}.
\end{proof}

\begin{coro}\label{corodensity} If $h\leq q_{2i+2}$ and $M_h$ is between $M_{q_{2i+2}}$ and $M_{q_{2i}}$ for some $0\leq i<\frac{n-1}{2}$, there is $0\leq r\leq a_{2i+2}$ such that $h=q_{2i,r}$.
\end{coro}
\begin{proof}
By Lemma \ref{lemadensity}, there is $0\leq r\leq a_{2i+2}$ such that $d(h,q_{2i,r})<||q_{2i+1}\frac{a}{p}||$. Now, since $0<q_{2i,r}\leq q_{2i+2}$, $h\leq q_{2i+2}$ implies that $|h-q_{2i,r}|<q_{2i+2}$, so we have that $h=q_{2i,r}$.
\end{proof}

\begin{rmk}\normalfont\label{densityrefinement} The argument in the proof of Lemma \ref{lemadensity} can be used to prove the following refinement: If $M_h$ is between $M_{q_{2i,r}}$ and $M_{q_{2i+2}}$ for some $0\leq i<\frac{n-1}{2}$ and some $0\leq r\leq a_{2i+2}$, there is $r\leq r'\leq a_{2i+2}$ such that $d(h,q_{2i,r'})<||q_{2i+1}\frac{a}{p}||$. We can state and proof a similar refinement of Corollary \ref{corodensity}.
\end{rmk}

\begin{proof} \textit{(of Proposition \ref{lemaparamE})}
First of all, by definition of $E$ we have that $q_0=1\in E$, and Proposition \ref{propcontfrac} (2),(3) give that $q_{2i}\in E$ for all $i$ such that $2i<n$, that is, $i<\frac{n}{2}$. If $n$ is odd, this says that $q_0,q_2,\dots,q_{n-3},q_{n-1}\in E$, where the last value corresponds to the choice $i=\frac{n-3}{2}$ and $r=a_{2i+2}$. If $n$ is even, we have that $q_0,q_2,\dots,q_{n-4},q_{n-2}\in E$. In this case, the last value corresponds to $i=\frac{n-2}{2}$ and $r=0$, while for $r=a_{2i+2}$ we have $a_{2i+2}q_{2i+1}+q_{2i}=p\notin E$.

Let us prove that the numbers $q_{2i,r}$ belong to $E$. For fixed $i$ and $r$, call $h=q_{2i,r}$. Pick $1\leq h'<h$ and assume that $\widehat{h'\frac{a}{p}}<\widehat{h\frac{a}{p}}$. Since $h'<q_{2i+2}$ and $q_{2i+2}\in E$, we have that $\widehat{q_{2i+2}\frac{a}{p}}<\widehat{h'\frac{a}{p}}$, so $M_{h'}$ is between $M_{q_{2i+2}}$ and $M_h$. By Remark \ref{densityrefinement}, $h'=q_{2i,r'}$ for some $r< r'<a_{2i+2}$. But then $r'>r$ implies that $h'>h$, a contradiction. Hence $\widehat{h'\frac{a}{p}}\geq\widehat{h\frac{a}{p}}$, and since $1\leq h'<h<p$, the equality is not possible.

Finally, we prove that there are no more integer numbers in $E$. Given $h\in E$, we have that $q_0=1\leq h<q_n=p$, so obviously $q_{2i}\leq h\leq q_{2i+2}$ for some $i$. Now, using the definition of $E$, we find that $M_h$ is between $M_{q_{2i+2}}$ and $M_{q_{2i}}$. Then Corollary \ref{corodensity} gives the existence of some $0\leq r\leq a_{2i+2}$ such that $h=q_{2i,r}$.
\end{proof}

To sum up, we have proved that the numbers $n_i$, which are related with the structure of $\mathcal{O}_L$ as $\mathfrak{A}_{L/K}$-module, can be formulated depending on the set $E$ and that this set admits a parametrization in terms of the numbers $q_{2i,r}$, coming from the continued fraction expansion of $\frac{\ell}{p}$. This is the origin of the link between these two phenomena in the statement of Theorem \ref{maintheorem} (4).

\subsection{The matrices $M(\pi_K^{-\nu_k}w^k)$}\label{sectmatrices}

We turn to our problem of determining the projections $\overline{\mu}_{j,i}^{(k)}$ of the coefficients of the matrices $M(\pi_K^{-\nu_k}w^k)$. By definition, we have that \begin{equation*}
    \pi_K^{-\nu_k-n_i}w^{k+i}=\sum_{j=0}^{p-1}\mu_{j,i}^{(k)}\pi_K^{-\nu_j}w^j
\end{equation*} Using the uniqueness of coordinates with respect to the basis $\{w^l\}_{l=0}^{p-1}$, we see that \begin{equation}\label{coordbasis}
    \mu_{j,i}^{(k)}=C(w^j,w^{k+i})\pi_K^{\nu_j-\nu_k-n_i},
\end{equation} where $C(w^j,w^{k+i})$ is the coefficient of $w^j$ in the expression of $w^{k+i}$ with respect to that basis. Note that if such a coefficient is $0$, then $\mu_{j,i}^{(k)}=0$.

The easier case $k+i\leq p-1$ depends only on the fraction $\frac{a}{p}$, and therefore is as in \cite[Lemme 2.a]{bertrandiasbertrandiasferton}.

\begin{pro}\label{proentrieseasy} Assume that $k+i\leq p-1$ and call $h=p-i$.
\begin{itemize}
    \item[(1)] If $j\neq k+i$, then $\overline{\mu}_{j,i}^{(k)}=0$.
    \item[(2)] If $h\notin E$, $$\overline{\mu}_{k+i,i}^{(k)}=\begin{cases}
    1 & \hbox{if }h=p\hbox{ or }\widehat{(k+1)\frac{a}{p}}<\widehat{h\frac{a}{p}}\\
    0 & \hbox{otherwise}\end{cases}$$
    \item[(3)] If $h\in E$, $\overline{\mu}_{k+i,i}^{(k)}=1$.
\end{itemize}
\end{pro}
\begin{proof}
The first statement is immediate from the fact that $C(w^j,w^{k+i})=0$ for $j\neq k+i$, $0\leq j<p$. Let us prove (2) and (3). For $j=k+i$, we have that $C(w^j,w^{k+i})=1$, so $\mu_{k+i,i}^{(k)}=\pi_K^{\nu_{k+i}-\nu_k-n_i}$. By Lemma \ref{lemanuiE}, we have that $$\nu_k+n_i=(k+i)a_0+\Big\lfloor i\frac{a}{p}\Big\rfloor+\Big\lfloor(k+1)\frac{a}{p}\Big\rfloor+\epsilon,$$ where $\epsilon=1$ if $h\in E$ and $\epsilon=0$ otherwise. Now, we have that $$\Big\lfloor(i+k+1)\frac{a}{p}\Big\rfloor-\Big\lfloor i\frac{a}{p}\Big\rfloor-\Big\lfloor(k+1)\frac{a}{p}\Big\rfloor=\widehat{i\frac{a}{p}}+\widehat{(k+1)\frac{a}{p}}-\widehat{(i+k+1)\frac{a}{p}},$$ so $$\nu_{k+i}-\nu_k-n_i=\widehat{i\frac{a}{p}}+\widehat{(k+1)\frac{a}{p}}-\widehat{(i+k+1)\frac{a}{p}}-\epsilon.$$ 

Assume that $h\notin E$, so that $$\nu_{k+i}-\nu_k-n_i=\widehat{i\frac{a}{p}}+\widehat{(k+1)\frac{a}{p}}-\widehat{(i+k+1)\frac{a}{p}}.$$ If $h=p$, we have $\nu_{k+i}=\nu_k+n_i$, so $\overline{\mu}_{k+i,i}^{(k)}=1$. Otherwise, we have the equality $\nu_{k+i}=\nu_k+n_i$ if and only if $\widehat{(k+1)\frac{a}{p}}+\widehat{i\frac{a}{p}}<1$, that is, $\widehat{(k+1)\frac{a}{p}}<\widehat{h\frac{a}{p}}$.

Assume that $h\in E$. In this case, we have $$\nu_{k+i}-\nu_k-n_i=\widehat{i\frac{a}{p}}+\widehat{(k+1)\frac{a}{p}}-\widehat{(i+k+1)\frac{a}{p}}-1.$$ It is clear that $\widehat{(k+1)\frac{a}{p}}+\widehat{i\frac{a}{p}}<2$. Then $\nu_{k+i}-\nu_k-n_i\leq0$, and it is also non-negative because $\mu_{k+i,i}^{(k)}\in\mathcal{O}_K$. Hence $\nu_{k+i}-\nu_k-n_i=0$ and $\overline{\mu}_{k+i,i}^{(k)}=1$.
\end{proof}

If $k+i>p-1$, there are some differences with respect to the corresponding result \cite[Lemme 2.b]{bertrandiasbertrandiasferton} for the cyclic case. For each $0\leq m\leq p-1$, call $$d_m=\mathrm{max}\Big\{d\in\{0,1,\dots,p-1-m\}\,|\,d\hbox{ even, }\nu_{m+d}=\nu_m+\frac{d}{2}\Big\}.$$ Note that $0$ satisfies the property of the defining set, so $d_m$ is well defined.

\begin{pro}\label{proentries} Assume that $k+i>p-1$, call $h=p-i$ and let $m=k+i-(p-1)$.
\begin{itemize}
    \item[(1)] If $j\not\equiv m\,(\mathrm{mod}\,2)$, then $\overline{\mu}_{j,i}^{(k)}=0$.
\end{itemize}
Otherwise, if $j\equiv m\,(\mathrm{mod}\,2)$, we have:
\begin{itemize}
    \item[(2)] If $m<p-1$ and $j<m$ or $j>m+d_m$, then $\overline{\mu}_{j,i}^{(k)}=0$.
    \item[(3)] If $h\notin E$ and $m\leq j\leq m+d_m$, $\overline{\mu}_{j,i}^{(k)}\neq0$ if $\frac{a}{p}+\widehat{(k+1)\frac{a}{p}}<\widehat{h\frac{a}{p}}$, and $\overline{\mu}_{j,i}^{(k)}=0$ otherwise.
    \item[(4)] If $h\in E$ and $m\leq j\leq m+d_m$, $\overline{\mu}_{j,i}^{(k)}\neq0$ if $\frac{a}{p}+\widehat{(k+1)\frac{a}{p}}<\widehat{h\frac{a}{p}}+1$, and $\overline{\mu}_{j,i}^{(k)}=0$ otherwise.
\end{itemize}
\end{pro}
\begin{proof}
We have that $w^{k+i}=w^{p-1+m}$, and the expression of this element with respect to the basis $\{w^j\}_{j=0}^{p-1}$ is known from Proposition \ref{w-highpoly}. Since the coordinates therein are determined up to multiplication by elements $r_i^{(j)}$ of $\mathcal{O}_K$, we look for the value of $v_K(\mu_{j,i}^{(k)})$. Taking valuations in \eqref{coordbasis} gives $$v_K(\mu_{j,i}^{(k)})=v_K(C(w^j,w^{k+i}))+\nu_j-\nu_k-n_i.$$ The first statement is immediate: if $j\not\equiv m\,(\mathrm{mod}\,2)$, then $C(w^j,w^{k+i})=0$, so $v_K(C(w^j,w^{k+i}))=\infty$ and $\mu_{j,i}^{(k)}=0$. 

From now on, we assume $j\equiv m\,(\mathrm{mod}\,2)$. Let us prove (2). We assume that $m<p-1$ and $j<m$. Then, we have that $$v_K(\mu_{j,i}^{(k)})\geq2e+\frac{p-1}{2}+\frac{m-j}{2}+\nu_j-\nu_k-n_i.$$ Note that since $\pi_K^{-\nu_k-n_i}w^{k+i}\in\mathfrak{A}_{\theta}$, Lemma \ref{lemasubr} gives that $\nu_k+n_i\leq e+\frac{p-1}{2}+\nu_m$. Then $$v_K(\mu_{j,i}^{(k)})\geq e+\frac{m-j}{2}+\nu_j-\nu_m.$$ Now, since $m<p-1$, Proposition \ref{lemaparam} (7) gives that $\nu_m<e+\Big\lceil\frac{m}{2}\Big\rceil$, and since $\frac{m-j}{2}=\Big\lceil\frac{m}{2}\Big\rceil-\Big\lceil\frac{j}{2}\Big\rceil$, we have that $v_K(\mu_{j,i}^{(k)})>\nu_j-\Big\lceil\frac{j}{2}\Big\rceil$. We claim that this last quantity is always non-negative. Indeed, if $a_0>0$ the claim is trivial, and if $a_0=0$ then $a>\frac{p}{2}$, so $\nu_j=\Big\lfloor(j+1)\frac{a}{p}\Big\rfloor\geq\Big\lfloor\frac{j+1}{2}\Big\rfloor=\Big\lceil\frac{j}{2}\Big\rceil$. This proves that $\overline{\mu}_{j,i}^{(k)}=0$. On the other hand, if $j>m+d_m$ and $j\equiv m\,(\mathrm{mod}\,2)$, we have that $\nu_j>\nu_m+\frac{j-m}{2}$. Now, $$v_K(\mu_{j,i}^{(k)})=e+\frac{p-1}{2}-\frac{j-m}{2}+\nu_j-\nu_k-n_i>e+\frac{p-1}{2}+\nu_m-\nu_k-n_i,$$ and this is non-negative again by Lemma \ref{lemasubr}. Therefore, $v_K(\mu_{j,i}^{(k)})>0$ and $\overline{\mu}_{j,i}^{(k)}=0$.

Let us focus on the case $m\leq j\leq m+d_m$, so that $\nu_j=\nu_m+\frac{j-m}{2}$. By \eqref{coordbasis} and Lemma \ref{lemanuiE}, we have that \begin{equation*}
    \begin{split}
        v_K(\mu_{j,i}^{(k)})&=e+\frac{p-1}{2}-\frac{j-m}{2}+\nu_j-\nu_k-n_i\\&=e+\frac{p-1}{2}+\nu_m-\nu_k-n_i\\&=e+\frac{p-1}{2}+ma_0+\Big\lfloor(m+1)\frac{a}{p}\Big\rfloor-(k+i)a_0-\Big\lfloor i\frac{a}{p}\Big\rfloor-\Big\lfloor(k+1)\frac{a}{p}\Big\rfloor-\epsilon,
    \end{split}
\end{equation*} where $\epsilon=1$ if $h\in E$ and $\epsilon=0$ if $h\notin E$. Since $ma_0=(k+i)a_0-(p-1)a_0$ and $$\Big\lfloor(m+1)\frac{a}{p}\Big\rfloor-\Big\lfloor i\frac{a}{p}\Big\rfloor-\Big\lfloor(k+1)\frac{a}{p}\Big\rfloor=-(p-1)\frac{a}{p}-\widehat{(m+1)\frac{a}{p}}+\widehat{i\frac{a}{p}}+\widehat{(k+1)\frac{a}{p}},$$ we have that $$v_K(\mu_{m,i}^{(k)})=e+\frac{p-1}{2}-(p-1)a_0-(p-1)\frac{a}{p}-\widehat{(m+1)\frac{a}{p}}+\widehat{i\frac{a}{p}}+\widehat{(k+1)\frac{a}{p}}-\epsilon.$$ Next, we note that $a+(p-1)a_0=\nu_{p-1}=e+\frac{p-1}{2}$, so $$v_K(\mu_{m,i}^{(k)})=\frac{a}{p}-\widehat{(m+1)\frac{a}{p}}+\widehat{i\frac{a}{p}}+\widehat{(k+1)\frac{a}{p}}-\epsilon.$$ Replacing $i=p-h$, we find that $$v_K(\mu_{m,i}^{(k)})=\frac{a}{p}-\widehat{(m+1)\frac{a}{p}}+1-\widehat{h\frac{a}{p}}+\widehat{(k+1)\frac{a}{p}}-\epsilon.$$ Let $\gamma=\frac{a}{p}-\widehat{h\frac{a}{p}}+\widehat{(k+1)\frac{a}{p}}$. Since $\frac{a}{p}-h\frac{a}{p}+(k+1)\frac{a}{p}=(m+1)\frac{a}{p}$, we have $$\gamma=\begin{cases}
\widehat{(m+1)\frac{a}{p}} & \hbox{if }0\leq\gamma<1, \\
\widehat{(m+1)\frac{a}{p}}-1 & \hbox{if }-1\leq\gamma<0, \\
\widehat{(m+1)\frac{a}{p}}+1 & \hbox{if }1\leq\gamma<2, \\
\end{cases}$$ Therefore, $$v_K(\mu_{j,i}^{(k)})=\begin{cases}
1-\epsilon & \hbox{if }0\leq\gamma<1, \\
-\epsilon & \hbox{if }-1\leq\gamma<0, \\
2-\epsilon & \hbox{if }1\leq\gamma<2. \\
\end{cases}$$

Assume that $h\notin E$, so $\epsilon=0$. Then $v_K(\mu_{j,i}^{(k)})=0$ if and only if $-1\leq\gamma<0$. The first inequality is trivial, and the other one is equivalent to $\frac{a}{p}+\widehat{(k+1)\frac{a}{p}}<\widehat{h\frac{a}{p}}$, which proves (3). Otherwise, if $h\in E$, then $\epsilon=1$ and $v_K(\mu_{j,i}^{(k)})=0$ if and only if $0\leq\gamma<1$. Now it is the first inequality that is always true. Indeed, if $\widehat{h\frac{a}{p}}>\widehat{(k+1)\frac{a}{p}}$, since $h\in E$ we have by definition that $h\leq k+1$, which contradicts our hypothesis, so $\widehat{h\frac{a}{p}}\leq\widehat{(k+1)\frac{a}{p}}$. On the other hand, the second inequality may be rewritten as $\frac{a}{p}+\widehat{(k+1)\frac{a}{p}}<\widehat{h\frac{a}{p}}+1$, finishing the proof of (4).
\end{proof}

\begin{rmk}\normalfont If $k+i>p-1$, Propositions \ref{lemaparam} (6) and \ref{proentries} show that $\overline{\mu}_{p-1,i}^{(k)}=0$. This follows from (1) when $m=k+i-(p-1)$ is odd, and when it is even we have that $p-1>m+d_m$ because $\nu_{p-1}-\nu_{p-3}\geq2$, so it is implied by (2).
\end{rmk}

\begin{rmk}\normalfont If $m=p-1$, the behaviour of $\mu_{j,i}^{(k)}$ for $j<m$ and $j\equiv m\,(\mathrm{mod}\,2)$ is trickier. In that case, we have that $\nu_{p-1}=n_{p-1}=e+\frac{p-1}{2}$, and consequently $v_K(\mu_{j,i}^{(k)})\geq\nu_j-\frac{j}{2}$. This may vanish, for instance for $j=0$. In those cases, the vanishing of $\overline{\mu}_{j,i}^{(k)}$ depends on whether the coefficient $r_j^{(p-1)}\in\mathcal{O}_K$ of the expression of $w^{k+i}$ as in Proposition \ref{w-highpoly} is a unit or not.
\end{rmk}

\subsection{The $\mathfrak{A}_{L/K}$-freeness of $\mathcal{O}_L$}

Let $n$ be the length of the continued fraction expansion of $\frac{\ell}{p}$. In this part, we complete the proof of Theorem \ref{maintheorem} (4), that is, a necessary and sufficient condition for the $\mathfrak{A}_{L/K}$-freeness of $\mathcal{O}_L$ is that $n\leq 4$.

\subsubsection{Sufficiency}

We devote this section to proving the following implication.

\begin{pro}\label{statementsuff} If $n\leq 4$, then $\mathfrak{A}_{\theta}$ is $\mathfrak{A}_{L/K}$-principal for $\theta=\pi_L^a$.
\end{pro}

It is seen easily that $n\leq 2$ if and only if $a=0$ or $a\mid p-1$, in which cases we already know from Proposition \ref{propequalorders} that $\mathfrak{A}_{\theta}=\mathfrak{A}_{L/K}$. Thus, we assume that $n=3$ or $n=4$. From Proposition \ref{lemaparamE} we know that \begin{align*}
    &E=\{q_0,q_1+q_0,2q_1+q_0,\dots,(a_2-1)q_1+q_0,q_2\} && \hbox{if }n=3, \\
    &E=\{q_0,q_1+q_0,2q_1+q_0,\dots,(a_2-1)q_1+q_0,q_2,q_3+q_2,2q_3+q_2,\dots,(a_4-1)q_3+q_2\} && \hbox{if }n=4.
\end{align*}

Let us choose $k=q_2-1$. We will eventually prove that for some unit $u\in\mathcal{O}_K^*$, the element $$\alpha=u+\pi_K^{-\nu_k}w^k$$ works as an $\mathfrak{A}_{L/K}$-generator of $\mathfrak{A}_{\theta}$ (we make implicit the dependence of $\alpha$ on $u$). Here $u$ is multiplying the identity element of $H$, which corresponds to the case $k=0$ and which we denote by $1$. Consequently, with our notation, we have $M(1)=(\mu_{j,i}^{(0)})_{j,i=0}^{p-1}$. Now, $\alpha$ is a generator if and only if $\mathrm{det}(\overline{M(\alpha)})\neq0$. We have that $$M(\alpha)=uM(1)+M(\pi_K^{-\nu_k}w^k).$$ Clearly, the matrix $M(1)$ is diagonal and $\overline{\mu}_{ii}^{(0)}=1$ (resp. $0$) if $n_i=\nu_i$ (resp. $n_i<\nu_i$). Since $n>2$, we have that $\mathfrak{A}_{L/K}\neq\mathfrak{A}_{\theta}$, so there is some $0$ in the diagonal of $M(1)$. We proceed to study the matrix $M(\pi_K^{-\nu_k}w^k)$ for $k>0$ using Propositions \ref{proentrieseasy} and \ref{proentries}. As in the previous part, for $0\leq i\leq p-1$ we denote $h=p-i$.

\begin{lema}\label{lemaentrieseasy} Assume that $n\leq 4$. If $h\geq k+1$, then $\overline{\mu}_{k+i,i}^{(k)}=1$ and $\overline{\mu}_{j,i}^{(k)}=0$ for $j\neq k+i$.
\end{lema}
\begin{proof}
Following Proposition \ref{proentrieseasy}, it is enough to prove that if $h\notin E$ and $h\neq p$, then $\widehat{(k+1)\frac{a}{p}}<\widehat{h\frac{a}{p}}$. If $n=3$, Proposition \ref{propcontfrac} (4) gives that $\widehat{(k+1)\frac{a}{p}}=\widehat{q_2\frac{a}{p}}=\frac{1}{p}<\widehat{h\frac{a}{p}}$, since $p>h>q_2$. If $n=4$, in this case by Proposition \ref{propcontfrac} (4) we have that $||q_3\frac{a}{p}||=\frac{1}{p}$. If $\widehat{h\frac{a}{p}}<\widehat{q_2\frac{a}{p}}$, then the point $M_h$ in the circle $C$ corresponding to $h$ (see Note \ref{circleC}) is between $M_{q_4}$ and $M_{q_2}$. Since $h<q_4=p$, Lemma \ref{lemadensity} gives that $h=q_{2,r}$ for some $0\leq r<a_4$, in particular $h\in E$, which is a contradiction. Therefore $\widehat{(k+1)\frac{a}{p}}\leq\widehat{h\frac{a}{p}}$, and the equality is not possible because $h\notin E$ and $k+1=q_2\in E$.
\end{proof}

\begin{lema}\label{lemaentries} Suppose that $n\leq 4$. If $h<k+1$ and we call $m=k+1-h$, $\overline{\mu}_{m,i}^{(k)}\neq0$ and $\overline{\mu}_{j,i}^{(k)}=0$ for $j<m$.
\end{lema}
\begin{proof}
We note that since $n>2$, we have that $k=q_2-1<p-1$, so $m<p-1$, and then the second part follows from Proposition \ref{proentries} (2). Now, we prove the first part. Suppose that $h\in E$; if we prove that $\widehat{(k+1)\frac{a}{p}}+\frac{a}{p}<\widehat{h\frac{a}{p}}+1$ we will be done by Proposition \ref{proentries} (4). Since $h<q_2$ and $q_2\in E$, we have $\widehat{h\frac{a}{p}}>\widehat{q_2\frac{a}{p}}=\widehat{(k+1)\frac{a}{p}}$. Then $$\widehat{h\frac{a}{p}}+1>\widehat{(k+1)\frac{a}{p}}+\frac{a}{p}.$$ On the contrary, if $h\notin E$, by Proposition \ref{proentries} (3) it is enough to prove that $\widehat{h\frac{a}{p}}>\frac{a}{p}+\widehat{(k+1)\frac{a}{p}}$. If $\widehat{h\frac{a}{p}}<\frac{a}{p}$, then we have that $\widehat{q_2\frac{a}{p}}<\widehat{h\frac{a}{p}}<\frac{a}{p}$ if $n=3$, and if $n=4$ it is also possible that $\widehat{q_4\frac{a}{p}}<\widehat{h\frac{a}{p}}<\widehat{q_2\frac{a}{p}}$. Since $h<q_2<q_4$, we can apply Corollary \ref{corodensity} in both cases to obtain that $h\in E$, which is a contradiction. Hence we obtain that $\widehat{h\frac{a}{p}}>\frac{a}{p}$. Finally, by Proposition \ref{propcontfrac} (2), $h-1<q_2$ implies that $||h\frac{a}{p}-\frac{a}{p}||>\widehat{q_2\frac{a}{p}}$, and since $\widehat{h\frac{a}{p}}>\frac{a}{p}$, $||h\frac{a}{p}-\frac{a}{p}||=\widehat{h\frac{a}{p}}-\frac{a}{p}$.
\end{proof}

Finally, we study the determinant of $\mathrm{M(\alpha)}$ modulo $\mathfrak{p}_K$.

\begin{pro}\label{prosufcondfreeness} Let us assume that $n\leq 4$. Then there is some $u\in\mathcal{O}_K^*$ such that $\mathrm{det}(M(\alpha))\not\equiv0\,(\mathrm{mod}\,\mathfrak{p}_K)$.
\end{pro}
\begin{proof}
Note that $\overline{\mu}_{0,0}^{(0)}=1$ and $\overline{\mu}_{0,i}^{(k)}=0$ for every $i\geq0$ and $k=q_2-1$, by Propositions \ref{proentrieseasy} (1) and \ref{proentries} (1), (2). Then the first row of $\overline{M(\alpha)}$ is $(\overline{u},0,\dots,0)$, and developing the determinant by this row, we obtain $\mathrm{det}(\overline{M(\alpha)})=\overline{u}\mathrm{det}(M')$, where $M'\in\mathcal{M}_{p-1}(\mathcal{O}_K/\mathfrak{p}_K)$.

To ease the notation, we number the entries of $M'$ with subscripts starting in $1$. We will study $\mathrm{det}(M')$ as a sum of terms corresponding to permutations of $(1,\dots,p-1)$. Since $\overline{\mu}_{0,0}^{(0)}=1$ and the diagonal of $M(1)$ has some zero, so does the diagonal of $M'$, and the non-zero diagonal elements in $M'$ are $\overline{u}$. Recall that for the $i$-th column of $M'$, $1\leq i\leq p-1$, we write $h=p-i$ (and note that the dependence of $h$ on $i$ is implicit). Within the first $p-k-1$ columns of $M'$ (that is, $1\leq i\leq p-k-1$), we have $h\geq k+1$. Then, we know from Lemma \ref{lemaentrieseasy} that the terms above the main diagonal of $M'$ are zero, and that the only non-zero terms below are the entries of a subdiagonal of $1$'s in positions $(k+i,i)$, $1\leq i\leq p-k-1$. Next, we consider the other columns, for which $i\geq p-k$. If we call $m_i=k+i-(p-1)$, we have from Lemma \ref{lemaentries} that $\overline{\mu}_{m_i,i}^{(k)}\neq0$. The subscripts of all these non-zero terms correspond to (the inverse of) the permutation of $(1,\dots,p-1)$ that sends $i$ to $k+i$ if $1\leq i\leq p-k-1$ and to $m_i$ otherwise, so their product $$b_0=\prod_{i=1}^{p-k-1}\overline{\mu}_{k+i,i}^{(k)}\prod_{i=p-k}^{p-1}\overline{\mu}_{m_i,i}^{(k)}$$ is a non-zero summand of $\mathrm{det}(M')$. Since $\overline{\mu}_{j,i}^{(k)}=0$ for $j<m$ also by Lemma \ref{lemaentries}, $m_{p-k}=1$, $m_{p-1}=k$ and the subdiagonal of $1$'s in the first $p-k-1$ columns starts in the $k+1$-th row, the aforementioned permutation is the unique one sending $k+i$ to $i$ for every $1\leq i\leq p-k-1$ that corresponds to a non-zero summand of $\mathrm{det}(M')$. Necessarily, any non-zero summand other than $b_0$ includes some factor from the first $p-k-1$ entries of the main diagonal of $M'$, and therefore depends on $\overline{u}$. In other words, $b_0$ is the unique non-zero summand of $\mathrm{det}(M')$ independent of $\overline{u}$. Since the diagonal of $M'$ has at least one zero entry (and hence independent of $\overline{u}$), we deduce that there is a polynomial $P\in\mathcal{O}_K[X]$ of degree at most $p-2$ and independent of $u$ such that $\mathrm{det}(M(\alpha))\equiv uP(u)\,(\mathrm{mod}\,\mathfrak{p}_K)$ and $\overline{P(0)}\neq0$. Since $P$ is of degree at most $p-2$, we can choose $u$ such that $\overline{P(u)}\neq0$, and therefore $\mathrm{det}(M(\alpha))\not\equiv0\,(\mathrm{mod}\,\mathfrak{p}_K)$.
\end{proof}

If we choose $u\in\mathcal{O}_K^*$ such that $\mathrm{det}(M(\alpha))\not\equiv0\,(\mathrm{mod}\,\mathfrak{p}_K)$, we find that $\alpha$ is indeed an $\mathfrak{A}_{L/K}$-generator of $\mathfrak{A}_{\theta}$ for $\theta=\pi_L^a$, proving Proposition \ref{statementsuff}.

\begin{example}\normalfont Let us consider a degree $7$ extension $L/K$ of $7$-adic fields with dihedral normal closure $\widetilde{L}$ such that $t=e=1$ (note that such an extension exists by Proposition \ref{progivenram}). Since $K/\mathbb{Q}_7$ is unramified, we can take $\pi_K=7$ as uniformizer. We calculate $\ell=a=4$, which does not divide $p-1=6$, and since since $t>\frac{2pe}{p-1}-2$, we must use Theorem \ref{maintheorem} (4). We have $$\frac{\ell}{p}=[0;1,1,3],$$ which has $n=3$ and therefore $\mathcal{O}_L$ is $\mathfrak{A}_{L/K}$-free. If we want to find a generator, we must go over the proof of Proposition \ref{prosufcondfreeness}. The values of $\nu_i$ and $n_i$ are shown in the following table: \begin{center}
\begin{tabular}{|c|c|c|c|c|c|c|} \hline
    i & 1 & 2 & 3 & 4 & 5 & 6 \\ \hline
    $\nu_i$ & 1 & 1 & 2 & 2 & 3 & 4 \\ \hline
    $n_i$ & 0 & 1 & 1 & 2 & 3 & 4 \\ \hline
\end{tabular}
\end{center} 

Since $q_2=2$, we set $k=1$. Let $u\in\mathcal{O}_K^*$ and take $\alpha=u+\frac{w}{7}$, so that $M(\alpha)=uM(1)+M(\frac{w}{7})$. Since $\nu_1\neq n_1$ and $\nu_3\neq n_3$, $\overline{M(1)}$ has zeros in positions $(1,1)$ and $(3,3)$, while the other entries of its diagonal are $1$. As for the matrix $\overline{M(\frac{w}{7})}$, the only $0\leq i\leq 6$ for which $k+i>p-1=6$ is $i=6$. Thus, for this value of $i$ we need to use Lemma \ref{lemaentries}. Set $m=k+i-(p-1)=1$. Since $\nu_1=1$, $\nu_3=2$ and $\nu_5=3$, we have that $d_1=4$ and the entries of $\overline{M(\frac{w}{7})}$ in positions $(1,6)$, $(3,6)$ and $(5,6)$ are non-zero. For the remaining values of $i$ (i.e. $0\leq i\leq 5$), we use Lemma \ref{lemaentrieseasy} to identify a subdiagonal of $1$ just below the main diagonal. Then $$\overline{M(\alpha)}=\begin{pmatrix}
\overline{u} & 0 & 0 & 0 & 0 & 0 & 0 \\
1 & 0 & 0 & 0 & 0 & 0 & \overline{\mu}_{1,6}^{(1)} \\
0 & 1 & \overline{u} & 0 & 0 & 0 & 0 \\
0 & 0 & 1 & 0 & 0 & 0 & \overline{\mu}_{3,6}^{(1)} \\
0 & 0 & 0 & 1 & \overline{u} & 0 & 0 \\
0 & 0 & 0 & 0 & 1 & \overline{u} & \overline{\mu}_{5,6}^{(1)} \\
0 & 0 & 0 & 0 & 0 & 1 & \overline{u}
\end{pmatrix}.$$ In this case we find that $\mathrm{det}(\overline{M(\alpha)})=\overline{u}\overline{P(u)}$, where $P(x)=-\mu_{1,6}^{(1)}$. Thus, in this case any unit $u\in\mathcal{O}_K^*$ does the job. For instance, for $u=1$ we see that $\alpha=1+\frac{w}{7}$ is a generator of $\mathcal{O}_L$ as $\mathfrak{A}_{L/K}$-module.
\end{example}

\subsubsection{Necessity}

Our considerations will culminate in the proof of the following.

\begin{pro}\label{statementnecess} If $n\geq 5$, then $\mathfrak{A}_{\theta}$ is not $\mathfrak{A}_{L/K}$-principal for $\theta=\pi_L^a$.
\end{pro}

For the remainder of this section, we assume that $n\geq 5$. We have to prove that given any element $\alpha\in\mathcal{O}_L$, $\mathrm{det}(M(\alpha))\equiv0\,(\mathrm{mod}\,\mathfrak{p}_K)$. We will proceed as in the proof for the corresponding result in the cyclic case \cite[Proposition 3]{bertrandiasbertrandiasferton}, \cite[Chapitre II, Proposition 7]{fertonthesis}.

As usual, we number the columns of $M(\alpha)$ from $0$ to $p-1$. Let us call $n=2s+2$ if $n$ is even and $n=2s+1$ if $n$ is odd. We consider the columns of $M(\alpha)$ with index $i$ such that $h=p-i$ takes one of the following values \begin{equation}
    \begin{split}\label{valuesh}
        &h=2q_{2s-2}, \\
        &h=q_{2s-2}+rq_{2s-1}+q_{2s},\quad 0\leq r\leq a_{2s}.
    \end{split}
\end{equation}

\begin{lema}\label{lemanonfreeless} Assume that $n\geq5$. Let $h$ be as in \eqref{valuesh} and let $1\leq k\leq h-1$. If $\overline{\mu}_{j,i}^{(k)}\neq0$, then $j=k+i$ and $$\begin{cases}
k+1=q_{2s-2} & \hbox{if }h=2q_{2s-2}, \\
k+1=q_{2s-2,r'} & \hbox{if }h=q_{2s-2,r}+q_{2s},\hbox{ where }r\leq r'\leq a_{2s}.
\end{cases}$$
\end{lema}
\begin{proof}
It is straightforward to check that the values of $h$ in \eqref{valuesh} do not match with any of the elements in the description of $E$ in Lemma \ref{lemaparamE}, so $h\notin E$ for those values. Suppose that $\overline{\mu}_{j,i}^{(k)}\neq0$. By Proposition \ref{proentrieseasy} (2), we have that $\widehat{(k+1)\frac{a}{p}}<\widehat{h\frac{a}{p}}$ (in particular $k+1<h$). 

\begin{itemize}
    \item Assume that $h$ is as in the first line of \eqref{valuesh}, i.e. $h=2q_{2s-2}$. We will prove that $k+1=q_{2s-2}$. Note that $k+1<h=2q_{2s-2}<q_{2s}$.
    
    First, we prove that $\widehat{(k+1)\frac{a}{p}}>\widehat{q_{2s}\frac{a}{p}}$. If $n=2s+1$, then Proposition \ref{propcontfrac} (4) gives that $\widehat{q_{2s}\frac{a}{p}}=\frac{1}{p}$, and since $k+1<q_{2s}$, the inequality is strict. For $n=2s+2$, we need to discard the situation that $M_{k+1}$ is between $M_{q_{2s+2}}$ and $M_{q_{2s}}$ (see Note \ref{circleC}). In that case, since $k+1<q_{2s+2}$, we would have by Corollary \ref{corodensity} and our assumption that $k+1=q_{2s,r}$ with $r>0$, so $k+1>q_{2s}$, a contradiction.
    
    Suppose that $\widehat{(k+1)\frac{a}{p}}\leq\widehat{q_{2s-2}\frac{a}{p}}$. Then $M_{k+1}$ is between $M_{q_{2s}}$ and $M_{q_{2s-2}}$. Since in addition $k+1<q_{2s}$, again Corollary \ref{corodensity} gives that $k+1=q_{2s-2,r}$ with $0\leq r<a_{2s}$. Now, since $k+1<h=2q_{2s-2}$, necessarily $r=0$ and $k+1=q_{2s-2}$.
    
    Otherwise, suppose that $\widehat{q_{2s-2}\frac{a}{p}}\leq\widehat{(k+1)\frac{a}{p}}$. Let us assume that $k+1\neq q_{2s-2}$. Then $1\leq|k+1-q_{2s-2}|<q_{2s-2}$, so $||(k+1-q_{2s-2})\frac{a}{p}||>\widehat{q_{2s-2}\frac{a}{p}}$. But $\widehat{k+1\frac{a}{p}}<\widehat{h\frac{a}{p}}\leq2\widehat{q_{2s-2}\frac{a}{p}}$, and since $\widehat{q_{2s-2}\frac{a}{p}}\leq\widehat{(k+1)\frac{a}{p}}$, we obtain that $||(k+1-q_{2s-2})\frac{a}{p}||\leq\widehat{q_{2s-2}\frac{a}{p}}$, which is a contradiction. Necessarily $k+1=q_{2s-2}$, as we wanted.

    \item Assume that $h$ is as in the second line of \eqref{valuesh}. Call $\lambda=q_{2s-2,r}$ with $0\leq r\leq a_{2s}$, so $h=q_{2s}+\lambda$. We will prove that $\widehat{q_{2s}\frac{a}{p}}<\widehat{(k+1)\frac{a}{p}}<\widehat{\lambda\frac{a}{p}}$ and $k+1\leq q_{2s}$. This is enough for our purposes since then Remark \ref{densityrefinement} gives that $k+1$ is as stated. 
    
    As in the previous case, we prove that $\widehat{(k+1)\frac{a}{p}}>\widehat{q_{2s}\frac{a}{p}}$. In this one, for the case $n=2s+2$ we have that $k+1<q_{2s}+q_{2s-2,r}<q_{2s}+q_{2s+1}\leq q_{2s+2}$ and the same argument applies.
    
    Next, let us prove that $\widehat{(k+1)\frac{a}{p}}\leq\widehat{\lambda\frac{a}{p}}$. Assume that $\widehat{(k+1)\frac{a}{p}}>\widehat{\lambda\frac{a}{p}}$. Since $1\leq k+1,\lambda\leq q_{2s}$, $|k+1-\lambda|<q_{2s}$, so $\Big|\Big|(k+1-\lambda)\frac{a}{p}\Big|\Big|>\widehat{q_{2s}\frac{a}{p}}$. But at the same time $\widehat{(k+1)\frac{a}{p}}<\widehat{h\frac{a}{p}}\leq\widehat{q_{2s}\frac{a}{p}}+\widehat{\lambda\frac{a}{p}}$, so $\Big|\Big|(k+1-\lambda)\frac{a}{p}\Big|\Big|<\widehat{q_{2s}\frac{a}{p}}$, which is a contradiction.
    
    Let us check that $k+1\leq q_{2s}$. Arguing as before, we prove that $\widehat{(k+1)\frac{a}{p}}<\widehat{q_{2s}\frac{a}{p}}+\widehat{\lambda\frac{a}{p}}$. If $k+1>q_{2s}$, we deduce $\widehat{(k+1-q_{2s})\frac{a}{p}}<\widehat{\lambda\frac{a}{p}}$. But since $\lambda\in E$ and $1\leq k+1-q_{2s}<\lambda$ by assumption, we deduce that $\widehat{(k+1-q_{2s})\frac{a}{p}}>\widehat{\lambda\frac{a}{p}}$, which is a contradiction.
\end{itemize}
\end{proof}

\begin{lema} Suppose that $n\geq 5$. Let $h$ be as in \eqref{valuesh} and let $0\leq k\leq p-1$. If $k+1>h$ or $k=0$, then $\mu_{j,i}^{(k)}=0$ for every $0\leq j\leq p-1$.
\end{lema}
\begin{proof}
First, we check that if $h$ is as in \eqref{valuesh}, then $\widehat{h\frac{a}{p}}<\frac{a}{p}$. For this, we observe that the distance within $C$ of two denominators $q_{2i-2}$, $q_{2i}$ is just $d(q_{2i-2},q_{2i})=a_{2i}||q_{2i-1}\frac{a}{p}||$ (see the proof of Lemma \ref{lemadensity}), so recursively $$\frac{a}{p}=\widehat{q_{2i}\frac{a}{p}}+\sum_{j=1}^{i}a_{2j}\Big|\Big|q_{2j-1}\frac{a}{p}\Big|\Big|$$ for every $i$. Moreover, $n\geq 5$ implies that $s\geq2$. Now, for $h=2q_{2s-2}$, we have that \begin{equation*}
    \begin{split}
        \frac{a}{p}=\widehat{q_{2s-2}\frac{a}{p}}+\sum_{j=1}^{s-1}a_{2j}\Big|\Big|q_{2j-1}\frac{a}{p}\Big|\Big|&\geq\widehat{q_{2s-2}\frac{a}{p}}+\Big|\Big|q_{2s-3}\frac{a}{p}\Big|\Big|\\&>2\widehat{q_{2s-2}\frac{a}{p}}=\widehat{h\frac{a}{p}}.
    \end{split}
\end{equation*} For the second inequality, we have used that $||q_{2s-3}\frac{a}{p}||>\widehat{q_{2s-2}\frac{a}{p}}$, which follows from Proposition \ref{propcontfrac} (1), (3), while the last equality is a consequence of $h=2q_{2s-2}<p$. On the other hand, if $h=q_{2s-2,r}+q_{2s}$ is as in the second line of \eqref{valuesh}, then \begin{equation*}
    \begin{split}
        \frac{a}{p}=\widehat{q_{2s}\frac{a}{p}}+\sum_{j=1}^{s}a_{2j}\Big|\Big|q_{2j-1}\frac{a}{p}\Big|\Big|&\geq\widehat{q_{2s}\frac{a}{p}}+r\Big|\Big|q_{2s-1}\frac{a}{p}\Big|\Big|+\Big|\Big|q_{2s-3}\frac{a}{p}\Big|\Big|\\&>\widehat{q_{2s}\frac{a}{p}}+\widehat{q_{2s-2,r}\frac{a}{p}}\geq\widehat{h\frac{a}{p}}.
    \end{split}
\end{equation*} For the second inequality, we have used again that $||q_{2s-3}\frac{a}{p}||>\widehat{q_{2s-2}\frac{a}{p}}$, and that $\widehat{q_{2s-2,r}\frac{a}{p}}=\widehat{q_{2s-2}\frac{a}{p}}+r||q_{2s-1}\frac{a}{p}||$. Thus, we obtain that $\widehat{h\frac{a}{p}}<\frac{a}{p}$ as claimed.

For $k=0$, let us prove that $\overline{\mu}_{i,i}^{(0)}=0$. We know that $\nu_i=ia_0+\Big\lfloor(i+1)\frac{a}{p}\Big\rfloor$. Since $h\notin E$, Lemma \ref{lemanuiE} gives that $n_i=ia_0+\Big\lfloor i\frac{a}{p}\Big\rfloor$. Now, the condition $\widehat{h\frac{a}{p}}<\frac{a}{p}$ is the same as $\widehat{i\frac{a}{p}}>1-\frac{a}{p}$. Then, we see that $\Big\lfloor(i+1)\frac{a}{p}\Big\rfloor=\Big\lfloor i\frac{a}{p}\Big\rfloor+1$, proving that $\nu_i=n_i+1$, and hence $\overline{\mu}_{i,i}^{(0)}=0$, as we wanted.

Now, assume that $k+1>h$ and call $m=k+1-h$. From Proposition \ref{proentries} (3) we see that the condition $\widehat{h\frac{a}{p}}<\frac{a}{p}$ gives directly that $\overline{\mu}_{j,i}^{(k)}=0$ for every $m\leq j\leq m+d_m$. Since $i<p-1$ for $h=p-i$ as in \eqref{valuesh}, we have that $m<p-1$. Then Proposition \ref{proentries} (2) allows us to conclude that $\overline{\mu}_{j,i}^{(k)}=0$ for all $j$.
\end{proof}

We use the data obtained in these two results to prove that $\mathrm{det}(\overline{M(\alpha)})$ vanishes.

\begin{coro} Let us assume that $n\geq5$ and take $\theta=\pi_L^a$. Given $\alpha\in\mathfrak{A}_{\theta}$, $\mathrm{det}(M(\alpha))\equiv0\,(\mathrm{mod}\,\mathfrak{p}_K)$.
\end{coro}
\begin{proof}
By the two previous lemmas, the matrix $\overline{M(\alpha)}$ has $a_{2s}+2$ columns (the ones corresponding to \eqref{valuesh}) with at most $a_{2s}+1$ non-zero elements, which are the ones in Lemma \ref{lemanonfreeless}. These elements are in the rows of index $j=k+i$, where $k$ is as stated there. If $h$ is as in the first line of \eqref{valuesh}, $p-j-1=q_{2s-2}$. Otherwise, $$p-j-1=q_{2s-2}+(a_{2s}-(r'-r))q_{2s-1},\quad0\leq r\leq r'\leq a_{2s}.$$ In total, these correspond to $a_{2s}+1$ different values of $j$. Then the $a_{2s}+1$ possibly non-zero elements in each of these columns are allocated in the same rows. Now, since each summand of $\mathrm{det}(\overline{M(\alpha)})$ includes exactly one entry of each column as a factor, it is necessarily zero.
\end{proof}

This result immediately implies Proposition \ref{statementnecess}, as it gives that for $\theta=\pi_L^a$ no element $\alpha\in\mathfrak{A}_{\theta}$ is a generator of $\mathfrak{A}_{\theta}$ as $\mathfrak{A}_{L/K}$-ideal, and therefore $\mathfrak{A}_{\theta}$ is not $\mathfrak{A}_{L/K}$-principal.

\begin{example}\normalfont\label{examplenonfree} Let $K$ be a $13$-adic field with absolute ramification index $e=2$, and let $L/K$ be a degree $13$ extension of $13$-adic fields with dihedral normal closure $\widetilde{L}$ having ramification jump $t=3$. We have that $\frac{2\cdot13\cdot e}{12}-2=\frac{7}{3}<t<\frac{2\cdot13\cdot e}{12}=\frac{13}{3}$ and $\ell=a=8$, so we may apply Theorem \ref{maintheorem} (4). Since $\frac{8}{13}=[0;1,1,1,1,2]$, we see that $n=5$ and $\mathcal{O}_L$ is not $\mathfrak{A}_{L/K}$-free.
\end{example}

\subsubsection*{Proof of Theorem \ref{maintheorem} (4)}

We finish the proof of Theorem \ref{maintheorem} (4) as follows. By Proposition \ref{statementsuff} if $n\leq 4$ then $\mathfrak{A}_{\theta}$ is $\mathfrak{A}_{L/K}$-principal. Now we use the paragraph following Proposition \ref{prohnormalfreeness}, concretely the same proof of (4) implies in (1) in that result gives that $\mathcal{O}_L$ is $\mathfrak{A}_{L/K}$-free. Conversely, if $n\geq 5$, from Proposition \ref{statementnecess} we deduce that $\mathfrak{A}_{\theta}$ is not $\mathfrak{A}_{L/K}$-principal, and Proposition \ref{prohnormalfreeness} directly gives that $\mathcal{O}_L$ is not $\mathfrak{A}_H$-free.

\section{Consequences}\label{sectconseq}

We examine some interesting consequences and particular cases that can be derived from Theorem \ref{maintheorem}.

\begin{pro} Let $p$ be an odd prime and let $L/K$ be a degree $p$ extension of $p$-adic fields with dihedral normal closure $\widetilde{L}$. If $K/\mathbb{Q}_p$ is unramified (in particular, if $K=\mathbb{Q}_p$), then $\mathcal{O}_L$ is $\mathfrak{A}_{L/K}$-free.
\end{pro}
\begin{proof}
The hypothesis means that $e=1$. If $\widetilde{L}/K$ is not totally ramified, this implies that $1\leq t\leq\frac{p}{p-1}$, giving that $t=1$. Moreover $1>\frac{p}{p-1}-1$, so the $\mathfrak{A}_{L/K}$-freeness of $\mathcal{O}_L$ is obtained by Corollary \ref{coropnonram} (3).

Now, assume that $\widetilde{L}/K$ is totally ramified. Then, the ramification jump $t$ of $\widetilde{L}/K$ is upper bounded by $3$ if $p=3$ and by $2$ otherwise. Moreover, $t$ is the ramification jump of the totally ramified extension $\widetilde{L}/M$, where $M$ is the unique quadratic intermediate field of $\widetilde{L}/K$, so as discussed in Section \ref{sectarithm}, $t$ must be odd. Hence, we have that $t=1$ unless $p=3$ and $t=3$.

Assume that $t=1$. Then we have $\ell=a=\frac{p+1}{2}$. If $p=3$, we have that $a=2$ divides $p-1=2$, and we are done by Theorem \ref{maintheorem} (3). For $p>3$, since $\frac{2p}{p-1}-2<1$, we may apply Theorem \ref{maintheorem} (4). Now, one calculates $$\frac{\ell}{p}=\Big[0;1,1,\frac{p-1}{2}\Big],$$ which has length $n=3$, and then $\mathcal{O}_L$ is $\mathfrak{A}_{L/K}$-free.

Finally, suppose that $p=3$ and $t=3$. Then we have that $\widetilde{L}/K$ is maximally ramified, and we apply Theorem \ref{maintheorem} (1).
\end{proof}

Recall that by Proposition \ref{progivenram} we can assure the existence of a dihedral degree $2p$ extension of $p$-adic fields (and consequently of a degree $p$ extension with dihedral normal closure) with given ramification. We will use this result and Theorem \ref{maintheorem} to construct examples of extensions with behaviour which was not already known.

The first of the examples is related with weakly ramified extensions, i.e. those Galois extensions of $p$-adic fields for which the second ramification group $G_2$ is trivial. In \cite[Theorem 1.2]{johnston}, Johnston proved that the ring of integers of such an extension is always free over the associated order in the classical Galois structure. The following result will give rise to a weakly ramified extension whose ring of integers is not free over the associated order of some non-classical Hopf-Galois structure.

\begin{pro}\label{proweakly} There exists a degree $p$ extension of $p$-adic fields $L/K$ with dihedral normal closure $\widetilde{L}$ such that $\widetilde{L}/K$ is weakly ramified and $e>1$. For such an extension, $\mathcal{O}_L$ is $\mathfrak{A}_{L/K}$-free if and only if $p=3$. 
\end{pro}
\begin{proof}
We first prove the existence using Proposition \ref{progivenram}. Let $K$ be an absolutely ramified $p$-adic field, i.e. with ramification index $e>1$ over $\mathbb{Q}_p$. Since $t=1$ is odd and coprime with $p$, by Proposition \ref{progivenram} there is a degree $p$ extension $L/K$ with dihedral normal closure $\widetilde{L}$ whose ramification jump is $t$.

Let us prove the equivalence. Let $L/K$ be an extension as in the statement, so that $t=1$ and $e>1$. Then $$\frac{2pe}{p-1}-2\geq\frac{4p}{p-1}-2>1,$$ and $a=\frac{p+1}{2}$, which divides $p-1$ if and only if $p=3$. Thus, the statement follows from Theorem \ref{maintheorem} (3).
\end{proof}

If $L$ is as in Proposition \ref{proweakly}, the Galois group of $\widetilde{L}/K$ is a semidirect product, and hence we can consider induced Hopf-Galois structures on $\widetilde{L}/K$ (see \cite{cresporiovela} for the definition and the main properties of induced Hopf-Galois structures).

\begin{example} If $p>3$, there is some dihedral degree $2p$ extension $E/K$ of $p$-adic fields which is weakly ramified and such that $\mathcal{O}_E$ is not $\mathfrak{A}_{H'}$-free for some induced Hopf-Galois structure $H'$ on $E/K$.
\end{example}
\begin{proof}
Let $L/K$ be as in Proposition \ref{proweakly}, and let $E=\widetilde{L}$ be the normal closure of $L/K$, so that $E/K$ is weakly ramified. Since $G\coloneqq\mathrm{Gal}(E/K)\cong D_p$, $\mathrm{Gal}(E/L)$ has a normal complement inside $G$. Then, the unique Hopf-Galois structures $H$ and $\overline{H}$ on $L/K$ and $E/L$ respectively induce a Hopf-Galois structure $H'$ on $E/K$. Now, we can apply \cite[Proposition 5.1]{truman2018}. Since $\mathcal{O}_L$ is not $\mathfrak{A}_{L/K}$-free by Proposition \ref{proweakly}, we conclude that $\mathcal{O}_E$ is not $\mathfrak{A}_{H'}$-free.
\end{proof}

Another consequence of Proposition \ref{proweakly} is that there are infinitely many degree $p$ extensions $L/K$ with dihedral normal closure such that $\mathcal{O}_L$ is not $\mathfrak{A}_{L/K}$-free, where $p$ runs through the odd prime numbers greater than $3$.

The second example is related with the study of the freeness after tensoring by $\mathcal{O}_M$. In Proposition \ref{proaftertensoring}, we proved that under the restriction that $L/K$ and $M/K$ are arithmetically disjoint, the $\mathfrak{A}_{L/K}$-freeness of $\mathcal{O}_L$ is equivalent to the $\mathfrak{A}_{\widetilde{L}/M}$-freeness of $\mathcal{O}_{\widetilde{L}}$. We show that this statement does not hold if we remove the arithmetic disjointness.

\begin{example} There is some degree $p$ extension $L/K$ of $p$-adic fields with dihedral normal closure $\widetilde{L}$ such that $\mathcal{O}_{\widetilde{L}}$ is $\mathfrak{A}_{\widetilde{L}/M}$-free but $\mathcal{O}_L$ is not $\mathfrak{A}_{L/K}$-free.
\end{example}
\begin{proof}
Choose the extension $L/K$ of $13$-adic fields in Example \ref{examplenonfree}. We already know that $\mathcal{O}_L$ is not $\mathfrak{A}_{L/K}$-free (in the only Hopf-Galois structure on $L/K$). Let $M/K$ be the only quadratic subextension of $\widetilde{L}/K$. Since $\widetilde{L}/K$ is also of degree $13$, it admits a unique Hopf-Galois structure as well. Now, $t=3$ is the ramification jump of $\widetilde{L}/M$ and its remainder mod $13$ is $a=3$, which divides $p-1=12$, so from Theorem \ref{maintheorem} (3) it follows that $\mathcal{O}_{\widetilde{L}}$ is $\mathfrak{A}_{\widetilde{L}/M}$-free.
\end{proof}

We finish by giving a bunch of extensions $L/K$ such that $\mathcal{O}_L$ is not $\mathfrak{A}_{L/K}$-free. If we write the continued fraction expansion in Example \ref{examplenonfree} and sum the successive denominators, we see that each one is the sum of the previous two, starting with $1$ and $2$, and consequently the fraction is a quotient of two consecutive terms of the Fibonacci sequence. This inspires the following general construction.

\begin{example}\normalfont Let $\{x_k\}_{k=1}^{\infty}$ be a sequence of positive integers such that $x_k=x_{k-1}+x_{k-2}$ for every $k\geq3$ and $\gcd(x_1,x_2)=1$. Fix $k\geq6$ such that $p=x_k$ is an odd prime number. For instance:
\begin{itemize}
    \item If $x_1=1$ and $x_2=2$, $\{x_k\}_{k=1}^{\infty}$ is the Fibonacci sequence and $p$ may be chosen as any Fibonacci prime greater than $5$, such as $x_6=13$ or $x_{10}=89$.
    \item If $x_1=2$ and $x_2=1$, $\{x_k\}_{k=1}^{\infty}$ is the sequence of Lucas numbers and $p$ may be chosen as any Lucas prime greater than $11$, such as $x_7=29$ or $x_8=47$.
\end{itemize} Using Proposition \ref{progivenram}, choose a degree $p$ extension $L/K$ of $p$-adic fields with dihedral normal closure such that $t=x_{k-3}$ and $e\in\mathbb{Z}_{>1}$. Then $$\ell=\frac{p+t}{2}=\frac{x_k+x_{k-3}}{2}=x_{k-1},$$ and consequently also $a=x_{k-1}$. We claim that $\mathcal{O}_L$ is not $\mathfrak{A}_{L/K}$-free. Assume that $t<\frac{2pe}{p-1}-2$. Since $x_k=x_{k-1}+x_{k-2}\geq2x_{k-2}+1$, we see that $x_{k-2}\leq\frac{x_k-1}{2}$, so $x_{k-1}\geq\frac{x_k+1}{2}$. Since $k\geq 5$, $p>3$ and then $a$ does not divide $p-1$, so $\mathcal{O}_L$ is not $\mathfrak{A}_{L/K}$-free. On the other hand, if $t\geq\frac{2pe}{p-1}-2$, then $\frac{x_{k-1}}{x_k}=\frac{1}{1+\frac{x_{k-2}}{x_{k-1}}}$, and recursively we see that $$\frac{x_{k-1}}{x_k}=\cfrac{1}{1+\cfrac{1}{\ddots+\cfrac{1}{1+\cfrac{x_1}{x_2}}}}.$$ Since $k\geq 6$, the length of the continued fraction expansion of $\frac{\ell}{p}$ is $n\geq5$, so from Theorem \ref{maintheorem} (4) we conclude that $\mathcal{O}_L$ is not $\mathfrak{A}_{L/K}$-free.
\end{example}

\section*{Acknowledgements} 

I want to thank Nigel Byott and Griffith Elder for enriching discussions and helpful suggestions on the topic of study on this paper, and Paul Truman for pointing out the idea to derive the counterexample from Proposition \ref{proweakly}. I want to thank the referee as well, for their insightful comments and their suggestions to improve the quality and the readability of this work. I am also grateful with the library of Charles University, for requesting a copy of \cite{fertonthesis} in my name, and with the library of the University of Grenoble, for letting me borrow the original version. This work was supported by Czech Science Foundation, grant 21-00420M, and by Charles University Research Centre program UNCE/SCI/022.

\printbibliography[heading=bibintoc]

\Addresses

\end{document}